\documentclass[oneside,english]{amsart}
\usepackage[T1]{fontenc}
\usepackage[latin9]{inputenc}
\usepackage{geometry}
\usepackage{babel}
\usepackage{mathrsfs}
\usepackage{amsbsy}
\usepackage{amstext}
\usepackage{amsthm}
\usepackage{amssymb}
\PassOptionsToPackage{normalem}{ulem}
\usepackage{ulem}
\usepackage[unicode=true,pdfusetitle,
 bookmarks=true,bookmarksnumbered=false,bookmarksopen=false,
 breaklinks=false,pdfborder={0 0 1},backref=false,colorlinks=false]
 {hyperref}

\makeatletter
\theoremstyle{plain}
\newtheorem{thm}{\protect\theoremname}
\theoremstyle{plain}
\newtheorem{question}[thm]{\protect\questionname}
\theoremstyle{remark}
\newtheorem{rem}[thm]{\protect\remarkname}
\theoremstyle{definition}
\newtheorem{defn}[thm]{\protect\definitionname}
\theoremstyle{plain}
\newtheorem{prop}[thm]{\protect\propositionname}
\theoremstyle{plain}
\newtheorem{lem}[thm]{\protect\lemmaname}
\theoremstyle{plain}
\newtheorem{assumption}[thm]{\protect\assumptionname}
\theoremstyle{plain}
\newtheorem{cor}[thm]{\protect\corollaryname}
\theoremstyle{definition}
\newtheorem{example}[thm]{\protect\examplename}
\theoremstyle{plain}
\newtheorem{conjecture}[thm]{\protect\conjecturename}

\@ifundefined{date}{}{\date{}}
\usepackage{tikz}
\usepackage{pgfplots}
\usetikzlibrary{matrix,arrows,decorations.pathmorphing}
\usepackage{verbatim}
\usepackage{pgf}
\usepackage{color}
\definecolor{BLACK}{RGB}{0, 0, 0}
\definecolor{green}{RGB}{0, 0, 0}
\definecolor{cyan}{RGB}{0,100, 0}
\definecolor{yellow}{RGB}{0,100,0}
\definecolor{red}{RGB}{0,100,0}
\definecolor{BLUE}{RGB}{0,100,0}

\usetikzlibrary[patterns]
\definecolor{blue}{RGB}{0,100,0}
\colorlet{linkequation}{BLACK}

\usepackage{scalerel,stackengine}
\stackMath
\newcommand\widecheck[1]{%
\savestack{\tmpbox}{\stretchto{%
  \scaleto{%
    \scalerel*[\widthof{\ensuremath{#1}}]{\kern-.6pt\bigwedge\kern-.6pt}%
    {\rule[-\textheight/2]{1ex}{\textheight}}
  }{\textheight}%
}{0.5ex}}%
\stackon[1pt]{#1}{\scalebox{-1}{\tmpbox}}%
}
\DeclareRobustCommand{\oddevenhead} {%
\scshape
\ifodd\value{page}
A Generalization of the Tristram-Levine Knot 
Signatures as a Singular Furuta-Ohta Invariant for Tori %
\else
Mariano Echeverria%
\fi
}

\makeatother

\providecommand{\assumptionname}{Assumption}
\providecommand{\conjecturename}{Conjecture}
\providecommand{\corollaryname}{Corollary}
\providecommand{\definitionname}{Definition}
\providecommand{\examplename}{Example}
\providecommand{\lemmaname}{Lemma}
\providecommand{\propositionname}{Proposition}
\providecommand{\questionname}{Question}
\providecommand{\remarkname}{Remark}
\providecommand{\theoremname}{Theorem}

\begin{document}
\title[\oddevenhead]{ a Generalization of the Tristram-Levine Knot Signatures as a Singular
Furuta-Ohta Invariant for Tori}
\author{Mariano Echeverria}
\begin{abstract}
Given a knot $K$ inside an integer homology sphere $Y$, the Casson-Lin-Herald
invariant can be interpreted as a signed count of conjugacy classes
of irreducible representations of the knot complement into $SU(2)$
which map the meridian of the knot to a fixed conjugacy class. It
has the interesting feature that it determines the Tristram-Levine
signature of the knot associated to the conjugacy class chosen. 

Turning things around, given a 4-manifold $X$ with the integral homology
of $S^{1}\times S^{3}$, and an embedded torus $T$ inside $X$ such
that $H_{1}(T;\mathbb{Z})$ surjects onto $H_{1}(X;\mathbb{Z})$,
we define a signed count of conjugacy classes of irreducible representations
of the torus complement into $SU(2)$ which satisfy an analogous fixed
conjugacy class condition to the one mentioned above for the knot
case. Our count recovers the Casson-Lin-Herald invariant of the knot
in the product case, thus it can be regarded as implicitly defining
a Tristram-Levine signature for tori. 

This count can also be considered as a singular Furuta-Ohta invariant,
and it is a special case of a larger family of Donaldson invariants
which we also define. In particular, when $(X,T)$ is obtained from
a self-concordance of a knot $(Y,K)$ satisfying an admissibility
condition, these Donaldson invariants are related to the Lefschetz
number of an Instanton Floer homology for knots which we construct.
Moreover, from these Floer groups we obtain Frøyshov invariants for
knots which allows us to assign a Frøyshov invariant to an embedded
torus whenever it arises from such a self-concordance. 

\end{abstract}

\maketitle

\section{Introduction}

\subsection{An observation of some analogies: defining $\lambda_{FO}(X,T,\alpha)$}

\ 

The main impetus behind this work stems from the following observation:
given an integer homology sphere $Y$, Casson defined \cite{MR1030042}
an invariant $\lambda_{C}(Y)\in\mathbb{Z}$ which morally can be regarded
as $1/2$ of a signed count of conjugacy classes of irreducible representations
$\pi_{1}(Y)\rightarrow SU(2)$. Moving one dimension up, if $X$ is
a four manifold with the same integral homology as $S^{1}\times S^{3}$,
i.e, $H_{*}(X;\mathbb{Z})\simeq H_{*}(S^{1}\times S^{3};\mathbb{Z})$,
Furuta and Ohta defined \cite{MR1237394} a Casson-type invariant
$\lambda_{FO}(X)\in\mathbb{Z}$ which again can morally be interpreted
as a signed count of conjugacy classes of irreducible representations
$\pi_{1}(X)\rightarrow SU(2)$. Strictly speaking, $X$ must satisfy
additional conditions besides being a homology $S^{1}\times S^{3}$,
but it will always be the case that if one takes $X=S^{1}\times Y$,
then $\lambda_{FO}(X)$ is well defined, and can be made to agree
with $\lambda_{C}(Y)$.

Now, if we consider the case of a knot $K$ inside $Y$, one can follow
the same strategy and try to define a signed count of conjugacy classes
of irreducible representations $\pi_{1}(Y\backslash K)\rightarrow SU(2)$.
As will become clear soon, in this case it is natural to fix the conjugacy
class where the meridian of the knot $\mu_{K}$ is being sent, regarded
as the generator of the reducible representations $H_{1}(Y\backslash K)\simeq\mathbb{Z}[\mu_{K}]\rightarrow SU(2)$.
This conjugacy class is specified by a (holonomy) parameter $\alpha\in(0,1/2)$,
and whenever $\triangle_{K}(e^{-4\pi i\alpha})\neq0$, where $\triangle_{K}$
is the Alexander polynomial of the knot, Herald showed \cite[Theorem 0.1]{MR1456309}
that a Casson-type count of representations can be made, which we
will denote as $\lambda_{CLH}(Y,K,\alpha)\in\mathbb{Z}$. Here the
$L$ in $\lambda_{CLH}$ stands for Lin, who had studied before Herald
\cite{MR1158339} the case where the meridian is sent to a trace zero
matrix, which corresponds to the parameter $\alpha=1/4$ (and using
a symplectic rather than gauge theoretic approach).

The invariant $\lambda_{CLH}$ recovers the Tristram-Levine knot signatures
$\sigma_{K}$ \cite{MR248854,MR0246314}, in the sense that (with
our orientation conventions) $\lambda_{CLH}(Y,K,\alpha)=4\lambda_{C}(Y)+\frac{1}{2}\sigma_{K}(e^{-4\pi i\alpha})$.
This clearly begs the question:
\begin{question}
\label{Question}Given an embedded torus $T$ inside a four manifold
$X$ with the integer homology of $S^{1}\times S^{3}$ and a number
$\alpha\in(0,1/2)$, which conditions must be imposed on $T$ and
$\alpha$ so that one can define an invariant $\lambda_{FO}(X,T,\alpha)\in\mathbb{Z}$
that satisfies the property that for any knot $K\subset Y$ and $\alpha$
such that $\triangle_{K}(e^{-4\pi i\alpha})\neq0$, $\lambda_{FO}(S^{1}\times Y,S^{1}\times K,\alpha)$
coincides with $\lambda_{CLH}(Y,K,\alpha)$ (perhaps up to rescaling)?
\end{question}

Before stating such conditions, we recall quickly why it is possible
to define $\lambda_{CLH}(Y,K,\alpha)$, since the reader may be less
aware of this construction. In fact, as far as we can tell there was
no standard notation for this invariant, which is why we decided to
name it $\lambda_{CLH}(Y,K,\alpha)$. More context and motivation
can be found in an ``annotated'' version of this paper posted on
the author's website \cite{Echeverria[DraftFuruta]}.

As mentioned before, given a knot $K$ inside $Y$, suppose we want
to define a (signed) count of irreducible representations of $\pi_{1}(Y\backslash K)$
into $SU(2)$, analogous to the definition of the Casson invariant
$\lambda_{C}(Y)$ \cite{MR1030042}. For any Casson type count, we
need to understand first the reducible representations, given that
they correspond to singular points in the character variety $\mathcal{R}(Y\backslash K,SU(2))=\hom(\pi_{1}(Y\backslash K),SU(2))/SU(2)$
(or the space of flat $SU(2)$ connections mod gauge depending on
one's preference). 

The reducible representations are determined by $H_{1}(Y\backslash K;\mathbb{Z})\simeq\mathbb{Z}[\mu_{K}]$.
In particular, a reducible representation $\rho$ is completely specified
by $A=\rho(1)\in SU(2)$, and up to conjugacy is determined by its
trace $\text{tr}(A)\in[-2,2]$. Thus the reducible representations
inside $\mathcal{R}(Y\backslash K,SU(2))$ are in bijection with $[-2,2]$.

This means that inside $\mathcal{R}(Y\backslash K,SU(2))$ the reducible
representations are not isolated. In order to isolate them, we can
decompose this space as $\mathcal{R}(Y\backslash K,SU(2))=\bigcup_{\alpha}\mathcal{R}_{\alpha}(Y\backslash K,SU(2))$,
where $\mathcal{R}_{\alpha}(Y\backslash K,SU(2))$ consists of conjugacy
classes of representations which map the meridian $\mu_{K}$ to the
conjugacy class of the matrix
\[
\mu_{K}\rightarrow\left(\begin{array}{cc}
e^{-2\pi i\alpha} & 0\\
0 & e^{2\pi i\alpha}
\end{array}\right)
\]
Since the trace of this matrix is $2\cos(2\pi\alpha)$, it suffices
to take $\alpha\in[0,1/2]$ to exhaust all the possible conjugacy
classes where the meridian can be sent. We will assume that $\alpha\in(0,1/2)$
since the endpoints correspond to representations of $\pi_{1}(Y)$
into $SU(2)$ or $SO(3)$ (so $K$ is not involved in those cases).
Inside each $\mathcal{R}_{\alpha}(Y\backslash K,SU(2))$ there will
be only one reducible representation $\theta_{\alpha}$, so trying
to make a (signed) count $\lambda_{CLH}(Y,K,\alpha)$ of the elements
inside $\mathcal{R}_{\alpha}(Y\backslash K,SU(2))$ is very similar
to the problem of defining $\lambda_{C}(Y)$ as a (signed) count of
elements inside $\mathcal{R}(Y,SU(2))$, which also contains only
one reducible $\theta$, the trivial representation of $\pi_{1}(Y)$
into $SU(2)$.

The only caveat with this analogy is that contrary to the case of
$Y$, where $\theta$ is automatically isolated from the irreducible
representations, it may be the case that $\theta_{\alpha}$ is not
isolated from the irreducible representations inside $\mathcal{R}_{\alpha}(Y\backslash K,SU(2))$,
thus making it impossible to define $\lambda_{CLH}(Y,K,\alpha)$.
In fact, as we will review in section \ref{sec:Floer-Novikov-Homology-and},
whenever $\triangle_{K}(e^{-4\pi i\alpha})\neq0$, where $\triangle_{K}$
is the Alexander polynomial of the knot $K$, we can define $\lambda_{CLH}(Y,K,\alpha)$,
i.e, for such $\alpha$ the reducible $\theta_{\alpha}$ is isolated.

If we think of the representation $\theta_{\alpha}$ as corresponding
to a flat $SU(2)$ connection $A_{\alpha}$ on the trivial bundle
$E=(Y\backslash K)\times\mathbb{C}^{2}\rightarrow Y\backslash K$,
which is the point of view we will take, the condition that $\triangle_{K}(e^{-4\pi i\alpha})\neq0$
is equivalent to the vanishing of the cohomology group with local
coefficients $H^{1}(Y\backslash K;L_{\alpha}^{\otimes2})$, where
$A_{\alpha}$ determines a complex line bundle $L_{\alpha}\rightarrow Y\backslash K$
and a splitting $E=L_{\alpha}\oplus L_{\alpha}^{-1}$.

Going back to the case of an embedded torus $T$ inside $X$, we will
explain in Section \ref{sec:Singular-Orbifold-Furuta-Ohta} the analogue
of the spaces $\mathcal{R}_{\alpha}(Y\backslash K,SU(2))$, which
we denote $\mathcal{R}_{\alpha}(X\backslash T,SU(2))$. The point
is that we need to understand again the reducible representations
of $\pi_{1}(X\backslash T)\rightarrow SU(2)$, which are completely
determined by $H_{1}(X\backslash T;\mathbb{Z})$. The natural condition
for our problem is to assume that $H_{1}(T;\mathbb{Z})$ surjects
onto $H_{1}(X;\mathbb{Z})$, so that homologically $T$ behaves like
$S^{1}$ times a knot. This condition also tell us that $H_{1}(X\backslash T;\mathbb{Z})\simeq\mathbb{Z}\oplus\mathbb{Z}$,
although for fixed $\alpha$ only one of these $\mathbb{Z}$ factors
will be of importance. This means that inside $\mathcal{R}_{\alpha}(X\backslash T,SU(2))$
the reducible representations should form a one dimensional family,
which is exactly what happens when one is trying to define $\lambda_{FO}(X)$
using $\mathcal{R}(X,SU(2))$. Given these preliminaries we can finally
state our answer to Question \ref{Question}.

\begin{thm}
\label{definition of the singular Furuta Ohta invariant} Suppose
that $X$ is a closed oriented four manifold with the integral homology
of $S^{1}\times S^{3}$. Let $T$ denote an embedded and oriented
torus such that $H_{1}(T;\mathbb{Z})\twoheadrightarrow H_{1}(X;\mathbb{Z})$
is a surjection. Consider a rational value $\alpha\in\mathbb{Q}\cap(0,1/2)$
such that for all reducible representations $\rho:\pi_{1}(X\backslash T)\rightarrow SU(2)$
satisfying the holonomy condition \ref{asymptotic holonomy} determined
by $\alpha$, i.e, elements of $\mathcal{R}_{\alpha}(X\backslash T,SU(2))$,
we have that $H^{1}(X\backslash T;L_{\rho}^{\otimes2})$ vanishes.
Here $E=L_{\rho}\oplus L_{\rho}^{-1}$ is the decomposition induced
by the flat connection $A_{\rho}$ associated to $\rho$, and $E$
is the trivial $SU(2)$ bundle over $X\backslash T$.

Then we can define a degree-zero Donaldson invariant, denoted the
\textbf{singular Furuta-Ohta invariant $\lambda_{FO}(X,T,\alpha)$},
which takes even values, i.e, $\lambda_{FO}(X,T,\alpha)\in2\mathbb{Z}$.

It has the property that for $X=S^{1}\times Y$, $T=S^{1}\times K$,
for an oriented knot $K$ inside an integer homology sphere $Y$,
and a value of $\alpha\in\mathbb{Q}\cap(0,1/2)$ such that $\triangle_{K}(e^{-4\pi i\alpha})\neq0$,
then $\lambda_{FO}(S^{1}\times Y,S^{1}\times K,\alpha)$ can be defined,
and  $\lambda_{FO}(S^{1}\times Y,S^{1}\times K,\alpha)=2\lambda_{CLH}(Y,K,\alpha)$.
\end{thm}

\begin{rem}
The condition that $H^{1}(X\backslash T;L_{\rho}^{\otimes2})$ vanishes
can be equivalently stated in terms of the Alexander polynomial $\triangle_{T}=\triangle_{X\backslash T}$
of the knot complement. Namely, we would require that $\triangle_{T}(\hat{\rho})\neq0$
for all the characters associated to the reducible representations
$\rho$ which satisfy the holonomy condition determined by $\alpha$.
The notation $\triangle_{T}(\hat{\rho})\neq0$ is explained in the
Appendix of the annotated version \cite{Echeverria[DraftFuruta]}
although we will not use the interpretation in terms of the Alexander
polynomial in this paper.
\end{rem}

It is important to point out that in order to study the representations
of $\pi_{1}(X\backslash T)$ into $SU(2)$, we will use the framework
of Kronheimer and Mrowka's papers \cite{MR1241873,MR1308489} on singular
gauge theory. We will review the novel features of this approach in
the next section, but the key points are the following. For pairs
$(X,\varSigma)$ consisting of a closed, oriented four manifold $X$
and an oriented surface $\varSigma$, the $SU(2)$ bundles $E$ over
$(X,\varSigma)$ are now classified by a pair of two integers $(k,l)$,
called the instanton and monopole numbers respectively.

Write $E(k,l)$ for the corresponding $SU(2)$ bundle associated to
the pair $(k,l)$. Then for each such $E(k,l)$, we can study the
space of connections $\mathcal{A}(E(k,l),\alpha)$ which have a prescribed
singular behavior along $\varSigma$. As before $0<\alpha<1/2$, and
the solutions of the anti-self-dual connections $F_{A}^{+}=0$ for
$A\in\mathcal{A}(E(0,0),\alpha)$ can again be interpreted (mod gauge)
as corresponding to representations of $\pi_{1}(X\backslash T)\rightarrow SU(2)$
which map the ``meridian'' of the torus $\mu_{T}$ to a specific
conjugacy class, if we take $\varSigma=T$. In general, the solutions
of $F_{A}^{+}=0$ for $A\in\mathcal{A}(E(k,l),\alpha)$ will be called
$\alpha$-$ASD$ connections. The moduli space of (perturbed) $\alpha$-$ASD$
connections modulo gauge will be denoted $\mathcal{M}(X,\varSigma,k,l,\alpha)$.
Finally, a gauge equivalence class will normally be denoted as $[A]$,
instead of $A$.

For technical reasons, we need to use an \textit{orbifold} metric
along $\varSigma$. This means that there will be an integer parameter
$\nu$ such that the orbifold metric along $\varSigma$ has a cone
angle of $2\pi/\nu$. Each $\alpha$ determines a set of allowable
cone parameters $\{\nu\}$, but there is no cone parameter that works
simultaneously for all values of $\alpha$. Moreover, it is convenient
to take the parameter $\alpha$ as a rational number, i.e, $\alpha\in\mathbb{Q}\cap(0,1/2)$,
so that we can talk about orbifold connections. 

From this point of view, $\lambda_{FO}(X,T,\alpha)$ corresponds to
a signed count of elements of the moduli space $\mathcal{M}(X,\varSigma,0,0,\alpha)$.
Strictly speaking, we should write $\lambda_{FO}(X,T,\alpha,\nu)$
and $\mathcal{M}(X,\varSigma,0,0,\alpha,\nu)$ since we must always
choose first a cone angle compatible with $\alpha$ before defining
our invariants. However, as we will mention near the end of the introduction
the invariants do not depend on the cone angle $\nu$ due to recent
work of Langte Ma \cite{Ma[torus]}.

In section \ref{sec:Some-examples} we will give some examples of
$\lambda_{FO}(X,T,\alpha)$, which arise from the mapping tori associated
to finite group actions on homology spheres, as well as certain circle
bundles over a 3-manifold with the homology of $S^{1}\times S^{2}$.
It will become clear that these are the orbifold versions of the calculations
of $\lambda_{FO}(X)$ done in \cite{MR2189939,MR2033479}.

\subsection{An observation of more analogies: defining $HI(Y,K,\alpha)$ and
the splitting formula}

\ 

Now we describe the second motivation for this project. Thanks to
Taubes \cite{MR1037415}, it is well known that for an integer homology
sphere $Y$, the instanton Floer homology $HI(Y)$ \cite{MR956166}
categorifies the Casson invariant in the sense that $\chi(HI(Y))=2\lambda_{C}(Y)$. 

Moreover, Frøyshov \textbf{}defined \cite{MR1910040} a refinement
of $HI(Y)$ which is known as the reduced instanton Floer homology
$HI_{red}(Y)$. From it one can define the Frøyshov $h$-invariant
$h(Y)=\frac{1}{2}(\chi(HI_{red}(Y))-\chi(HI(Y)))$, which is the precursor
to the $d$-invariant in Heegaard Floer homology and the Frøyshov
invariant in monopole Floer homology. We should notice that our conventions
are slightly different from those of Frøyshov, since we are working
with the homology (not cohomology) version of instanton Floer homology.

By the TQFT-like features of instanton Floer homology, a homology
cobordism $W:Y_{1}\rightarrow Y_{2}$ between two integer homology
spheres will induce a map between the corresponding Floer homologies
$HI(W):HI(Y_{1})\rightarrow HI(Y_{2})$. In particular, for a homology
cobordism $W:Y\rightarrow Y$ from $Y$ to itself, there are maps
$HI(W):HI(Y)\rightarrow HI(Y)$ and $HI_{red}(W):HI_{red}(Y)\rightarrow HI_{red}(Y)$
for the unreduced and reduced instanton Floer homologies. This is
an interesting case since closing up $W$ one obtains a four manifold
$X$ which is a homology $S^{1}\times S^{3}$ and for which $\lambda_{FO}(X)$
can be defined. In this case, Anvari proved \cite[Theorem A]{Anvari[2019]}
a splitting formula for $\lambda_{FO}(X)$ in terms of the Lefschetz
number $\text{Lef}(HI(W))$ of the cobordism map, which reads (in
our conventions)
\[
\lambda_{FO}(X)=\frac{1}{2}\text{Lef}(HI(W))=\frac{1}{2}\text{Lef}(HI_{red}(W))-h(Y)
\]
and is the analogue of the splitting formula \cite[Theorem A]{MR3811774}
Lin, Ruberman and Saveliev proved for a similarly-constructed invariant
$\lambda_{SW}(X)$ which is defined using the Seiberg-Witten equations
instead. Moreover, an argument due to Frøyshov \cite[Theorem 8]{MR2738582}
for the case of $\lambda_{SW}(X)$ (and which is readily adapted to
$\lambda_{FO}(X)$) shows that if $X$ is obtained as the closure
of a different homology cobordism $W':Y'\rightarrow Y'$ then $h(Y)=h(Y')$
and thus it makes sense to talk about the $h$-invariant of $X$,
which we denote as $h(X)$. In this way the splitting formula for
$\lambda_{FO}(X)$ reads 
\[
\lambda_{FO}(X)+h(X)=\frac{1}{2}\text{Lef}(HI_{red}(W))
\]
This naturally leads to the following question: 
\begin{question}
\label{second question}Is it possible to define instanton Floer homologies
$HI(Y,K,\alpha)$, $HI_{red}(Y,K,\alpha)$ for a value of $\alpha$
satisfying $\triangle_{K}(e^{-4\pi i\alpha})\neq0$, and corresponding
Frøyshov $h$-invariants $h(Y,K,\alpha)$ for the knot $K$ so that
whenever $(W,\varSigma):(Y,K)\rightarrow(Y,K)$ is a self-concordance
of $K$, there is a splitting formula involving $\lambda_{FO}(X,T,\alpha)$
and $\text{Lef}(HI(Y,Y,\alpha))$ (respectively $\text{Lef}(HI_{red}(Y,K,\alpha))$,
$h(Y,K,\alpha)$)? Moreover, can one define $HI(Y,K,\alpha)$ in such
a way that $\chi(HI(Y,K,\alpha))=\lambda_{CLH}(Y,K,\alpha)$?
\end{question}

As the reader may expect, the answers to all of these questions are
mostly in the affirmative, however, we need to point out some details
first. Our analytical framework will be based again on the singular
gauge theory developed by Kronheimer and Mrowka. In particular this
means that we will use a metric with a cone angle along the knot which
will give rise to the structure of an orbifold. 

The reader may be aware that Collin and Steer \cite{MR1703606} had
developed almost two decades ago a similar version of instanton Floer
homology for knots. Many aspects of our construction are identical
to theirs, however, as Kronheimer and Mrowka point out in \cite[Section 1.2.5]{MR2860345},
unless $\alpha=1/4$ (which corresponds to the case $k/n=1/4$ in
Collin and Steer's paper), the differential needed to define the chain
complex which gives rise to the Floer groups may be ill-defined, since
certain energy bounds for the moduli spaces cannot be guaranteed,
which are needed to appeal directly to some compactness theorems which
will be recalled in Section \ref{sec:Floer-Novikov-Homology-and}.
This is usually referred as saying that for $\alpha\neq1/4$ we are
in a non-monotone situation. Fortunately, as Kronheimer and Mrowka
also point out in that same section of their paper, there is a way
to get out of this conundrum provided one is willing to work with
an appropriate local coefficient system. The finite dimensional analogue
of this situation was first studied by Novikov in \cite{MR630459},
where an analogue of Morse theory for the case of circle valued functions
on a finite dimensional manifold was developed. 

Although the Chern-Simons functional is typically circle valued (on
the space of connections mod gauge that is), the reason why instanton
Floer homology is typically referred as an infinite dimensional Morse
theory, rather than an infinite dimensional Morse-Novikov theory,
is that the monotonicity condition usually holds, so the behavior
of the Chern Simons functional is more similar to the case of ordinary
Morse theory and not the slightly more complicated Morse-Novikov theory.
On the other hand, the use of Floer-Novikov theories is certainly
not a new thing on the symplectic versions of Floer homologies, and
was first investigated by Hofer and Salamon \cite{MR1362838} (see
also \cite{MR2199540,MR3590354,MR2553465} for more recent, but in
no way exhaustive references). 

As will be explained in Section \ref{sec:Monotonicity-and-Novikov},
there are at least three natural choices of local (Novikov) systems
we could use in our situation, although most of the time we will stick
with what we call the Universal Novikov/Local system, since it seems
to require the fewest amount of extra choices (for functoriality purposes,
that is). The end result will be that rather than defining the instanton
chain complex over a vector space, like $\mathbb{Q}$ or $\mathbb{C}$,
we will need to define it over a much larger vector space $\varLambda$
(the Novikov field), but in the end the groups $HI(Y,K,\alpha)$ we
will produce continue to be finite dimensional over $\varLambda$,
so formally many statements continue to hold. For example, the Euler
characteristic $\chi_{\varLambda}(HI(Y,K,\alpha))$ of $HI(Y,K,\alpha)$
with respect to the field $\varLambda$ recovers $\lambda_{CLH}(Y,K,\alpha)$.

What seems to be new is the idea that one can also define a reduced
version $HI_{red}(Y,K,\alpha)$ of these groups, although for the
case of $\alpha=1/4$, Christopher Scaduto and Aliakbar Daemi had
also realized this independently (and earlier as well \cite{Daemi-Scaduto[2019]}).
 We expect that their techniques could be used to further understand
the family of groups $HI(Y,K,\alpha)$, for arbitrary values of $\alpha$.
In any case, here is the answer to Question \ref{second question}.
\begin{thm}
Suppose that $K$ is an oriented knot inside an integer homology sphere
$Y$ and that a parameter $\alpha\in\mathbb{Q}\cap(0,1/2)$ is chosen
so that $\triangle_{K}(e^{-4\pi i\alpha})\neq0$. 

There is a family of vector spaces $HI_{i}(Y,K,\alpha)$ for $i\in\mathbb{Z}/4\mathbb{Z}$
, which are finite dimensional over a Novikov field $\varLambda$,
and which we will call the \textbf{instanton Floer-Novikov knot homology
groups of the knot $K$. }They recover the Casson-Lin-Herald invariant
(and hence the Tristram-Levine knot signatures), in the sense that
\[
\chi_{\varLambda}(HI(Y,K,\alpha))=\lambda_{CLH}(Y,K,\alpha)=4\lambda_{C}(Y)+\frac{1}{2}\sigma_{K}(e^{-4\pi i\alpha})
\]
where we are using an absolute $\mathbb{Z}/2\mathbb{Z}$ grading of
these Floer groups in order to compute the Euler characteristic with
respect to $\varLambda$. Moreover, each such $HI_{i}(Y,K,\alpha)$
admits a refinement $HI_{red,i}(Y,K,\alpha)$, which again will be
a finite dimensional vector space over the Novikov field $\varLambda$,
and which we will call the \textbf{reduced instanton Floer-Novikov
knot homology groups of the knot $K$. }Given these two Floer groups
one can define the\textbf{ Frøyshov knot-invariants 
\[
h(Y,K,\alpha)=\chi_{\varLambda}(HI^{red}(Y,K,\alpha))-\chi_{\varLambda}(HI(Y,K,\alpha))
\]
}For the case of $\alpha=1/4$, no Novikov field is needed and in
fact the Floer groups can be defined over $\mathbb{Q}$ (for example).
\end{thm}

\begin{rem}
In order to be completely accurate, we should specify the cone parameter
being used in the resulting groups, and write $HI(Y,K,\alpha,\nu)$,
$HI_{red}(Y,K,\alpha,\nu)$ and $h(Y,K,\alpha,\nu)$. However, we
expect all of these to be independent of $\nu$, which is why we will
continue to suppress $\nu$ from our notation. 
\end{rem}

Some examples/properties of these groups and the knot $h$-invariants
will be discussed in Section 8, but now we have to answer the other
part of Question \ref{second question}, namely, what is the relation
between $\lambda_{FO}(X,T,\alpha)$ and $\text{Lef}(HI(Y,K,\alpha))$
in the case of a self-concordance of a knot? 

First of all, we must observe that currently the functoriality properties
of the Floer groups $HI(Y,K,\alpha)$ are weaker than their non-singular
counterparts $HI(Y)$. By this we mean that a homology concordance
$(W,\varSigma):(Y_{1},K_{1})\rightarrow(Y_{2},K_{2})$ for which both
of $HI(Y_{1},K_{1},\alpha)$ and $HI(Y_{2},K_{2},\alpha)$ are defined,
may not induce a cobordism map between the Floer groups. The reason
for this has to do once again with the reducible connections on the
cobordism. In the case where $\varSigma$ is not present, there are
maps between $HI(Y_{1})$ and $HI(Y_{2})$ essentially because the
trivial connection $\theta_{W}$ on the cobordism, which is the unique
reducible up to gauge since $H_{1}(W;\mathbb{Z})=0$, is automatically
isolated and non-degenerate given that $H^{1}(W;\mathfrak{g}_{\theta_{W}})=H^{1}(W;\mathbb{R})\otimes\mathbb{R}^{3}=0$
and $H^{2,+}(W;\mathfrak{g}_{\theta_{W}})=H^{2,+}(W;\mathbb{R})\otimes\mathbb{R}^{3}=0$. 

In the case of a homology concordance there is still a unique reducible
connection $\theta_{W,\alpha}$ after we have fixed a choice of $\alpha$
since $H_{1}(W\backslash\varSigma;\mathbb{Z})=\mathbb{Z}$, but now
it is a priori not immediate that it will be isolated or non-degenerate.
The cobordisms for which this will happen will be called \textbf{$\alpha$-admissible,
}and after we explain more of the setup we are using in sections \ref{sec:Review-of-the}
and \ref{sec:Floer-Novikov-Homology-and} it will become clear that
the condition we are after is the following.
\begin{defn}
\label{def: alpha admissible}A homology concordance $(W,\varSigma):(Y_{1},K_{1})\rightarrow(Y_{2},K_{2})$
between two knots $K_{1}\subset Y_{1}$ and $K_{2}\subset Y_{2}$
is a homology cobordism $W:Y_{1}\rightarrow Y_{2}$ together with
an embedded annulus $\varSigma:K_{1}\rightarrow K_{2}$. 

The pair $(W,\varSigma)$ will be called \textbf{$\alpha$-admissible
}for $\alpha\in(0,1/2)$, if $\triangle_{K_{1}}(e^{-4\pi i\alpha})\neq0$,
$\triangle_{K_{2}}(e^{-4\pi i\alpha})\neq0$ and moreover $H^{1}(W\backslash\varSigma;L_{\theta_{W,\alpha}}^{\otimes2})=0$.
Here $\theta_{W,\alpha}$ denotes the unique reducible (up to gauge)
compatible with the holonomy condition $\alpha$, and $E=L_{\theta_{W,\alpha}}\oplus L_{\theta_{W,\alpha}}^{-1}$
denotes the decomposition of the trivial $SU(2)$ bundle over $W\backslash\varSigma$
induced by $\theta_{W,\alpha}$.
\end{defn}

We will discuss to what extent this condition on the cobordism is
really needed at the end of this section, but for now let's assume
that it holds. In that case it is straightforward to see that we have
cobordism maps $HI(W,\varSigma,\alpha)$ between $HI(Y_{1},K_{1},\alpha)$
and $HI(Y_{2},K_{2},\alpha)$ as explained in Section \ref{sec:Floer-Novikov-Homology-and}. 

In order to obtain a splitting formula analogous to the one $\lambda_{FO}(X)$
satisfies we need an additional piece of data. In general, it is more
accurate to regard $\lambda_{FO}(X,T,\alpha)$ as a degree zero Donaldson
invariant, associated to the moduli space $\mathcal{M}(X,T,0,0,\alpha)$,
since in practice one needs to perturb the $\alpha$-ASD equation
$F_{A}^{+}=0$, not the $\alpha$-flat equation $F_{A}=0$, in order
to define $\lambda_{FO}(X,T,\alpha)$. 

Interestingly enough, in the case that $\alpha\neq1/4$, there are
other moduli spaces $\mathcal{M}(X,T,k,l,\alpha)$ whose expected
dimension is zero (and are a priori non-empty), which means that there
are additional candidates for degree zero Donaldson invariants. As
we will explain in more detail in Section \ref{sec:Singular-Orbifold-Furuta-Ohta},
whenever $k$ is an integer such that $k(1-4\alpha)\geq0$, the moduli
space $\mathcal{M}(X,T,k,-2k,\alpha)$ is of expected dimension zero,
and can in fact be used to define a Donaldson-type invariant $D_{0}(X,T,\alpha,k)$
(see definition \ref{Def deg 0 Donaldson invariants} for a precise
statement). In particular, we can create a formal power series 
\begin{equation}
\sum_{k\in\mathbb{Z}}D_{0}(X,T,\alpha,k)T^{-\mathcal{E}_{top}(X,T,k,-2k,\alpha)}\label{formal power series}
\end{equation}
where $\mathcal{E}_{top}(X,T,k,-2k,\alpha)$ denotes the topological
energy of the moduli space $\mathcal{M}(X,T,k,-2k,\alpha)$. In fact,
this energy is equal to the quantity $k(1-4\alpha)$ which explains
the restriction $k(1-4\alpha)\geq0$ , since negative energy moduli
spaces of $\alpha$-$ASD$ instantons are always empty, just as in
the ordinary case where no holonomy condition is present. Now, the
formal power series \ref{formal power series} is in fact the sort
of object a Novikov field is equipped to handle, in other words, we
can think of the series \ref{formal power series} as an element of
$\varLambda$. 

This is a good thing, since the Lefschetz number of a degree-preserving
linear transformation $L:V\rightarrow V$ between two finite dimensional
vector spaces over some field $\mathbb{F}$ will be an element of
$\mathbb{F}$. In the case of a self-concordance $(W,\varSigma):(Y,K)\rightarrow(Y,K)$
which is $\alpha$-admissible we should think of $V$ as being either
$HI(Y,K,\alpha)$ (or $HI_{red}(Y,K,\alpha)$), $\mathbb{F}$ as the
Novikov field $\varLambda$ and $L$ as the map on corresponding Floer
groups induced by the cobordism. Therefore, one would expect that
the Lefschetz number of the map induced by the self-concordance equals
\ref{formal power series}.

That will be the case, modulo a final caveat. There is an action of
$H^{1}(X;\mathbb{Z}/2)$ on the moduli spaces $\mathcal{M}(X,T,k,-2k,\alpha)$,
which has been studied ad nauseam in other situations involving Instanton
Floer homology \cite{MR1362829,MR2081729,MR2033479,MR2209367,MR3704245,MR2805599}.
Unless the action of $H^{1}(X;\mathbb{Z}/2)$ is free on the irreducible
part $\mathcal{M}^{*}(X,T,k,-2k,\alpha)$ of the moduli space, one
cannot expect a relationship between the formal power series and the
Lefschetz number to hold, since there will be an ambiguity when solving
the gluing problem, as explained in \cite[Section 5]{MR2805599}.
 In fact, it will turn out that the action of of $H^{1}(X;\mathbb{Z}/2)$
on $\mathcal{M}^{*}(X,T,k,-2k,\alpha)$ is free,  in which case the
splitting (or Lefschetz) formula \ref{spliting intro} is essentially
a consequence of \cite[Proposition 5.5]{MR2805599}, which in Kronheimer
and Mrowka's situation comes from the assumption that the subgroup
$\phi^{*}\subset H^{1}(W^{*},S^{*},\mathbf{P}^{*})$ (in their notation)
satisfies a non-integral condition.

With these remarks in place, we can finish answering Question \ref{second question}. 
\begin{thm}
\label{splitting formula Furuta Ohta } Let $K\subset Y$ an oriented
knot inside an oriented integer homology sphere and $(W,\varSigma):(Y,K)\rightarrow(Y,K)$
a self-concordance of $K$. Consider $\alpha\in\mathbb{Q}\cap(0,1/2)$
such that $\triangle_{K}(e^{-4\pi i\alpha})\neq0$.

If $(W,\varSigma)$ is $\alpha$-admissible then there are degree-preserving
maps 
\begin{align*}
HI(W,\varSigma,\alpha):HI(Y,K,\alpha)\rightarrow HI(Y,K,\alpha)\\
HI_{red}(W,\varSigma,\alpha):HI(Y,K,\alpha)\rightarrow HI(Y,K,\alpha)
\end{align*}
induced by the cobordism $(W,\varSigma)$.

If $(X,T)$ is the pair obtained by closing up $(Y,K)$ then for $k\neq0$
the invariants $D_{0}(X,T,\alpha,k)$ are always well defined. Moreover,
$\lambda_{FO}(X,T,\alpha)=D_{0}(X,T,\alpha,0)$ can be defined if
and only if $(W,\varSigma)$ is $\alpha$-admissible, and  the following
\textbf{splitting formula }holds
\begin{equation}
\sum_{k\in\mathbb{Z}}D_{0}(X,T,\alpha,k)T^{-\mathcal{E}(X,T,k,-2k,\alpha)}=2\text{Lef}(HI(W,\varSigma,\alpha))=2\text{Lef}(HI_{red}(W,\varSigma,\alpha))-2h(Y,K,\alpha)\label{spliting intro}
\end{equation}

Finally, if $(X,T)$ arises as the closure of another self-concordance
$(W',\varSigma'):(Y',K')\rightarrow(Y',K')$ and $\triangle_{K'}(e^{-4\pi i\alpha})\neq0$
as well, then 
\begin{equation}
\text{Lef}(HI_{red}(W,\varSigma,\alpha))=\text{Lef}(HI_{red}(W',\varSigma',\alpha))\label{equality Lefschetz}
\end{equation}
and hence 
\[
h(Y,K,\alpha)=h(Y',K',\alpha)
\]
In particular,we can define a \textbf{Frøyshov torus invariant} $h(X,T,\alpha)$
for the embedded torus as $h(Y,K,\alpha)$ given an arbitrary ``slice''
$(Y,K)$.
\end{thm}

\begin{rem}
a) Again, a dependence on the cone angle is implicit. 

b) The statement regarding the equality of the Lefschetz numbers \ref{equality Lefschetz}
will follow the strategy employed by Frøyshov in the case of $\lambda_{SW}(X)$,
which as we mentioned before can be found in \cite[Theorem 8]{MR2738582}.
The argument is essentially the same, the only thing one needs to
verify is a suitable version of \cite[Lemma 10]{MR2738582} for the
case of matrices with coefficients in a Novikov field. 

c) Besides the case of $\lambda_{FO}(X)$ and $\lambda_{SW}(X)$,
there are other instances where a Lefschetz formula has appeared in
a similar context, in fact, in some ways more closely related to our
situation. In \cite{Juhasz-Zemker[2018]} Juhász and Zemke compute
the effect of concordance surgery on the Ozsváth-Szabó 4-manifold
invariant. One is given a smooth closed oriented 4-manifold $X$ with
$b_{2}^{+}(X)\geq2$ and a homologically essential torus $T\subset X$
with trivial self-intersection. If $(I\times Y,\varSigma)$ is a self-concordance
of a knot $K\subset Y$, then there is a natural torus $T_{C}$ inside
$S^{1}\times Y$ which they use to construct a 4-manifold $X_{C}$
generalizing the Fintushel and Stern knot surgery \cite{MR1650308}.
Namely, one forms the 4-manifold $X_{C}=(X\backslash N(T))\cup_{\phi}W_{C}$
where $N(T)$ denotes a tubular neighborhood of $T$ and $W_{C}=(S^{1}\times Y)\backslash N(T_{C})$,
while $\phi$ is a gluing diffeomorphism. In (knot) Heegaard-Floer
homology there is a concordance map 
\[
\hat{F}_{C}:\widehat{HFK}(Y,K)\rightarrow\widehat{HFK}(Y,K)
\]
 and an associated graded Lefschetz number 
\[
\text{Lef}_{t}(C)=\sum_{i\in\mathbb{Z}}\text{Lef}\left(\hat{F}_{C}\mid_{\widehat{HFK}(Y,K,i)}:\widehat{HFK}(Y,K,i)\rightarrow\widehat{HFK}(Y,K,i)\right)t^{i}
\]
Then Theorem 1.1 in \cite{MR1650308} shows that $\varPhi_{X_{C};\omega}=\text{Lef}_{t_{1}}(C)\varPhi_{X;\omega}$,
where $\varPhi_{X;\omega}$ denotes a version of the Ozsváth-Szabó
4-manifold invariant twisted by a certain collection of closed 2-forms. 

d) Another situation very close to our splitting formula is the one
Kronheimer and Mrowka found for the Seiberg-Witten invariants on a
closed 4-manifold \cite[Section 32.1]{MR2388043}. More precisely,
one is given a closed oriented 4-manifold with $b_{2}^{+}(X)\geq2$
and within it a separating hypersurface $Y$, so that $X=X_{1}\cup X_{2}$
, $\partial X_{1}=Y$ and $\partial X_{2}=-Y$. If $\omega_{X}$ is
a 2-form used to perturb the Seiberg Witten equations on $X$ and
$\omega$ is the restriction of $\omega_{X}$ to $X$, then for a
torsion spin-c structure $\mathfrak{s}$ on $Y$, in general $\omega$
will induce a non-balanced perturbation in the sense of \cite[Chapter 32]{MR2388043},
which essentially means that a Novikov system $\varLambda$ is required
to define the corresponding monopole Floer homology groups $HM(Y,\mathfrak{s},\omega)$.
Proposition 32.1.1 in \cite{MR2388043} shows that 
\begin{equation}
\sum_{\mathfrak{s}_{X}\mid_{Y}=\mathfrak{s}}T^{-\mathcal{E}_{\omega_{X}}^{top}(\mathfrak{s}_{X})}SW(X,\mathfrak{s}_{X})=\left\langle \psi_{+},\psi_{-}\right\rangle _{\omega_{\mu}}\label{pairing KM}
\end{equation}
Here the sum is taking place over all spin-c structures on $X$ which
restrict to the given one on $Y$, $SW(X,\mathfrak{s}_{X})$ denotes
the Seiberg-Witten invariant associated to such a spin-c structure,
$\mathcal{E}_{\omega_{X}}^{top}$ is a perturbed topological energy,
and the right hand side denotes a pairing of two relative invariants
($\psi_{\pm}$ being an element of $HM(\pm Y,\mathfrak{s},\omega)$),
the pairing taking place with respect to the Novikov ring. Notice
that one can interpret our splitting formula \ref{spliting intro}
as an analogue of the pairing formula \ref{pairing KM} when the 3-manifold
is non-separating (rather than separating), and with the Donaldson
invariants (rather than the Seiberg-Witten invariants). From this
perspective, the different bundles $E(k,-2k)$ are playing a role
analogous to the one the different isomorphism classes of spin-c structures
play in the Seiberg-Witten context.
\end{rem}

In section 8 we will give some applications of the splitting formula
\ref{spliting intro}, including a proof that $\lambda_{FO}(S^{1}\times Y,S^{1}\times K,\alpha)=2\lambda_{CLH}(Y,K,\alpha)$.
But as a way to entice the reader, we now prove the following.
\begin{thm}
Suppose that $K\subset Y$ and $K'\subset Y'$ are knots and there
exists a concordance $(C,A):(Y,K)\rightarrow(Y',K')$ where $C$ is
a homology $[0,1]\times S^{3}$ and $A$ an embedded annulus. Then
for any $\alpha\in\mathbb{Q}\cap(0,1/2)$  such that $(C,A)$ is
$\alpha$-admissible (if any), we have that $h(Y,K,\alpha)=h(Y',K',\alpha)$,
i.e, the knot $h$-invariants are $\alpha$-concordance invariants. 

In particular, $h(Y',K',\alpha)$ will vanish  whenever $K'$ is
$\alpha$-slice in the sense that there is an $\alpha$-concordance
$(C,A):(S^{3},\circ)\rightarrow(Y',K')$, where $\circ$ is the unknot.
\end{thm}

\begin{proof}
Observe that we can form a self-concordance of the knot $(Y,K)$ obtained
by ``stacking'' $(C,A)$ with the opposite concordance $(\bar{C},\bar{A}):(Y',K')\rightarrow(Y,K)$
obtained by reversing orientations, 
\[
(W,\varSigma)=(\bar{C},\bar{A})\circ(C,A):(Y,K)\rightarrow(Y,K)
\]
From $(W,\varSigma)$ we can close it up to obtain $(X,T)$. The corresponding
closed 4-manifold $(X,T)$ can also be obtained from doing the staking
in the opposite order, namely $(C,A)\circ(\bar{C},\bar{A})$, so the
last statement in the splitting formula (Theorem \ref{splitting formula Furuta Ohta })
implies that the knot $h$-invariants are the same, i.e, 
\[
h(Y,K,\alpha)=h(Y',K',\alpha)
\]

The last claim follows from the fact that for the unknot, for all
$\alpha\in\mathbb{Q}\cap(0,1/2)$, we have $HI(S^{3},\circ,\alpha)=HI_{red}(S^{3},\circ,\alpha)=0$,
and hence $h(S^{3},\circ,\alpha)=0$.
\end{proof}
We conclude this section by mentioning the ``flip symmetry'' property
our Floer groups enjoy (in the context of singular gauge theory on
4-manifolds it was introduced by Kronheimer and Mrowka, \cite[Lemma 2.12]{MR1241873}).
 
\begin{thm}
\label{Flip symmetry} Let $K\subset Y$ be an oriented knot inside
an oriented integer homology sphere and choose $\alpha\in\mathbb{Q}\cap(0,1/2)$
such that $\triangle_{K}(e^{-4\pi i\alpha})\neq0$. Then there is
a flip isomorphism 
\[
\mathcal{F}:HI(Y,K,\alpha)\rightarrow HI\left(Y,K,\frac{1}{2}-\alpha\right)
\]
which is grading preserving, with an analogous isomorphism for the
reduced groups $HI_{red}(Y,K,\alpha)$. In particular, 
\[
h(Y,K,\alpha)=h\left(Y,K,\frac{1}{2}-\alpha\right)
\]

Likewise, for a pair $(X,T)$ we have for all $k\neq0$ that
\[
D_{0}(X,T,k,\alpha)=D_{0}\left(X,T,-k,\frac{1}{2}-\alpha\right)
\]
and a similar statement holds for $\lambda_{FO}(X,T,\alpha)$ whenever
it can be defined.
\end{thm}

\subsection{Some Updates}

Since the first version of this paper appeared on the arxiv, some
important developments have taken place which we now mention. The
reader is referred to the section ``Some Speculations and Further
Directions of Work'' in the first version of this paper if interested
in some speculations we had indulged ourselves in. 

As mentioned before, $\lambda_{CLH}(Y,K,\alpha)=4\lambda_{C}(Y)+\frac{1}{2}\sigma_{K}(e^{-4\pi i\alpha})$,
which can be rewritten as 
\[
\lambda_{FO}(S^{1}\times Y,S^{1}\times K,\alpha)=8\lambda_{FO}(S^{1}\times Y)+\sigma_{K}(e^{-4\pi i\alpha})
\]
This suggests that if one could find a non-gauge theoretic definition
of a tori signature $\sigma_{T}$ for $T\subset X$, then the formula
\begin{equation}
\lambda_{FO}(X,T,\alpha)=8\lambda_{FO}(X)+\sigma_{T}(e^{-4\pi i\alpha})\label{definition tori signature}
\end{equation}
 should hold {[}assuming that $\sigma_{T}$ is normalized so that
$\sigma_{S^{1}\times K}=\sigma_{K}${]}. 

Regarding the non-gauge theoretic definition of a tori signature,
Ruberman \cite{Ruberman[signature]} defined a signature-type invariant
for embedded tori $\sigma_{T}$ on a homology $S^{1}\times S^{3}$,
which agrees with the Levine-Tristram invariant in the product case.
So the question is whether $\sigma_{T}$ satisfies formula (\ref{definition tori signature})
so that our conjecture would hold.

In even more recent work, Langte Ma \cite[Theorem 1.1]{Ma[torus]}
showed that the difference (whenever defined)
\[
\lambda_{FO}(X,T,\alpha)-8\lambda_{FO}(X)
\]
 equals $\tilde{\rho}_{\varphi_{2\alpha}}(X_{0}(T),Q)$, which is
a spectral invariant defined by Ma in the spirit of the Atiyah-Singer
$\rho$ invariant. The notation is explained in the introduction of
\cite{Ma[torus]}, but since $\tilde{\rho}_{\varphi_{2\alpha}}(X_{0}(T),Q)$
is a topological invariant, Ma's results shows the independence of
$\lambda_{FO}(X,T,\alpha)$ on the cone angle being used. Likewise,
Ma shows \cite[Theorem 1.2]{Ma[torus]} that 
\[
\tilde{\rho}_{\varphi_{2\alpha}}(X_{0}(T),Q)=\sigma_{2\alpha}(X,T)
\]
 where again the right hand side refers to the invariant defined by
Ruberman. Therefore, one can conclude that \cite[Corollary 1.3]{Ma[torus]}
\[
\lambda_{FO}(X,T,\alpha)-8\lambda_{FO}(X)=\sigma_{2\alpha}(X,T)
\]

\subsection{Outline of the paper:}

In section \ref{sec:Review-of-the} we review the main features of
singular gauge theory as developed by Kronheimer and Mrowka. 

Sections \ref{sec:Monotonicity-and-Novikov}, \ref{sec:Floer-Novikov-Homology-and},
\ref{sec:Reduced-Version} discuss how the monotonicity issue arises,
the definition of $HI(Y,K,\alpha)$ and of $HI_{red}(Y,K,\alpha)$
respectively. 

Sections \ref{sec:Singular-Orbifold-Furuta-Ohta} defines $\lambda_{FO}(X,T,\alpha)$
as well as its relatives $D_{0}(X,T,\alpha)$. Section \ref{sec:The-Splitting-Formula}
discusses the splitting formula $\lambda_{FO}(X,T,\alpha)$ and $D_{0}(X,T,\alpha)$
satisfy. Finally, section \ref{sec:Some-examples} includes some calculations
of $\lambda_{FO}(X,T,\alpha)$ for certain tori inside mapping tori
as well as some properties of the groups $HI(Y,K,\alpha)$ and $HI_{red}(Y,K,\alpha)$,
like the usual duality isomorphisms under orientation reversal as
well as the flip symmetry property.

In order to reduce the length of the paper, we deleted an appendix
which includes a summary of cohomology with local coefficients and
the Alexander polynomial for $CW$ complexes, which can still be found
in the annotated version of the paper \cite{Echeverria[DraftFuruta]}.

\ 

\textbf{Acknowledgments:}

The gestation and development of this project took place while the
author was a visiting student at the Massachusetts Institute of Technology
so I would like to thank the MIT Mathematics Department for their
tremendous hospitality. Moreover, I would like to thank Tom Mrowka,
Danny Ruberman and Nikolai Saveliev for several discussions which
were indispensable throughout the development of the paper, as well
as my advisor Tom Mark for his constant encouragement and useful suggestions.
I would also like to thank Langte Ma for pointing out a mistake in
an earlier version of this paper, involving the values of $\lambda_{FO}(X,T,\alpha)$
in the mapping torus case.

I would also like to thank Peter Kronheimer, Chris Herald, Christopher
Scaduto, Aliakbar Daemi, Jianfeng Lin, Matthew Stoffregen, Paul Feehan,
Raphael Zentner, Jennifer Hom, Michael Usher, Michael Miller and Donghao
Wang for direct and indirect conversations on this project.

\section{\label{sec:Review-of-the}Review of the Orbifold Setup}

As we mentioned in the introduction, the Floer homology we will construct
for knots is based on an orbifold approach, which was pioneered by
Kronheimer and Mrowka in their papers on gauge theory and embedded
surfaces \cite{MR1241873,MR1308489}. Slightly different versions
of this setup have been employed by them over the years \cite{MR3880205,MR2860345,MR2805599},
so the main purpose of this section is to give a brief review of their
orbifold construction, as well as discussing the advantages and disadvantages
of different strategies for building Floer homologies for knots.

Let $X$ be a smooth, oriented, closed four manifold and $\varSigma$
any oriented, smoothly embedded surface inside $X$. Similarly, suppose
that $Y$ is a smooth, oriented, closed three manifold and $K$ is
an oriented knot or link inside $Y$. For either of the pairs $(X,\varSigma)$,
$(Y,K)$, we want to do gauge theory using connections on an $SU(2)$
bundle which are singular along $\varSigma$ (or $K$). The nature
of the singularity is precisely what distinguishes one approach from
the other.

A natural strategy would be to work on either of the complements $X\backslash\varSigma$
, $Y\backslash K$ and to choose a riemannian metric which is simply
the restriction to one of these complements of a smooth metric $g$
defined on all of $X$ (or $Y$). Analytically, this means that one
is working on an open manifold with an incomplete metric, which from
the Sobolev package point of view is not a nice situation. However,
provided one works with weighted Sobolev spaces, one can still define
a reasonable family of function spaces \cite[Section 3 (i)]{MR1241873}.

The next issue is whether one wants to work with an $SU(2)$ bundle
$E$ which is defined only on the complement $X\backslash\varSigma$
(respectively $Y\backslash K$). In this case the major problem is
how to recover useful topological information. This was called the
\textit{extension problem} by Kronheimer and Mrowka, and discussed
in some detail in \cite[Section 2 (iv)]{MR1241873}.

The setup used in \cite{MR1241873,MR1308489,MR2860345}, which will
follow as well, is one for which the bundle $E$ is defined over the
entire manifold $X$ (or $Y$), subject to the condition that \textit{topologically}
the bundle $E$ admits a reduction to a $U(1)$ bundle along $\varSigma$
(or $K$). For example, this means that for $(X,\varSigma)$ an $SU(2)$
bundle $E$ should decompose as 
\begin{equation}
E\mid_{\nu(\varSigma)}=L\oplus L^{*}\label{reduction of E}
\end{equation}
where $\nu(\varSigma)$ is a closed tubular neighborhood of $\varSigma$
, and $L\rightarrow\nu(\varSigma)$ some complex line bundle compatible
with the hermitian metric. In particular, this allows us to define
two topological invariants, the \textbf{instanton number $k$ 
\begin{equation}
k=c_{2}(E)[X]\label{instanton number}
\end{equation}
}as well as the \textbf{monopole number $l$
\begin{equation}
l=-c_{1}(L\mid_{\varSigma})[\varSigma]\label{monopole number}
\end{equation}
 }Notice that the latter invariant is not present when $\varSigma=\emptyset$.
The reader may be interested in knowing that for their ``Khovanov
is an unknot-detector'' paper Kronheimer and Mrowka do work with
bundles which in principle are not defined over the entire manifold
\cite[Section 2.1]{MR2805599}.

For the 3-manifold case, the bundle topology is not very interesting
in our situation so we can consider $E$ simply as the trivial $SU(2)$
bundle over $Y$, i.e, $E=Y\times SU(2)$, together with the trivial
$U(1)$ sub-bundle corresponding to the diagonal embedding of $U(1)$
in $SU(2)$. In both cases what we will be more interesting is the
space of connections we will use, which we now describe.

These connections have the following singular behavior: after choosing
a riemannian metric we consider $\nu(\varSigma)$ as being diffeomorphic
to the unit disk bundle of the normal bundle, and we choose a connection
1-form $i\eta$ for the circle bundle $\partial\nu(\varSigma)$, so
that it coincides with the 1-form $d\theta$ on each circle fibre,
where $(r,\theta)$ are polar coordinates in some local trivialization
of the disk bundle ($dr\wedge d\theta$ orients the normal planes).
By radial projection $\eta$ is extended to $\nu(\varSigma)\backslash\varSigma$.

For each real number $0<\alpha<1/2$, we consider the matrix 1-form
with values in $\mathfrak{su}(2)$ which behaves near $\varSigma$
as
\[
i\left(\begin{array}{cc}
\alpha & 0\\
0 & -\alpha
\end{array}\right)d\theta
\]
Locally, the holonomy of a connection $A$ on $X\backslash\varSigma$
whose matrix connection coincides with the previous matrix 1-form
near $\varSigma$ on the positively-oriented small circles of constant
$r$ is approximately 
\begin{equation}
\exp\left(\begin{array}{cc}
-2\pi i\alpha & 0\\
0 & 2\pi i\alpha
\end{array}\right)\label{asymptotic holonomy}
\end{equation}
We are excluding the values $\alpha=0,1/2$ because those cases give
no new information, i.e, they essentially correspond to the situation
where the connections extend smoothly across the singularity (as $SU(2)$
or $SO(3)$ connections more generally).

For a more global description, using the reduction (\ref{reduction of E}),
choose any smooth $SU(2)$ connection $A^{0}$ on $E$ which reduces
in the same way, i.e, 
\[
A^{0}\mid_{\nu(\varSigma)}=\left(\begin{array}{cc}
b & 0\\
0 & -b
\end{array}\right)
\]
where $b$ is a smooth connection on $L$. Define the \textbf{model
connection} $A^{\alpha}$ on $E\mid_{X\backslash\varSigma}$ as 
\begin{equation}
A^{\alpha}=A^{0}+i\beta(r)\left(\begin{array}{cc}
\alpha & 0\\
0 & -\alpha
\end{array}\right)\eta\label{model connection}
\end{equation}
where $\beta$ is a smooth cut-off function equal to $1$ in a neighborhood
of $0$ and equal to $0$ for $r\geq1/2$. The model connection has
holonomy around small linking circles asymptotically equal to (\ref{asymptotic holonomy}).
The connections Kronheimer and Mrowka consider can be written as $A=A^{\alpha}+a$,
where $a\in\varOmega^{1}(X\backslash\varSigma,\mathfrak{su}(2))$,
and as always necessary Sobolev spaces are needed \cite[Eq. 2.2, 2.3]{MR1241873}.
\begin{rem}
In the 3-manifold case we can take $A^{0}$ to be the (usual) trivial
connection $\theta$ on the trivial $SU(2)$ bundle $E$ from before.
In the nomenclature of \cite[Section 3.6]{MR2805599}, our situation
corresponds to one where $\triangle$ is trivial, where $\triangle$
is a local system determining the possible extensions of the bundle
across the singularity.
\end{rem}

Most of the usual gauge theory technology can be extended to this
setup: for example, one can define moduli spaces of $ASD$ connections
with the singular behavior described before, compute the expected
dimension of the moduli space, as well as the energy of the connection
in terms of topological quantities. These moduli spaces are indexed
by the instanton and monopole number, so we will write them typically
as $\mathcal{M}(X,\varSigma,k,l,\alpha)$ or $\mathcal{M}(k,l)$,
depending on the context.

We will remind the reader about the precise formulas once we explain
the monotonicity condition usually required before defining Floer
homologies. For now, it suffices to say that there is one key aspect
missing by working in the previous setup described. Namely, Kronheimer
and Mrowka were not able to show that if one has a sequence of $ASD$
connections $A_{i}$ belonging to some moduli space $\mathcal{M}(X,\varSigma,k,l,\alpha)$,
then the limiting connection $A_{\infty}$, which in principle belongs
to a different moduli space $\mathcal{M}(X,\varSigma,k',l',\alpha)$
because of Uhlenbeck bubbling, must belong to a moduli space of lower
expected dimension. The fact that this should happen, that is, that
after bubbling one should land in a moduli space of smaller dimension,
would follow if Conjecture 8.2 in \cite{MR1241873} were proven true.
Fortunately, they were able to prove this fact by modifying slightly
the previous setup and working instead with \textit{orbifolds}.

On a first level, this means that instead of using the restriction
of a smooth metric on $X$ (or $Y$) to $X\backslash\varSigma$ (or
$Y\backslash K$), we consider metrics which have a cone-like singularity
along the surface (knot) \cite[Section 2, iii)]{MR1241873}. This
means that near the surface $\varSigma$, the metric is modeled on
\[
ds^{2}=du^{2}+dv^{2}+dr^{2}+\left(\frac{1}{\nu^{2}}\right)r^{2}d\theta^{2}
\]
where $u,v$ are coordinates on $\varSigma$, and $\nu\geq1$ is a
real parameter. To obtain a global metric on $\nu(\varSigma)$ of
this form replace $du^{2}+dv^{2}$ by the pull-back of any smooth
metric on $\varSigma$, and replace the form $d\theta$ by the 1-form
$\eta$. The metric is then patched to a smooth one on the complement
of $\nu(\varSigma)$ and extended to the rest of $X$. The resulting
metric has a cone-angle of angle $2\pi/\nu$ in the normal planes
to $\varSigma$. When $\nu$ is an integer greater than $1$ the metric
is an orbifold metric: locally there is a $\nu$-fold branched cover
on which the metric is smooth.

An advantage of the orbifold perspective is that it allow us to compute
certain quantities (like gradings for the generators of the Floer
complex) in terms of equivariant indices on appropriate branched covers
along the knot (and/or surface). Moreover, in the orbifold setup one
can still use a Coulomb slice determined explicitly by $\ker\check{d}_{A}^{*}$
, where $\check{}$ emphasizes that we are thinking of this operator
as being defined on an orbifold. If we were to use instead the non-orbifold
setup (with the smooth metric and the weighted Sobolev spaces mentioned
before), then the construction of the slices is more subtle because
of the lack of the usual $d_{A}^{*}$ operator \cite[Lemma 5.5]{MR1241873}.

However, as we pointed out in the introduction, a drawback of the
orbifold approach is that currently there is no way to show that the
groups one obtains are independent of the choice of orbifold structure
(cone angle) one uses. Moreover, each value of holonomy $\alpha$
determines an allowable set of possible orbifold cone angles $\nu$
compatible with $\alpha$ one can use for defining the Floer groups.
These values are described in Proposition 4.8, Lemma 4.9 (and the
remark after it), as well as Proposition 4.17 of \cite{MR1241873}. 

In practice we can think of the allowable $\nu$ in the following
way \cite[Remark p.894]{MR2860345}. If our holonomy parameter is
a rational number $\alpha\in\mathbb{Q}\cap(0,1/2)$, then $\nu>0$
can be taken to be any integer satisfying the property that 
\begin{equation}
\exp\left(\begin{array}{cc}
-2\pi i\alpha\nu & 0\\
0 & 2\pi i\alpha\nu
\end{array}\right)\in Z(SU(2))=\pm Id_{2\times2}\label{condition cone angle}
\end{equation}
 Therefore, if we write $\alpha$ as $\frac{p}{q}$ , where $p,q$
are relatively prime, it is not difficult to see from this description
that taking $\nu=q$ suffices. However, if $q$ happens to be even,
then in fact $\nu=\frac{q}{2}$ also satisfies the condition \ref{condition cone angle}.
This is why for $\alpha=\frac{1}{4}$ one can choose $\nu=2$ , so
that the metric has a cone angle $\pi$ along the knot (a \textit{bifold}
in the terminology of Kronheimer and Mrowka \cite{MR3880205}). If
one is annoyed by this dependence on the cone angle, it is reasonable
to choose the smallest cone parameter $\nu$ that works for $\alpha$
as the ``canonical'' cone angle. Incidentally, notice that the this
canonical cone angle works simultaneously for $\alpha$ and $\frac{1}{2}-\alpha$
so that the flip symmetry $\alpha\rightarrow\frac{1}{2}-\alpha$ can
be analyzed with this common choice. 

\begin{rem}
Besides describing the singular metric in the orbifold setup, one
can also adapt the notion of bundle and connections to this situation.
The definitions given in \cite{MR3880205} are particularly convenient
for our purposes (it also has the advantage of being suitable for
the case that the bundle does not extend across the singularity, which
as we said is more general than what we will consider).

Every point in an orbifold has a neighborhood $U$ which is the codomain
of an orbifold chart 
\[
\phi:\tilde{U}\rightarrow U
\]
The map $\phi$ is a quotient map for an action of a finite group,
which in our case will be either trivial or $\mathbb{Z}_{\nu}$. A
$C^{\infty}$ $SU(2)$ \textbf{orbifold} \textbf{connection} with
respect to $(X,\varSigma)$ (or $(Y,K)$) means an oriented $\mathbb{C}^{2}$
vector bundle $E$ over $X\backslash\varSigma$ (respectively $Y\backslash K$)
with an $SU(2)$ connection $A$ having the property that the pull-back
of $(E,A)$ via any orbifold chart $\phi:\tilde{U}\rightarrow U$
extends to a smooth pair $(\tilde{E},\tilde{A})$ on $\tilde{U}$.

If we have two $SU(2)$ bifold connections $(E,A)$ and $(E',A')$,
then an \textbf{isomorphism} between them is a bundle map $\tau:E\rightarrow E'$
over $X\backslash\varSigma$ (respectively $Y\backslash K$) such
that $\tau^{*}(A')=A$. The group $\varGamma_{E,A}$ of automorphisms
of $(E,A)$ can be identified with the group of parallel sections
of the associated bundle. It is isomorphic to either the trivial group,
the circle $U(1)$ embedded as a maximal torus inside $SU(2)$, or
all of $SU(2)$.
\end{rem}

Now that we have explained the geometric setup, we will describe of
the space of connections and gauge transformations that we will be
working with, in the spirit of \cite{MR2860345}. Suppose there is
an oriented link $K$ inside an integer homology sphere $Y$. For
$\alpha\in\mathbb{Q}\cap(0,1/2)$ consider the model connection $B^{\alpha}$
described before (\ref{model connection}). We will typically use
$B$ when referring to connections defined on three-dimensional manifolds
from now on. At the same time we choose an orbifold metric $g^{\nu}$
compatible with $\alpha$ in the sense we mentioned before.

We will usually follow the convention of Kronheimer and Mrowka, and
use the $\check{}$ notation when we want to emphasize that the orbifold
metric is being used in a specific construction. Hence, $\check{L}_{m,B^{\alpha}}^{p}$
will denote the Sobolev spaces using the Levi-Civita derivative of
$g^{\nu}$ and the covariant derivative of $B^{\alpha}$ on $\mathfrak{g}_{P}$.
Taking $p=2$ and an integer\footnote{We denote this integer by $m$ instead of $k$ to avoid confusion
with the instanton number $k$.} $m>2$, the space of connections to be considered is 
\[
\mathcal{C}(Y,K,\alpha)=\left\{ B=B^{\alpha}+b\mid b\in\check{L}_{m,B^{\alpha}}^{2}(T^{*}Y\mid_{Y\backslash K}\otimes\mathfrak{g}_{P})\right\} 
\]

The gauge group $\mathcal{G}(Y,K,\alpha)$ consists of those automorphisms
of $E$ which preserve the reducibility of the model connection $B^{\alpha}$.
In fact, $\mathcal{G}(Y,K,\alpha)$ turns out to be independent of
$\alpha$ (and cone angle used) so we will just denote it as $\mathcal{G}(Y,K)$.
Here is an useful way to think about $\mathcal{G}(Y,K)$ \cite[Appendix I]{MR1308489}:

Let $\mathcal{G}^{K}(Y,K)$ denote the subgroup of $\mathcal{G}(Y,K)$
consisting of those gauge transformations which are the identity over
$K$. Let $\mathcal{G}_{K}(Y,K)$ denote the space of gauge transformations
of the line bundle $L\rightarrow K$, that is, the space of maps from
$K$ to $U(1)$. Then there is an exact sequence
\begin{equation}
1\rightarrow\mathcal{G}^{K}(Y,K)\rightarrow\mathcal{G}(Y,K)\rightarrow^{r}\mathcal{G}_{K}(Y,K)\rightarrow1\label{eq: exact sequence gauge group}
\end{equation}

In particular, we can think of $\mathcal{G}(Y,K)$ as the space of
maps $g:Y\rightarrow SU(2)$ satisfying $g(K)\subset U(1)\subset SU(2)$.
Moreover, $\mathcal{G}^{K}(Y,K)$ is weakly homotopy equivalent to
the space of smooth gauge transformations which are the identity over
a sufficiently small neighborhood of the knot $K$. From this we can
see that for knots the topological information of the map $g$ is
carried by two integers, corresponding to the degrees of $g:Y\rightarrow SU(2)$
and $g\mid_{K}:K\rightarrow U(1)$.

More generally, if the link $K$ has $r$ components, the space of
components $\pi_{0}\left(\mathcal{G}(Y,K)\right)$ is isomorphic to
\cite[Section 3]{MR2860345}
\[
\mathbb{Z}\oplus\mathbb{Z}^{r}
\]
Also, there is a preferred homomorphism 
\begin{equation}
d:\mathcal{G}(Y,K)\rightarrow\mathbb{Z}\oplus\mathbb{Z}\label{map on gauge transformations}
\end{equation}
 where the map to the second factor is obtained by taking the sum
of all the entries in the factor $\mathbb{Z}^{r}$. After introducing
Sobolev completions, we will take our gauge group to be
\[
\mathcal{G}(Y,K)=\left\{ g\mid g\in\check{L}_{m+1,B^{\alpha}}^{2}(\text{Aut}(P))\right\} 
\]
The space of connections $\mathcal{C}(Y,K,\alpha)$ is an affine space,
and on the tangent space $T_{B}\mathcal{C}(Y,K,\alpha)$ we define
an inner product (independent of $B$) by 
\[
\left\langle b,b'\right\rangle _{L^{2}}=\int_{\check{Y}}-\text{tr}(\text{ad}(*b)\wedge\text{ad}(b'))
\]
 where we are using the Killing form to contract the Lie algebra indices,
and the Hodge star on $Y$ and the wedge product to contract the form
indices. It is important to notice that \textit{the Hodge star is
the one defined by the orbifold metric $g^{\nu}$.}

The \textbf{Chern-Simons functional }on $\mathcal{C}(Y,K,\alpha)$
is defined to be the unique function 
\[
CS:\mathcal{C}(Y,K,\alpha)\rightarrow\mathbb{R}
\]
satisfying $CS(B^{\alpha})=0$ and having gradient (with respect to
the above inner product) 
\[
(\text{grad}CS)_{B}=*F_{B}
\]
From this one obtains\footnote{Notice that there is a factor of $\frac{1}{2}$ missing in the last
term of the formula in \cite[Eq. 67]{MR2860345}.} 
\[
CS(B^{\alpha}+b)=<*F_{B^{\alpha}},b>_{L^{2}}+\frac{1}{2}<*d_{B^{\alpha}}b,b>_{L^{2}}+\frac{1}{6}<*[b\wedge b],b>_{L^{2}}
\]
The critical points $B$ of $CS$ then satisfy $F_{B}=0$. When restricted
to $Y\backslash K$ , this gives rise to a representation (here $y_{0}$
is a base-point in the complement of the knot or link)
\[
\rho:\pi_{1}(Y\backslash K,y_{0})\rightarrow SU(2)
\]
satisfying the constraint that for each oriented meridian $\mu_{K}$,
$\rho(\mu_{K})$ is conjugate to 
\begin{equation}
\rho(\mu_{K})\sim\left(\begin{array}{cc}
\exp(-2\pi i\alpha) & 0\\
0 & \exp(2\pi i\alpha)
\end{array}\right)\label{behavior meridian}
\end{equation}
There is one such conjugacy class for each component of $K$, once
the components are oriented. 

As usual in the gauge theory context, we will need to perturb the
flatness equation (in fact, the Chern-Simons functional) to obtain
the necessary transversality properties for the moduli spaces. The
perturbations that we will use are described in section 3.2 of \cite{MR2860345}.
As we will see soon, the support of these perturbations must stay
away from the knot if we want to appeal to the relationship Herald
found between the count these flat connections (modulo conjugacy)
and the knot signatures \cite{MR1456309}. Kronheimer and Mrowka had
to address a similar problem in their paper on Tait colorings \cite{MR3880205},
although for different purposes. 

Before discussing this support condition, we will explain next how
the monotonicity issue arises when one tries to define a version of
Instanton Floer homology for an arbitrary value of $\alpha$.

\section{\label{sec:Monotonicity-and-Novikov}Monotonicity and Novikov Systems}

As usual, we would like to do ``Morse theory'' on 
\[
\mathcal{B}(Y,K,\alpha)=\mathcal{C}(Y,K,\alpha)/\mathcal{G}(Y,K)
\]
Notice that since $\mathcal{G}(Y,K)$ is independent of $\alpha,\nu$
and $\mathcal{C}(Y,K,\alpha)$ is contractible, the homotopy type
of the space $\mathcal{B}(Y,K,\alpha)$ is independent of $\alpha,\nu$.
As in the case when $K$ is not present, the Chern-Simons functional
is no longer single-valued on $\mathcal{B}(Y,K,\alpha)$. In fact
$CS$ is invariant only under the identity component of the gauge
group. If $B\in\mathcal{C}(Y,K,\alpha)$ is a connection and $g$
a gauge transformation, we can use the latter to form the bundle $S^{1}\times_{g}E$
over $S^{1}\times Y$ together with its reduction over $S^{1}\times K$,
defined by $B^{\alpha}$. Then the map $d$ from (\ref{map on gauge transformations})
becomes 
\[
d(g)=(k,l)\in\mathbb{Z}\oplus\mathbb{Z}
\]
where $k$ and $l$ are the instanton and monopole numbers respectively.
In this case we have \cite[p. 874]{MR2860345}
\begin{equation}
CS(B)-CS(g(B))=32\pi^{2}(k+2\alpha l)\label{drop in Chern Simons loop}
\end{equation}
 More generally, we define for a path $\gamma:[0,1]\rightarrow\mathcal{C}(Y,K,\alpha)$
the \textbf{topological energy }as twice the drop in the Chern Simons
functional, so that the last equation implies that a path from $B$
to $g(B)$ has topological energy 
\begin{equation}
\mathcal{E}(k,l)=64\pi^{2}(k+2\alpha l)\label{topological energy}
\end{equation}
For a path that formally solves the downward gradient-flow equation
for $CS$ on $\mathcal{C}(Y,K,\alpha)$, the topological energy coincides
with the modified path energy 
\[
\int_{0}^{1}(\|\dot{\gamma}(t)\|^{2}+\|\text{grad}CS(\gamma(t))\|^{2})dt
\]

Before comparing $\mathcal{E}(k,l)$ with the formula for the grading
of two flat connections, we discuss the perturbations needed to achieve
transversality, which involve using \textbf{holonomy perturbations}.
These will be the same as in \cite[Section 3.2]{MR2860345}, and we
will give more details about them in the next section, where we will
need to keep track on certain estimates involving their norms. The
basic idea is that each holonomy perturbation gives rise to a \textbf{cylinder
function 
\[
f:\mathcal{C}(Y,K,\alpha)\rightarrow\mathbb{R}
\]
}which depends on an $l$-tuple 
\[
\mathbf{q}=(q_{1},\cdots,q_{l})
\]
 of immersions $q_{i}:S^{1}\times D^{2}\rightarrow Y\backslash K$,
$i=1,\cdots,l$ and an $SU(2)$-invariant function $h$ on $SU(2)^{\times l}$
via the formula 
\[
f(B)=\int_{D^{2}}h(\text{Hol}_{\mathbf{q}}(B))\mu
\]
Here $\mu$ is a non-negative 2-form supported in the interior of
$D^{2}$ and having integral $1$. We assume that $h$ is invariant
under the action of $SU(2)$ on each of the $l$ factors separately
(see the remark after definition 3.3 in \cite{MR2860345}). It follows
that \textit{each cylinder function is invariant under the full gauge
group, not just the identity component.}

The space of perturbations one uses to achieve transversality involves
taking a countable collection of cylinder functions with an $l^{1}$
notion of convergence \cite[Definition 3.6]{MR2860345}. If $\mathcal{P}$
denotes the Banach space of perturbations, then $\mathfrak{p}\in\mathcal{P}$
induces a function $f_{\mathfrak{p}}$, and the perturbed Chern-Simons
functional we consider is 
\[
CS_{\mathfrak{p}}=CS+f_{\mathfrak{p}}
\]
The critical points are now connections satisfying the equation 
\[
*F_{B}+V_{\mathfrak{p}}(B)=0
\]
where $V_{\mathfrak{p}}$ is the formal gradient of $f_{\mathfrak{p}}$
with respect to the $L^{2}$ inner product. Proposition 3.10 in \cite{MR2860345}
states that there is a residual subset of $\mathcal{P}$ such that,
for all $\mathfrak{p}$ all the irreducible critical points of $CS+f_{\mathfrak{p}}$
are non-degenerate (we will return to the reducibles momentarily).
The \textbf{perturbed topological energy} is then defined for a path
$\gamma:[0,1]\rightarrow\mathcal{C}(Y,K,\alpha)$ as
\begin{equation}
\mathcal{E}_{\mathfrak{p}}(\gamma)=2\left((CS+f_{\mathfrak{p}})(B(t_{0}))-(CS+f_{\mathfrak{p}})(B(t_{1}))\right)\label{perturbed topological energy}
\end{equation}
 Notice that because the perturbations are invariant under the full
gauge group, for a path from $B$ to $g(B)$, the term $f_{\mathfrak{p}}(B)-f_{\mathfrak{p}}(g(B))$
will cancel, \textit{so the perturbed topological energy $\mathcal{E}_{\mathfrak{p}}$
on loops has the same expression as the unperturbed case, namely,
equation (\ref{topological energy})}.

Returning to the grading question, define the (perturbed) \textbf{Hessian
}of a connection $B\in\mathcal{C}(Y,K,\alpha)$ as \textbf{
\begin{equation}
\text{Hess}_{B,\mathfrak{p}}(b)=*d_{B}b+DV\mid_{B}(b)\label{Hessian}
\end{equation}
}where $b\in T_{B}\mathcal{C}(Y,K,\alpha)$ and the (perturbed) \textbf{extended
Hessian }\cite[p. 881]{MR2860345}\textbf{
\begin{equation}
\widehat{\text{Hess}}_{B,\mathfrak{p}}=\left(\begin{array}{cc}
0 & -d_{B}^{*}\\
-d_{B} & \text{Hess}_{B,\mathfrak{p}}
\end{array}\right):\check{L}_{j}^{2}(Y;\mathfrak{g}_{P})\oplus\mathcal{T}_{j}\rightarrow\check{L}_{j-1}^{2}(Y;\mathfrak{g}_{P})\oplus\mathcal{T}_{j-1}\label{extended Hessian}
\end{equation}
}To clarify the notation, here $j\leq m$ is an integer, $\check{L}_{j}^{2}(Y;\mathfrak{g}_{P})$
is a shorthand for $\check{L}_{j,B^{\alpha}}^{2}(Y\backslash K;\mathfrak{g}_{P}\mid_{Y\backslash K})$,
and $\mathcal{T}_{j}$ is shorthand for sections of $T_{B}\mathcal{C}(Y,K,\alpha)$
with regularity Sobolev $\check{L}_{j-1,B^{\alpha}}^{2}$ .

For two irreducible, non-degenerate (perturbed) flat connections $B_{0},B_{1}\in\mathcal{C}(Y,K,\alpha)$,
the \textbf{relative grading}
\[
\text{gr}(B_{0},B_{1})\in\mathbb{Z}
\]
will be defined by taking a path $B(t)$ in $\mathcal{C}(Y,K,\alpha)$
from $B_{0}$ to $B_{1}$ and letting $\text{gr}(B_{0},B_{1})$ be
equal to the spectral flow of the 1-parameter family $\widehat{\text{Hess}}_{B(t),\mathfrak{p}}$.
This number only depends on the endpoints and not the path since $\mathcal{C}(Y,K,\alpha)$
is contractible. If 
\[
[\beta_{0}]=[B_{0}],\;\;\;\;\;[\beta_{1}]=[B_{1}]
\]
are the corresponding gauge equivalence classes in $\mathcal{B}=\mathcal{B}(Y,K,\alpha)$,
then $[B(t)]$ determines a path $\zeta$ from $[\beta_{0}]$ to $[\beta_{1}]$,
defining a relative homotopy class $z\in\pi_{1}([\beta_{0}],\mathcal{B},[\beta_{1}])$.
This relative homotopy class only depends on $B_{0},B_{1}$. Conversely,
for a relative homotopy class of paths from $[\beta_{0}]$ to $[\beta_{1}]$,
we define 
\[
\text{gr}_{z}([\beta_{0}],[\beta_{1}])\in\mathbb{Z}
\]
 via the above procedure.

For the case of a closed loop $z$ based at a point $[\beta]\in\mathcal{B}(Y,K,\varPhi)$,
we obtain an element inside $\pi_{1}(\mathcal{B}(Y,K,\alpha))\simeq\pi_{0}(\mathcal{G}(Y,K,\alpha))$.
Lemma 3.14 in \cite{MR2860345} shows that 
\begin{equation}
\text{gr}_{z}([\beta],[\beta])=8k+4l\label{formula grading}
\end{equation}
where we think of $z$ as being specified by a gauge transformation
$g$ as before, and $d(g)=(k,l)$.

From this grading formula we can see that $\text{gr}_{z}([\beta],[\beta])\equiv0\mod4$,
which means that \textit{the Floer homology we will define is $\mathbb{Z}/4\mathbb{Z}$
graded}. We will eventually show how to promote this relative grading
to an absolute grading. More importantly, comparing the energy and
grading formulas on loops we can see that 
\begin{equation}
\mathcal{E}(k,l)=8\pi^{2}\text{gr}_{z}([\beta],[\beta])+32\pi^{2}l(4\alpha-1)\label{relationship grading energy}
\end{equation}
which means that \textit{the only way for $\mathcal{E}(k,l)$ and
$\text{gr}_{z}([\beta],[\beta])$ to be proportional regardless of
the value of $(k,l)$ is if $\alpha=\frac{1}{4}$}.

To explain what happens when $\alpha\neq1/4$, we recall Proposition
3.23 in \cite{MR2860345}. Let $\mathcal{M}_{z}([\beta_{0}],[\beta_{1}])$
denote a connected component of the moduli space of $ASD$ connections
on $\mathbb{R}\times Y$ limiting to $[\beta_{0}]$ as $t\rightarrow-\infty$
and $[\beta_{1}]$ as $t\rightarrow\infty$ (with suitable perturbations
thrown into the picture to guarantee transversality).
\begin{prop}
\label{"Gromov" compactness}(\cite[Proposition 3.23]{MR2860345})
Given any $C>0$, there are only finitely many $[\beta_{0}],[\beta_{1}]$
and $z$ for which the moduli space $\mathcal{M}_{z}([\beta_{0}],[\beta_{1}])$
is non-empty and has topological energy $\mathcal{E}_{\mathfrak{p}}(\gamma)$
at most $C$.
\end{prop}

Now suppose one wants to define the differential via the usual formula
\[
\partial[\beta_{0}]=\sum_{[\beta_{1}]\mid\text{gr}([\beta_{0}],[\beta_{1}])=1}\sum_{z}n_{z}([\beta_{0}],[\beta_{1}])[\beta_{1}]
\]
 where we are taking the sum over all moduli spaces $\mathcal{M}_{z}([\beta_{0}],[\beta_{1}])$
which are one-dimensional and $n_{z}([\beta_{0}],[\beta_{1}])\in\mathbb{Z}$
denotes the number (with orientations included) of trajectories inside
the zero-dimensional (compact) space $\check{\mathcal{M}}_{z}([\beta_{0}],[\beta_{1}])=\mathcal{M}_{z}([\beta_{0}],[\beta_{1}])/\mathbb{R}$.
The issue is that for fixed $[\beta_{1}]$, the sum
\[
\sum_{z}n_{z}([\beta_{0}],[\beta_{1}])[\beta_{1}]
\]
 could be infinite! To see why, suppose that $z,z'$ represent two
trajectories from $[\beta_{0}]$ to $[\beta_{1}]$ with $\mathcal{M}_{z}([\beta_{0}],[\beta_{1}])$
and $\mathcal{M}_{z'}([\beta_{0}],[\beta_{1}])$ one dimensional.
Then $z^{-1}\circ z'$ represents a loop based at $[\beta_{0}]$ with
$\text{gr}_{z^{-1}\circ z'}([\beta_{0}],[\beta_{0}])=0$. From (\ref{relationship grading energy})
we see that the topological energy of this loop must be $32\pi^{2}l(4\alpha-1)$,
which is a priori unbounded since $\alpha\neq1/4$ and the monopole
number $l$ can in principle be any integer. Since $\mathcal{E}_{\mathfrak{p}}(z)-\mathcal{E}_{\mathfrak{p}}(z')$
differ (up to sign) by $32\pi^{2}l(4\alpha-1)$, we see there is no
way to guarantee the assumptions in Proposition (\ref{"Gromov" compactness}).

The way to get out of this conundrum is via a\textbf{ Novikov system},
but first we recall a few facts about the construction of homology
groups using \textbf{local coefficients}. A good reference for the
general construction is \cite[Chapter 5]{MR1841974}. We will use
$\varGamma$ to denote a local systems of coefficients. Also, if $[\beta]\in\mathcal{B}(Y,K,\alpha)$
, then $o([\beta])$ will denote\footnote{The notation Kronheimer and Mrowka use for the system of orientations
is $\varLambda([\beta])$ but we will use $\varLambda$ for the Novikov
field so a sacrifice had to be made somewhere. } the 2-element of orientations for $[\beta]$ \cite[Section 3.6]{MR2860345}.

Suppose that to each $[\beta]$ we assign an abelian group $\varGamma_{[\beta]}$
and for each homotopy class of paths $z$ from $[\beta_{0}]$ to $[\beta_{1}]$
there is an isomorphism $\varGamma_{z}$ from $\varGamma_{[\beta_{0}]}$
to $\varGamma_{[\beta_{1}]}$ satisfying the usual composition law
for paths.

Define the chain group $C_{*}(Y,K,\alpha;\varGamma)$ generated by
irreducible critical points $[\beta]$ of $CS_{\mathfrak{p}}$ as
\[
C_{*}(Y,K,\alpha;\varGamma)=\bigoplus_{[\beta]\in\mathfrak{C}}\mathbb{Z}o([\beta])\otimes\varGamma_{[\beta]}
\]
and the boundary map as 
\begin{equation}
\partial=\sum_{([\alpha],[\beta],z)}\sum_{[\check{A}]}\epsilon[\check{A}]\otimes\varGamma_{z}\label{boundary local coefficients}
\end{equation}
 where $[A]\in\mathcal{M}_{z}([\beta_{0}],[\beta_{1}])$ belongs to
a 1-dimensional space, $[\check{A}]$ denotes the connection modulo
the $\mathbb{R}$ action, and 
\[
\epsilon[\check{A}]:o[\beta_{0}]\rightarrow o[\beta_{1}]
\]
is an isomorphism obtained by comparing orientations. Written differently,
we can think of the differential as
\[
\partial(e)=\sum_{[\beta_{1}]\in\mathfrak{C}}\sum_{z}n_{z}([\beta_{0}],[\beta_{1}])\varGamma_{z}(e),\;\;\;e\in\varGamma_{[\beta_{0}]}
\]
which we will occasionally write as 
\[
\partial[\beta_{0}]=\sum_{[\beta_{1}]\in\mathfrak{C}}\sum_{z}n_{z}([\beta_{0}],[\beta_{1}])\varGamma_{z}[\beta_{1}]
\]
 After choosing trivializations for $o([\beta_{0}])$ and $o([\beta_{1}])$,
the contribution for a given pair of critical points in the differential
takes the form 
\[
\sum_{z}n_{z}\varGamma_{z}
\]
 where $z$ runs through all relative homotopy classes satisfying
the conditions 
\[
\text{gr}_{z}([\beta_{0}],[\beta_{1}])=1
\]
Now, define the \textbf{support }
\[
\text{sup}([\beta_{0}],[\beta_{1}])=\{z\mid n_{z}\neq0\}
\]
By Proposition (\ref{"Gromov" compactness}), for all $C$, the intersection
\[
\text{sup}([\beta_{0}],[\beta_{1}])\cap\{z\mid\mathcal{E}_{\mathfrak{p}}(z)\leq C\}
\]
 is finite. In the notation of definition 30.2.2 in \cite{MR2388043},
we will define a $c$-complete local system of coefficients which
will allow us to make sense of infinite sums like $\sum_{z}n_{z}\varGamma_{z}$. 

A definition of the $c$-complete condition can be found in our annotated
version \cite[Definition 18]{Echeverria[DraftFuruta]}, as well as
some additional details. For our purposes we will be content with
describing the $c$-complete local system we will be using throughout
the paper, based on \cite[Section 30.2]{MR2388043}, \cite[Section 4]{MR1362838},
\cite[Section 2.2]{MR2199540}, \cite{MR2474322} and \cite[Section 2.1]{MR3590354}.

In fact, there are a couple of natural $c$-complete systems one could
use, but in order to minimize the number of auxiliary choices we have
to make, we will stick with a particular local system which we will
call the \textbf{universal Novikov field/local system }\cite[Section 4]{MR3590354}.
Moreover, the universal Novikov field/local system has the advantage
that the fibers $\varGamma_{[\beta]}$ assigned to each $[\beta]$
is independent of $\alpha$ (which does not happen for the other candidates),
so it might be a more suited candidate if one wanted to compare the
Floer homologies corresponding to different values of $\alpha$ (though
$\varGamma_{z}$ will depend on $\alpha$ which is to be expected
as we will see in a moment). 

First of all, if $\mathbb{F}$ is a ground field and $\varGamma\leq\mathbb{R}$
an additive subgroup, then the \textbf{Novikov field }$\varLambda^{\mathbb{F},\varLambda}$
associated to $\mathbb{F},\varLambda$ is 
\[
\varLambda^{\mathbb{F},\varGamma}=\left\{ \sum_{r\in\varGamma}a_{r}T^{r}\mid a_{r}\in\mathbb{F}\text{ and }\#\{r\mid a_{r}\neq0,r>C\}<\infty\text{ for all }C\in\mathbb{R}\right\} 
\]
In other words, we allow infinitely many terms in the negative direction.
The example we will have in mind most of the time is $\mathbb{F}=\mathbb{Q}$
or $\mathbb{F}=\mathbb{C}$. 

Now write $\pi_{1}(\mathcal{B})=\pi_{1}(\mathcal{B}(Y,K,\alpha))=\mathbb{Z}\oplus\mathbb{Z}$,
and consider the universal covering $\tilde{\mathcal{B}}(Y,K,\alpha)$
of $\mathcal{B}(Y,K,\alpha)$. This is the space we obtain after taking
the quotient of $\mathcal{C}(Y,K,\alpha)$ using only gauge transformations
$g\in\mathcal{G}(Y,K)$ which satisfy $d(g)=(0,0)$ (so they are in
the connected component of the identity gauge transformation). The
Chern-Simons functional $CS$ becomes real valued on $\tilde{\mathcal{B}}(Y,K,\alpha)$
so we can regard it as a map 
\[
CS:\tilde{\mathcal{B}}(Y,K,\alpha)\rightarrow\mathbb{R}
\]
In this context $\tilde{\beta}\in\tilde{\mathcal{B}}(Y,K,\alpha)$
will denote a lift of $[\beta]\in\mathcal{B}(Y,K,\alpha)$. On $\pi_{1}(\mathcal{B}(Y,K,\alpha))$
we can also define the \textbf{period homomorphism}

\textbf{
\begin{align*}
\triangle_{CS}:\pi_{1}(\mathcal{B}(Y,K,\alpha))\rightarrow\mathbb{R}\\
g\rightarrow64\pi^{2}(k+2\alpha l) &  & d(g)=(k,l)
\end{align*}
}as well as the \textbf{spectral flow }map 
\begin{align*}
\text{sf}:\pi_{1}(\mathcal{B}(Y,K,\alpha))\rightarrow\mathbb{Z}\\
z\rightarrow\text{sf}(z)
\end{align*}
whose kernel is the \textbf{annihilator} 
\[
\text{Ann}=\ker\text{sf}\subset\pi_{1}(\mathcal{B}(Y,K,\alpha))
\]
We also define the additive subgroup of $\mathbb{R}$
\[
I=\text{im}\triangle_{CS}\subset\mathbb{R}
\]
The different local systems we will now explain depend on how they
are related to $I$, and each has their own set of advantages and
disadvantages.

\textbf{$\bullet$ Minimal Novikov field/Local System: }here we work
with Novikov field $\varLambda^{\mathbb{F},I'}$, where $I'$ is defined
as
\[
I'=\triangle_{CS}(\text{Ann})
\]
Recalling the formula \ref{formula grading} for the grading, this
means that we must evaluate $\triangle_{CS}$ on those loops for which
\[
8k+4l=0
\]
so in fact
\[
I'=\{64\pi^{2}k(1-4\alpha)\mid k\in\mathbb{Z}\}
\]
The appealing feature of $\varLambda^{\mathbb{F},I'}$ is that when
$\alpha=1/4$ (the monotone case), $I'$ collapses to $\{0\}$, in
which case working over $\varLambda^{\mathbb{F},I'}$ is the same
as working over $\mathbb{F}$, as should be the case in the monotone
situation. The local system associated to $\varLambda^{\mathbb{F},I'}$
will assign to each configuration $[\beta]$ a copy of $\varLambda^{\mathbb{F},I'}$,
i.e, $\varGamma_{[\beta]}=\varLambda^{\mathbb{F},I'}$. To specify
what happens on paths, i.e, what the maps $\varGamma_{z}:\varGamma_{[\beta_{0}]}\rightarrow\varGamma_{[\beta_{1}]}$
should be, observe that since $\triangle_{CS}:\pi_{1}(\mathcal{B}(Y,K,\alpha))\rightarrow\mathbb{R}$
is a homomorphism whose domain is a finitely generated abelian group
and whose image is torsion free, there is an exact sequence 
\[
\ker\triangle_{CS}\rightarrowtail\pi_{1}(\mathcal{B}(Y,K,\alpha))\twoheadrightarrow I
\]
which splits, so $\pi_{1}(\mathcal{B}(Y,K,\alpha))$ can be identified
with $\ker\triangle_{CS}\oplus I$, i.e, 
\begin{equation}
\pi_{1}(\mathcal{B}(Y,K,\alpha))\simeq\ker\triangle_{CS}\oplus I\label{eq:identification 1}
\end{equation}
 By the same token, there is an exact sequence 
\[
\ker\text{sf}\rightarrowtail\pi_{1}(\mathcal{B}(Y,K,\alpha))\twoheadrightarrow4\mathbb{Z}
\]
which splits, so $\pi_{1}(\mathcal{B}(Y,K,\alpha))$ can be identified
with $\text{Ann}\oplus4\mathbb{Z}$, i.e
\begin{equation}
\pi_{1}(\mathcal{B}(Y,K,\alpha))\simeq\text{Ann}\oplus4\mathbb{Z}\label{identification 2}
\end{equation}
 This allows us to construction a projection 
\[
p:I\rightarrow I'
\]
as follows. Regard $I$ as a subset of $\pi_{1}(\mathcal{B}(Y,K,\alpha))$
under the first identification \ref{eq:identification 1}. Then using
the second identification \ref{identification 2}, project $I$ onto
$\text{Ann}$ and let $p$ be $\triangle_{CS}$ of this element. In
other words 
\[
p(i)=\triangle_{CS}(\pi_{\text{Ann}}(i))
\]
Now, we can define $\varGamma_{z}$ on \textbf{loops} as multiplication
by $T^{p(-\mathcal{E}_{top}(z))}$. Fixing the map $\varGamma_{z}$
on loops is enough to pin down the local system up to isomorphism.
The ``issue'' with this option is that it is not obvious how the
resulting Floer groups depend on the choice of projection $p$. In
fact, we were unable to write \textit{explicitly} a projection $p$,
so if one ever wanted to compute explicitly the resulting homology
groups, it is not clear what to do in this circumstance.

\textbf{$\bullet$ Strongly $c$-complete Novikov field/Local System:
}we call this intermediate case strongly $c$-complete in analogy
with the examples given by Kronheimer and Mrowka in \cite[Section 30.2]{MR2388043}.
In this case we use $\varLambda^{\mathbb{F},I}$ (so $I$ instead
of $I'$) as the Novikov system and a local system which assigns to
$[\beta]$ a copy of $\varLambda^{\mathbb{F},I}$, i.e, $\varGamma_{[\beta]}=\varLambda^{\mathbb{F},I}$.
On loops $\varGamma_{z}$ will act as multiplication by $T^{-\mathcal{E}_{top}(z)}$.
Notice that this system has the advantage that we no longer need to
choose a projection $p$. However, in order to specify actual groups
(and not groups up to isomorphisms), further choices are needed, since
knowing how $\varGamma_{z}$ is defined on loops is not enough to
specify a unique formula for how it should act on paths. More precisely,
for a path $z:[\beta_{0}]\rightarrow[\beta_{1}]$ we would like $\varGamma_{z}:\varGamma_{[\beta_{0}]}\rightarrow\varGamma_{[\beta_{1}]}$
to be multiplication by $T^{-\mathcal{E}_{top}(z)}$. However, notice
that even for a simple element inside $\varGamma_{[\beta_{0}]}$ like
$T^{\triangle_{CS}(g)}$ , for $g\in\pi_{1}([\beta_{0}],\mathcal{B}(Y,K,\alpha),[\beta_{0}])$
, there is no reason why we must have that $\varGamma_{z}T^{\triangle_{CS}(g)}=T^{-\mathcal{E}_{top}(z)+\triangle_{CS}(g)}$
can be written as $T^{\triangle_{CS}(g')}$ for some $g'\in\pi_{1}([\beta_{1}],\mathcal{B}(Y,K,\alpha),[\beta_{1}])$.
However, with additional choices we make sense of this formula. Namely,
suppose that we choose a preferred lift $\tilde{\beta}$ for each
$[\beta]\in\mathcal{B}(Y,K,\alpha)$. Then from formula (\ref{drop in Chern Simons loop})
we have that $\triangle_{CS}(g)=2\left(CS(\tilde{\beta})-CS(g\tilde{\beta})\right)$,
which means that an element inside $\varGamma_{[\beta_{0}]}$ can
be rewritten as a formal power series 
\[
\sum_{r=\triangle_{CS}(g)}c_{r}T^{\triangle_{CS}(g)}=T^{2CS(\tilde{\beta}_{0})}\sum_{r=\triangle_{CS}(g)}c_{r}T^{-2CS(g\tilde{\beta}_{0})}
\]
Then $\varGamma_{z}=\varGamma_{z,\tilde{\beta}_{0},\tilde{\beta}_{1}}$
can be taken to be 
\[
\varGamma_{z,\tilde{\beta}_{0},\tilde{\beta}_{1}}\left(\sum_{r=\triangle_{CS}(g)}c_{r}T^{\triangle_{CS}(g)}\right)\equiv T^{2CS(\tilde{\beta}_{1})}\sum_{r=\triangle_{CS}(g)}c_{r}T^{-2CS(z\cdot g\cdot\tilde{\beta}_{0})}
\]
Notice that we added the ``basepoints'' in our notation to make
it clear that an additional choice is needed. Here $z\cdot g\cdot\tilde{\beta}_{0}$
means the following: regard $g\cdot\tilde{\beta}_{0}$ as an element
in $\tilde{\mathcal{B}}(Y,K,\alpha)$, projecting to $[\beta_{0}]\in\mathcal{B}(Y,K,\alpha)$,
and use the path $z$ to determine a (unique) element $z\cdot g\cdot\tilde{\beta}_{0}$,
projecting to $[\beta_{1}]\in\mathcal{B}(Y,K,\alpha)$. Moreover,
this formula uses implicitly the fact that if $r=\triangle_{CS}(g)=\triangle_{CS}(g')$
, i.e, $CS(g'\tilde{\beta})=CS(g\tilde{\beta})$ or equivalently $g'\circ g^{-1}\in\ker\triangle_{CS}$,
then $CS(z\cdot g\cdot\tilde{\beta}_{0})=CS(z\cdot g'\cdot\tilde{\beta}_{0})$.
In any case, these extra choices of basepoints will not make our lives
any easier and do not gives us much of an advantage, so we will prefer
to use the universal Novikov/ local system, which we will describe
next. 

\textbf{$\bullet$ Universal Novikov field/Local System}: as the name
suggests, here we are working with $\varLambda^{\mathbb{F},\mathbb{R}}$.
The local system assigns $[\beta]$ a copy of $\varLambda^{\mathbb{F},\mathbb{R}}$,
i.e, $\varGamma_{[\beta]}=\varLambda^{\mathbb{F},\mathbb{R}}$. In
order to define $\varGamma_{z}$ for a path $z:[\beta_{0}]\rightarrow[\beta_{1}]$,
we lift $z$ to a path $\tilde{z}:\tilde{\beta}_{0}\rightarrow\tilde{\beta}_{1}$
on the universal cover $\tilde{\mathcal{B}}(Y,K,\alpha)$ (where again
$CS$ is real valued), and define 
\[
\varGamma_{z}\equiv\text{multiplication by }T^{-2(CS(\tilde{\beta}_{0})-CS(\tilde{\beta}_{1}))}
\]
 The factor of $2$ in the exponent is chosen so that on loops it
acts as multiplication by $T^{-\mathcal{E}_{top}(z)}$. Because of
the gauge invariance of the perturbation $\mathfrak{p}$ we could
also write this as 
\[
T^{-2(CS_{\mathfrak{p}}(\tilde{\beta}_{0})-CS_{\mathfrak{p}}(\tilde{\beta}_{1}))}T^{-2(f_{\mathfrak{p}}([\beta_{1}])-f_{\mathfrak{p}}([\beta_{0}]))}
\]
Notice that the second term is independent of the lift of $[\beta_{0}],[\beta_{1}]$.
In a way, it would have been more tempting to let $\varGamma_{z}$
depend on the perturbation being used and define $\varGamma_{z,\mathfrak{p}}$
as multiplication by $T^{-2(CS_{\mathfrak{p}}(\tilde{\beta}_{0})-CS_{\mathfrak{p}}(\tilde{\beta}_{1}))}$.
This would give the advantage that for gradient flow-lines the exponent
$-2(CS_{\mathfrak{p}}(\tilde{\beta}_{0})-CS_{\mathfrak{p}}(\tilde{\beta}_{1}))$
would always be negative, hence making automatic the condition that
the differential is well defined on the chain complex. However, once
we analyze the case of cobordisms, it is more tricky to figure out
what the right notion of perturbed topological energy is, given that
we have to use many different kinds of holonomy perturbations, so
this monotonicity property would not be automatic anyways. In other
words, there is no way to avoid the fact that some control on the
perturbations has to be imposed, and this is precisely what we will
do in the next section (based on ideas from \cite{Daemi[2019]}).\textbf{}

\textbf{ }

\section{\label{sec:Floer-Novikov-Homology-and}Floer-Novikov Homology and
its Relationship To the Knot Signature}

Recall that for the construction of the usual instanton Floer homology
groups on an integer homology sphere $Y$ \cite{MR956166}, the chain
complex $CI(Y)$ is generated by (perturbed) irreducible flat connections
and the proof that the differential $\partial$ on $CI(Y)$ squares
to zero requires showing that no broken flow lines can factorize through
the trivial flat connection $\theta$, which cannot be eliminated
through the use of perturbations.

Moreover, the proof that the homology $HI(Y)$ one obtains does not
depend on the choice of metric and perturbations, requires using the
fact that as we deform this data, we do not meet reducibles, since
once again they are singular points in the space of connections mod
gauge. 

To show these properties one uses the fact that the stabilizer of
the trivial connection $\theta$ is positive dimensional and that
the trivial connection is isolated from the irreducible (perturbed)
flat connections. Even if there were no need for using local coefficients
suitable analogues of these results are needed in our situation, which
is what we will proceed to discuss next.

Recall that the critical points $B$ of the unperturbed Chern-Simons
functional $CS$ on $(Y,K,\alpha)$ can be interpreted as flat connections
on $Y\backslash K$, which modulo gauge will correspond to conjugacy
classes of representations 
\[
\rho:\pi_{1}(Y\backslash K,y_{0})\rightarrow SU(2)
\]
Since $H_{1}(Y\backslash K;\mathbb{Z})\simeq\mathbb{Z}[\mu_{K}]$
regardless of the pair $(Y,K)$, the abelian representations of $\pi_{1}(Y\backslash K,y_{0})$
are completely determined by their action on an oriented meridian
$\mu_{K}$ of the knot $K$, which we already required to be conjugate
to the matrix 
\[
\rho(\mu_{K})\sim\left(\begin{array}{cc}
\exp(-2\pi i\alpha) & 0\\
0 & \exp(2\pi i\alpha)
\end{array}\right)
\]
 In other words,\textbf{ }\textit{for any fixed choice of $\alpha$,
there is only one reducible critical point of the unperturbed Chern-Simons
functional $CS$,} which we denoted previously as $[\theta_{\alpha}]$.

In general, a representation $\rho:\pi_{1}(Y\backslash K,y_{0})\rightarrow SU(2)$
will determine a local coefficient system $\mathfrak{g}_{\rho}$ on
$Y\backslash K$, with fiber $\mathfrak{g}=\mathfrak{su}(2)$. In
turn this gives rise to cohomology groups $H^{i}(Y\backslash K;\mathfrak{g}_{\rho})$.
If we identify $\rho$ with a flat connection $B$, then its gauge
orbit $[B]$ being isolated among the set of critical points $\mathfrak{C}$
of $CS$ is equivalent to $B$ being \textbf{non-degenerate, }that
is, the kernel of the map 
\begin{equation}
\ker:H^{1}(Y\backslash K;\mathfrak{g}_{\rho})\rightarrow H^{1}(\mu_{K};\mathfrak{g}_{\rho})\label{non-degeneracy kernel}
\end{equation}
 is zero, where here $\mu_{K}$ is a collection of loops representing
the meridians of all the components of $K$ (which will be one since
we are focusing on the case of a knot).

Here is an explanation of this condition as well of the notation.
The map (\ref{non-degeneracy kernel}) is simply the one induced in
cohomology by the inclusion of $\mu_{K}\hookrightarrow Y\backslash K$.
When $B$ is an \textit{irreducible} flat connection, Lemma 3.13 in
\cite{MR2860345} shows that this condition is equivalent to the extended
Hessian $\widehat{\text{Hess}}_{B}$ defined in (\ref{extended Hessian})
being invertible.

At the \textit{reducible} flat connection $\theta_{\alpha}$, one
can still use (\ref{non-degeneracy kernel}) as the criterion for
being isolated, the only thing that changes is that this condition
does not imply that the extended Hessian $\widehat{\text{Hess}}_{\theta_{\alpha}}$
is invertible, since at reducibles connections it will always have
a non-trivial kernel \cite[Section 2.5.4]{MR1883043}. To better understand
the requirement (\ref{non-degeneracy kernel}) at the reducible $\theta_{\alpha}$,
notice that the Lie algebra $\mathfrak{g}_{\theta_{\alpha}}\simeq\mathfrak{su}(2)\simeq\mathfrak{so}(3)$
decomposes as
\[
\mathfrak{g}_{\theta_{\alpha}}\simeq\mathbb{R}\oplus L_{\alpha}^{\otimes2}
\]
 where $E=L_{\alpha}\oplus L_{\alpha}^{-1}$ is the decomposition
of the $SU(2)$ bundle induced by $\theta_{\alpha}$. Hence 
\[
H^{1}\left(Y\backslash K;\mathfrak{g}_{\theta_{\alpha}}\right)\simeq H^{1}(Y\backslash K;\mathbb{R})\oplus H^{1}(Y\backslash K;L_{\alpha}^{\otimes2})=\mathbb{R}\oplus H^{1}(Y\backslash K;L_{\alpha}^{\otimes2})
\]
One should think of the $\mathbb{R}$ summand as being the directions
in the Zariski tangent space obtained from deforming the value of
the holonomy $\alpha$. Given that $\alpha$ is fixed for our problem
and $H^{1}(Y\backslash K;\mathbb{R})\rightarrow H^{1}(\mu_{K};\mathfrak{g}_{\rho})$
is an injection since $H^{1}(\mu_{K};\mathfrak{g}_{\rho})\simeq H^{1}(\mu_{K};\mathbb{R})$
\cite[Lemma 63]{Echeverria[DraftFuruta]}, the condition (\ref{non-degeneracy kernel})
captures that we only need to worry about about the factor $H^{1}(Y\backslash K;L_{\alpha}^{\otimes2})$.
As we mentioned in the introduction, the condition (\ref{non-degeneracy kernel})
is only satisfied for certain values of $\alpha$ determined by the
Alexander polynomial $\triangle_{K}(t)$ of $K$.

\begin{lem}
\label{Alexander Knot isolated}Suppose that $K$ is a knot and that
$\triangle_{K}(e^{-4\pi i\alpha})\neq0$. Then the reducible connection
$[\theta_{\alpha}]$ is isolated from the irreducible flat connections.
\end{lem}

\begin{proof}
Again, it suffices to guarantee that $H^{1}(Y\backslash K;L_{\alpha}^{\otimes2})$
vanishes. In the notation of \cite[Appendix]{Echeverria[DraftFuruta]},
the reducible connection $[\theta_{\alpha}]$ defines a local coefficient
system with fiber $\mathbb{C}$ where the monodromy map $\hat{\rho}_{\alpha}:H_{1}(Y\backslash K;\mathbb{Z})\rightarrow\mathbb{C}^{*}$
maps the meridian $\mu_{K}$ to $e^{-4\pi i\alpha}$. The vanishing
of the cohomology $H^{1}(Y\backslash K;L_{\alpha}^{\otimes2})$ now
follows from \cite[Corollary 65]{Echeverria[DraftFuruta]}.

\end{proof}
For completeness sake, we will now give more details to Lemma 3.13
in \cite{MR2860345}, since our method of proof will generalize to
the case of an embedded torus $T$ inside a four manifold $X$ with
the homology of $S^{1}\times S^{3}$.  
\begin{lem}
\label{orbifold cohomology}Suppose that $\rho$ is a flat connection
corresponding to a critical point of $CS$ and $\mathfrak{g}_{\rho}$
is the corresponding local system. Then the first (orbifold) cohomology
$\check{H}^{1}(\check{Y};\mathfrak{g}_{\rho})$ can be identified
with $\ker:H^{1}(Y\backslash K;\mathfrak{g}_{\rho})\rightarrow H^{1}(\mu_{K};\mathfrak{g}_{\rho})$.
\end{lem}

\begin{proof}
For the proof we will follow the suggestion of Lemma 3.13 in \cite{MR2860345}
and use the Mayer-Vietoris sequence for cohomology with local coefficients.
More precisely, for $\epsilon>0$ write $\check{Y}$ as 
\[
\check{Y}=(\check{Y}\backslash\check{\nu}_{\epsilon}(K))\cup\check{\nu}(K)
\]
 where $\check{\nu}_{\epsilon}(K)$, $\check{\nu}(K)$ are tubular
neighborhoods of $K$, and $\check{\nu}_{\epsilon}(K)\subset\check{\nu}(K)$.
Then the first terms in the Mayer-Vietoris sequence for this decomposition
read 
\begin{align*}
\cdots & \check{H}^{0}((\check{Y}\backslash\check{\nu}_{\epsilon}(K))\cap\check{\nu}(K);\mathfrak{g}_{\rho})\\
\rightarrow & \check{H}^{1}(\check{Y};\mathfrak{g}_{\rho})\\
\rightarrow^{(\imath_{\check{Y}\backslash\check{\nu}_{\epsilon}(K)}^{*},\imath_{\check{\nu}(K)}^{*})} & \check{H}^{1}(\check{Y}\backslash\check{\nu}_{\epsilon}(K);\mathfrak{g}_{\rho})\oplus\check{H}^{1}(\check{\nu}(K);\mathfrak{g}_{\rho})\\
\rightarrow & \check{H}^{1}((\check{Y}\backslash\check{\nu}_{\epsilon}(K))\cap\check{\nu}(K);\mathfrak{g}_{\rho})\\
\rightarrow & \check{H}^{2}(\check{Y};\mathfrak{g}_{\rho})\\
\rightarrow & \cdots
\end{align*}
Here $\imath_{\check{Y}\backslash\check{\nu}_{\epsilon}(K)}^{*}$
and $\imath_{\check{\nu}(K)}^{*}$ denote the pullback in cohomology
of the inclusions $\imath:\check{Y}\backslash\check{\nu}_{\epsilon}(K)\hookrightarrow\check{Y}$
and $\imath:\check{\nu}(K)\hookrightarrow\check{Y}$ . Now using the
usual properties of invariance of the cohomology groups under deformation
retractions we can rewrite the previous sequence as
\begin{align*}
\cdots & H^{0}(T_{\epsilon};\mathfrak{g}_{\rho})\\
\rightarrow & \check{H}^{1}(\check{Y};\mathfrak{g}_{\rho})\\
\rightarrow^{(\imath_{Y\backslash K}^{*},\imath_{\check{\nu}(K)}^{*})} & H^{1}(Y\backslash K;\mathfrak{g}_{\rho})\oplus\check{H}^{1}(\check{\nu}(K);\mathfrak{g}_{\rho})\\
\rightarrow^{i_{T_{\epsilon},\nu(K)}^{*}-i_{T_{\epsilon,Y\backslash K}}^{*}} & H^{1}(T_{\epsilon};\mathfrak{g}_{\rho})\\
\rightarrow & \cdots
\end{align*}
 where $T_{\epsilon}=\partial\check{\nu}_{\epsilon}(K)$. Notice that
we have dropped the $\check{}$ notation on the regions where the
orbifold singularity is not present.

Observe first of all that since $T_{\epsilon}$ has abelian fundamental
group then the restriction of $\mathfrak{g}_{\rho}$ to $T_{\epsilon}$
automatically becomes reducible and we can write 
\[
\mathfrak{g}_{\rho}\mid_{T_{\epsilon}}=\mathbb{R}\oplus L_{\rho}^{\otimes2}
\]
Now, because of the holonomy condition, for $\epsilon$ sufficiently
small the local system defined by $L_{\rho}^{\otimes2}$ is non-trivial,
which means in particular that 
\[
H^{\bullet}(T_{\epsilon};L_{\rho}^{\otimes2})\equiv0
\]
by \cite[Lemma 63]{Echeverria[DraftFuruta]}. If we write $\mu_{K}$
and $\lambda_{K}$ for the meridian and longitudes of $K$, then the
Mayer-Vietoris sequence simplifies to 
\begin{alignat}{1}
\cdots & \mathbb{R}\nonumber \\
\rightarrow & \check{H}^{1}(\check{Y};\mathfrak{g}_{\rho})\\
\rightarrow^{(\imath_{Y\backslash K}^{*},\imath_{\check{\nu}(K)}^{*})} & H^{1}(Y\backslash K;\mathfrak{g}_{\rho})\oplus\check{H}^{1}(\check{\nu}(K);\mathfrak{g}_{\rho})\label{simplified mayer vietoris}\\
\rightarrow^{i_{T_{\epsilon},\nu(K)}^{*}-i_{T_{\epsilon,Y\backslash K}}^{*}} & \mathbb{R}[\mu_{K}]\oplus\mathbb{R}[\lambda_{K}]\nonumber \\
\rightarrow & \check{H}^{2}(\check{Y};\mathfrak{g}_{\rho})\\
\rightarrow & \cdots\nonumber 
\end{alignat}

For computing $\check{H}^{1}(\check{\nu}(K);\mathfrak{g}_{\rho})$
write 
\[
\check{\nu}(K)=S^{1}\times\check{D}^{2}
\]
where we regard $\check{D}^{2}$ as an orbifold. In fact, we can think
of $\check{D}^{2}$ as $D^{2}/\mathbb{Z}_{\nu}$, where $\mathbb{Z}_{\nu}$
denotes a cyclic action determining the cone angle $2\pi/\nu$. If
$p:D^{2}\backslash\{0\}\rightarrow\check{D}^{2}\backslash\{0\}$ denotes
the quotient map, then as long as the cone angle is sufficiently sharp
and $\alpha$ is a rational value, the pullback of the flat connection
$\rho$ to $D^{2}\backslash\{0\}$ extends smoothly to a flat connection
$p^{*}\rho$ on all of $D^{2}$, and in fact we can identify $\check{H}^{1}(\check{D}^{2};\mathfrak{g}_{\rho})$
with the equivariant cohomology $H^{1,\nu}(D^{2};p^{*}\mathfrak{g}_{\rho})$.
Hence 
\[
\check{H}^{1}(\check{\nu}(K);\mathfrak{g}_{\rho})=H^{0}(S^{1};\mathfrak{g}_{\rho})\otimes H^{1,\nu}(D^{2};p^{*}\mathfrak{g}_{\rho})\oplus H^{1}(S^{1};\mathfrak{g}_{\rho})\otimes H^{0,\nu}(D^{2};p^{*}\mathfrak{g}_{\rho})
\]
Now, the factor $H^{1,\nu}(D^{2};p^{*}\mathfrak{g}_{\rho})$ will
vanish since $D^{2}$ equivariantly retracts to the origin so that
\[
H^{1,\nu}(D^{2};p^{*}\mathfrak{g}_{\rho})\simeq H^{1,\nu}(\{0\},p^{*}\mathfrak{g}_{\rho})=0
\]
Likewise, we have that 
\[
H^{0,\nu}(D^{2};p^{*}\mathfrak{g}_{\rho})\simeq H^{0,\nu}(\{0\};p^{*}\mathfrak{g}_{\rho})\simeq\mathbb{R}
\]
since only one factor of $\mathbb{R}^{3}\simeq p^{*}\mathfrak{g}_{\rho}\mid_{\{0\}\subset D^{2}}$
is preserved under the group action. To compute $H^{1}(S^{1};\mathfrak{g}_{\rho})$
we use again the fact that $S^{1}$ has abelian fundamental group
which in particular means that $\mathfrak{g}_{\rho}$ reduces to the
system 
\[
\mathfrak{g}_{\rho}\mid_{S^{1}}\simeq\mathbb{R}\oplus L_{\rho}^{\otimes2}
\]
In this case, $L_{\rho}^{\otimes2}$ may or may not be the trivial
system which means that 
\[
H^{1}(S^{1};\mathfrak{g}_{\rho})=\begin{cases}
H^{1}(S^{1};\mathbb{R}) & \text{if }L_{\rho}^{\otimes2}\mid_{S^{1}}\text{ non-trivial}\\
H^{1}(S^{1};\mathbb{R})\otimes\mathbb{R}^{3} & \text{if }L_{\rho}^{\otimes2}\mid_{S^{1}}\text{ trivial}
\end{cases}
\]
In any case that means that 
\begin{equation}
\check{H}^{1}(\check{\nu}(K);\mathfrak{g}_{\rho})\simeq H^{1}(S_{\lambda_{K}}^{1};\mathfrak{g}_{\rho})\label{iso cohomology bundle}
\end{equation}
so the maps in the Mayer-Vietoris sequence \ref{simplified mayer vietoris}
become
\begin{align}
\check{H}^{1}(\check{Y};\mathfrak{g}_{\rho}) & \rightarrow^{(\imath_{Y\backslash K}^{*},\imath_{\check{\nu}(K)}^{*})} & H^{1}(Y\backslash K;\mathfrak{g}_{\rho})\oplus\check{H}^{1}(\check{\nu}(K);\mathfrak{g}_{\rho})\label{Map Mayer 1}\\
\check{\omega} & \rightarrow & \left(\check{\omega}\mid_{Y\backslash K},\check{\omega}\mid_{\check{\nu}(K)}\right)\nonumber 
\end{align}
and 
\begin{align}
H^{1}(Y\backslash K;\mathfrak{g}_{\rho})\oplus\check{H}^{1}(\check{\nu}(K);\mathfrak{g}_{\rho}) & \rightarrow^{i_{T_{\epsilon},\nu(K)}^{*}-i_{T_{\epsilon,Y\backslash K}}^{*}} & \mathbb{R}[\mu_{K}]\oplus\mathbb{R}[\lambda_{K}]\label{Map mayer 2}\\
\left(\omega,\check{\varpi}\right) & \rightarrow & \left(\left\langle \omega\mid_{S_{\mu_{K}}^{1}},[\mu_{K}]\right\rangle ,\left\langle (\omega-\check{\varpi})\mid_{S_{\lambda_{K}}^{1}},[\lambda_{K}]\right\rangle \right)\nonumber 
\end{align}
where the notation means that we are restricting the forms to the
$S^{1}$ factors generated by the meridian and longitude respectively,
and then using the pairing coming from the isomorphisms we established
previously. In particular, since Mayer-Vietoris is an exact sequence
this means that for any $\check{\omega}\in\check{H}^{1}(\check{Y};\mathfrak{g}_{\rho})$
we have 
\[
\left(i_{T_{\epsilon},\nu(K)}^{*}-i_{T_{\epsilon,Y\backslash K}}^{*}\right)\circ(\imath_{Y\backslash K}^{*},\imath_{\check{\nu}(K)}^{*})\check{\omega}=0
\]
which according to the formulas \ref{Map Mayer 1} and \ref{Map mayer 2}
imply that 
\[
\left\langle \check{\omega}\mid_{S_{\mu_{K}}^{1}},[\mu_{K}]\right\rangle =0
\]
In other words, there is a well defined map 
\begin{align*}
\imath^{*}:\check{H}^{1}(\check{Y};\mathfrak{g}_{\rho}) & \rightarrow & \left\{ \omega\in H^{1}(Y\backslash K;\mathfrak{g}_{\rho})\mid\left\langle \omega\mid_{S_{\mu_{K}}^{1}},[\mu_{K}]\right\rangle =0\right\} \\
\check{\omega} & \rightarrow & \check{\omega}\mid_{Y\backslash K}
\end{align*}
or more succinctly $\imath^{*}:\check{H}^{1}(\check{Y};\mathfrak{g}_{\rho})\rightarrow\ker:\left(H^{1}(Y\backslash K;\mathfrak{g}_{\rho})\rightarrow H^{1}(m;\mathfrak{g}_{\rho})\right)$
as we wrote in the statement of the lemma. We just need to show that
this map is an isomorphism.

First we address the surjectivity of the map. Suppose that $\omega\in H^{1}(Y\backslash K;\mathfrak{g}_{\rho})$
and that $\left\langle \omega\mid_{S_{\mu_{K}}^{1}},[\mu_{K}]\right\rangle =0$.
Using the isomorphism \ref{iso cohomology bundle} we can consider
the element 
\[
\left(\omega,\omega\mid_{S_{\lambda_{K}}^{1}}\right)\in H^{1}(Y\backslash K;\mathfrak{g}_{\rho})\oplus\check{H}^{1}(\check{\nu}(K);\mathfrak{g}_{\rho})
\]
Because $\left\langle \omega\mid_{S_{\mu_{K}}^{1}},[\mu_{K}]\right\rangle =0$
we have that $\left(\omega,\omega\mid_{S_{\lambda_{K}}^{1}}\right)\in\ker\left(i_{T_{\epsilon},\nu(K)}^{*}-i_{T_{\epsilon,Y\backslash K}}^{*}\right)$
which means by exactness of the Mayer-Vietoris sequence that $\left(\omega,\omega\mid_{S_{\lambda_{K}}^{1}}\right)\in\text{im}(\imath_{Y\backslash K}^{*},\imath_{\check{\nu}(K)}^{*})$.
In other words, there is $\check{\omega}\in\check{H}^{1}(\check{Y};\mathfrak{g}_{\rho})$
such that $\imath^{*}(\check{\omega})=\omega$, which is what we wanted
to show.

For injectivity we simply observe that if $\imath^{*}(\check{\omega})=[0]$
then $\omega=\check{\omega}\mid_{Y\backslash K}$ is exact, that is,
$\omega=d_{B}\xi$ for some $\xi\in\varOmega^{0}(Y\backslash K;\mathfrak{g}_{\rho})$
and $B$ is the connection representative of the flat connection $\rho$.
Now we can compute the norm of $\check{\omega}$ in the following
way, which is essentially the same computation as in \cite[Proposition 2.10]{MR2696986}
\begin{align*}
\|\check{\omega}\|_{\check{L}^{2}(\check{Y})}^{2} & = & -\lim_{\epsilon\rightarrow0}\int_{Y\backslash\check{\nu}_{\epsilon}(K)}\text{tr}(*\omega\wedge\omega)\\
 & = & -\lim_{\epsilon\rightarrow0}\int_{Y\backslash\check{\nu}_{\epsilon}(K)}\text{tr}(*\omega\wedge d_{B}\xi)\\
 & = & \lim_{\epsilon\rightarrow0}\left[\int_{Y\backslash\check{\nu}_{\epsilon}(K)}\text{tr}(d_{B}(*\omega)\wedge\xi)-\int_{T_{\epsilon}}\text{tr}(*\omega\wedge\xi)\right]\\
 & = & -\lim_{\epsilon\rightarrow0}\int_{T_{\epsilon}}\text{tr}(*\omega\wedge\xi)\\
 & = & 0
\end{align*}
 where we have used that $d_{B}(*\omega)=*d_{B}^{*}\omega=0$ since
$\check{\omega}\in\ker\check{\triangle}_{B}=\ker(\check{d}_{B})\cap\ker(\check{d}_{B}^{*})$.
That $\lim_{\epsilon\rightarrow0}\int_{T_{\epsilon}}\text{tr}(*\omega\wedge\xi)=0$
vanishes is essentially a restatement that on orbifolds one can use
integration by parts \cite[Section 2.1]{MR3376575}. Alternatively,
one can follow the approximation scheme in the proof of \cite[Proposition 8.3]{MR1241873}.
Hence $\|\check{\omega}\|_{\check{L}^{2}(\check{Y})}^{2}=0$ which
also means that $\check{\omega}=0$, which is what we needed to show
for the proof of injectivity.
\end{proof}

Using values of $\alpha\in\mathbb{Q}\cap(0,1/2)$ which satisfy $\triangle_{K}(e^{-4\pi i\alpha})\neq0$,
and with a local system which satisfies the properties discussed in
the previous section, we can define the chain complex 
\[
(C_{*}(Y,K,\alpha;\varGamma,\mathfrak{p}),\partial)
\]
where we are also indicating the dependence of the chain complex on
the perturbation used to achieve transversality. Again, the local
coefficient system we will have in mind is the universal Novikov field/local
system described in the previous section, in order to minimize additional
choices to get an actual Floer group, rather than a group up to isomorphism.

We will choose our perturbation $\mathfrak{p}$ as 
\[
\mathfrak{p}=\mathfrak{p}_{\text{crit}}+\mathfrak{p}_{\partial}
\]
where $\mathfrak{p}_{\text{crit}}$ and $\mathfrak{p}_{\partial}$
mean the following. If $\mathfrak{p}_{\text{crit}}$ is a perturbation
which makes the critical set $\mathfrak{C}_{\mathfrak{p}_{\text{crit}}}$
non-degenerate, then we can approximate $\mathfrak{p}_{\text{crit}}$
by a finite sum of holonomy perturbations. Since the non-degeneracy
of the compact set $\mathfrak{C}_{\mathfrak{p}_{\text{crit}}}$ is
an open condition, we can arrange that $f_{\mathfrak{p}_{\text{crit }}}$
is a finite sum by truncating $\mathfrak{p}_{\text{crit}}$ after
finitely many terms. This guarantees that the support of the holonomy
perturbations do not meet a neighborhood of the knot $K$. This argument
appears in Proposition 3.1 of \cite{MR3880205}, although the setup
in Kronheimer and Mrowka's paper is more complicated since they also
need to guarantee certain finiteness condition for the moduli spaces
of flow lines of dimension less or equal to $2$. Then $\mathfrak{p}_{\partial}$
is the remaining perturbation needed to cut out the moduli spaces
of trajectories transversely.

Because of this support condition, we can appeal to Herald's result
\cite[Theorem 0.1]{MR1456309}, which relates the signed count of
elements in the (finite) set $\mathfrak{C}_{\mathfrak{p}_{\text{crit}}}$
to the Casson invariant and the knot signature of $K$. 
\begin{thm}
\label{Signed Count}Suppose that $\alpha\in\mathbb{Q}\cap(0,1/2)$
satisfies $\triangle_{K}(e^{-4\pi i\alpha})\neq0$. Then the signed
count of $\mathfrak{C}_{\mathfrak{p}_{\text{crit}}}$, which we denote
$\#_{s}|\mathfrak{C}_{\mathfrak{p}_{\text{crit}}}|$, equals
\begin{equation}
\#_{s}|\mathfrak{C}_{\mathfrak{p}_{\text{crit}}}|=4\lambda_{C}(Y)+\frac{1}{2}\sigma_{K}(e^{-4\pi i\alpha})\label{signed count critical points}
\end{equation}
where $\lambda_{C}(Y)$ is the Casson invariant of $Y$ and $\sigma_{K}(e^{-4\pi i\alpha})$
the Tristram-Levine knot signature of $K$ evaluated at $e^{-4\pi i\alpha}$.

Moreover, as vector spaces over the Novikov field $\varLambda$, 
we have that 
\[
\chi_{\varLambda}(HI(Y,K,\alpha))=4\lambda_{C}(Y)+\frac{1}{2}\sigma_{K}(e^{-4\pi i\alpha})
\]
\end{thm}

Notice that the second part of previous statement implicitly assumed:

a) A choice for the orientation of the moduli spaces involved in the
definition of $\mathfrak{C}_{\mathfrak{p}_{crit}}$ and $HI(Y,K,\alpha)$.

b) The claim that there is a well defined differential $\partial$
which squares to zero so that we can take $HI(Y,K,\alpha)=\ker\partial/\text{im}\partial$. 

c) A proof that $HI(Y,K,\alpha)$ is independent of the perturbations
used. 

We will start by discussing the orientation of the moduli spaces,
which we have kept under the rug until this point. For our conventions
we will follow closely the exposition in \cite[Section 4.3]{MR3394316},
which in turn are based on those used in \cite{MR1883043,MR2805599,MR2860345}.
They also agree with the conventions used by Collin and Steer in their
paper \cite{MR1703606}, which we used to pin down the signs in the
formula (\ref{signed count critical points}). 

The orientation set $o([\beta])$ of a critical point $\mathfrak{C}_{\mathfrak{p}}$
refers to the 2-element set of orientations of the real line $\det D_{A}$,
where $A$ is a connection on $\mathbb{R}\times Y$ such that $A\mid_{\{t\}\times Y}$
is (gauge) equivalent to $\theta_{\alpha}$ for $t$ sufficiently
negative and (gauge) equivalent to $\beta$ for $t$ sufficiently
large. Here $D_{A}$ is the Fredholm operator $-d_{A}^{*}\oplus d_{A}^{+}$,
after suitable Sobolev completions have been introduced. In general
any reference connection would work, but the advantage of using the
reducible connection $\theta_{\alpha}$ as one of the limiting connections
is that it will automatically gives us an absolute $\mathbb{Z}/4\mathbb{Z}$
grading as we will explain momentarily.

In general, given two connections $\beta_{1}\in\mathcal{C}(Y_{1},K_{1},\alpha)$
and $\beta_{2}\in\mathcal{C}(Y_{2},K_{2},\alpha_{2})$ and a cobordism
\[
(W,\varSigma,\alpha):\emptyset\rightarrow(-Y_{1},-K_{1},\alpha)\sqcup(Y_{2},K_{2},\alpha)
\]
 one can consider the operator (for $\epsilon>0$ sufficiently small)
\[
D_{A}'=-d_{A}^{*}\oplus d_{A}^{+}:\check{L}_{m,\phi'_{\epsilon}}^{p}(\varLambda^{1}\otimes\mathfrak{g}_{P})\rightarrow\check{L}_{m-1,\phi'_{\epsilon}}^{p}((\varLambda^{0}\oplus\varLambda^{+})\otimes\mathfrak{g}_{P}))
\]
where $A\in\mathcal{C}(W^{*},\varSigma^{*},\alpha)$ is a connection
on the completion
\begin{align*}
W^{*}=(\mathbb{R}^{-}\times Y_{1})\cup W\cup(\mathbb{R}^{+}\times Y_{2})\\
\varSigma^{*}=(\mathbb{R}^{-}\times K_{1})\cup W\cup(\mathbb{R}^{+}\times K_{2})
\end{align*}
 which is asymptotic to $\beta_{1}$ for $t$ sufficiently negative
and to $\beta_{2}$ for $t$ sufficiently positive. Here $\check{L}_{m,\phi'_{\epsilon}}^{p}=e^{\phi_{\epsilon}'}\check{L}_{m}^{p}$
denotes a weighted Sobolev space, weighted by a real function $e^{\phi'_{\epsilon}}$,
where $\phi_{\epsilon}'$ is a non-positive smooth function equal
to $\epsilon t$ for $t$ sufficiently negative on $\mathbb{R}^{-}\times Y_{1}$,
to $-\epsilon t$ for $t$ sufficiently large on $\mathbb{R}^{+}\times Y_{2}$
, and equal to $0$ on $W$. 

When $\beta_{1}$ is a perturbed instanton on $(Y_{1},K_{1},\alpha,\mathfrak{p}_{1})$,
with respect to the perturbation $\mathfrak{p}_{1}$, and similarly
for $\beta_{2}$ with respect to the perturbation $\mathfrak{p}_{2}$
on $(Y_{2},K_{2},\alpha,\mathfrak{p}_{2})$, then the moduli space
$\mathcal{M}([\beta_{1}],(W,\varSigma),[\beta_{2}])$ of perturbed
$\alpha$-ASD instantons, which we will denote as $\mathcal{M}([\beta_{1}],\check{W},[\beta_{2}])$
will decompose into connected components 
\[
\mathcal{M}([\beta_{1}],\check{W},[\beta_{2}])=\bigcup_{z}\mathcal{M}_{z}([\beta_{1}],\check{W},[\beta_{2}])
\]
 and the dimension of $\mathcal{M}_{z}([\beta_{1}],\check{W},[\beta_{2}])$
can be computed as $\text{ind}D_{A}'$, where $A$ is an appropriate
representative of an element $[A]\in\mathcal{M}_{z}([\beta_{1}],\check{W},[\beta_{2}])$
. For a composite cobordism, one has the additivity relation (assuming
the critical points are non-degenerate) \cite[eq. 4.3]{MR3394316}
\begin{align}
 & \dim\mathcal{M}_{z'\circ z}([\beta_{0}],\check{W}'\circ\check{W},[\beta_{2}])\nonumber \\
= & \dim\mathcal{M}_{z}([\beta_{0}],\check{W},[\beta_{1}])+\dim\text{stab}[\beta_{1}]+\dim\mathcal{M}_{z'}([\beta_{1}],\check{W}',[\beta_{2}])\label{additivity indices}
\end{align}
When $[\beta_{0}],[\beta_{1}]$ are irreducible (perturbed) flat connections
on $(Y,K,\alpha)$, we had defined in Section \ref{sec:Monotonicity-and-Novikov}
a relative grading $\text{gr}([\beta_{0}],[\beta_{1}])\in\mathbb{Z}/4\mathbb{Z}$
in terms of the spectral flow of the extended Hessian. Since the spectral
flow of this operator can be interpreted as the index of the operator
$D_{A}'$ on the cylinder $\mathbb{R}\times Y$, this means that the
relative grading can be interpreted as the dimension (mod-4) of the
moduli space of flow-lines between $[\beta_{0}]$ and $[\beta_{1}]$:
\[
\text{gr}([\beta_{0}],[\beta_{1}])=\dim\mathcal{M}([\beta_{0}],[\beta_{1}])\mod4
\]
where again we are looking at moduli spaces asymptotic to $[\beta_{0}]$
as $t\rightarrow-\infty$ and to $[\beta_{1}]$ as $t\rightarrow\infty$
. 

To make this an absolute grading, set 
\begin{equation}
\text{gr}([\beta])=-1-\dim\mathcal{M}([\theta_{\alpha}],[\beta])\mod4=\dim\mathcal{M}([\beta],[\theta_{\alpha}])\mod4\label{absolute grading}
\end{equation}
and give grading $0$ to the reducible connection $[\theta_{\alpha}]$.
Notice that  because $[\theta_{\alpha}]$ has a one-dimensional stabilizer
according to (\ref{additivity indices}) we have that 
\begin{align*}
 & \text{gr}([\beta_{0}])-\text{gr}([\beta_{1}])\\
= & \dim\mathcal{M}([\theta_{\alpha}],[\beta_{1}])-\dim\mathcal{M}([\theta_{\alpha}],[\beta_{0}])\mod4\\
= & \dim\mathcal{M}([\theta_{\alpha}],[\beta_{1}])+\dim\mathcal{M}([\beta_{0}],[\theta_{\alpha}])+1\mod4\\
= & \dim\mathcal{M}([\beta_{0}],[\beta_{1}])\mod4\\
= & \text{gr}([\beta_{0}],[\beta_{1}])
\end{align*}
With respect to this absolute grading, if $\mathfrak{C}_{\mathfrak{p},i}$
denotes the set of critical points associated to the perturbation
$\mathfrak{p}$ whose absolute grading is $i\in\mathbb{Z}/4\mathbb{Z}$,
then an expression like (\ref{signed count critical points}) means
\[
\#_{s}|\mathfrak{C}_{\mathfrak{p}_{\text{crit}}}|=\#|\mathfrak{C}_{\mathfrak{p}_{\text{crit}},0}|+\#|\mathfrak{C}_{\mathfrak{p}_{\text{crit}},2}|-\#|\mathfrak{C}_{\mathfrak{p}_{\text{crit}},1}|-\#|\mathfrak{C}_{\mathfrak{p}_{\text{crit}},3}|
\]
 Now we must turn to a discussion of the properties of the perturbations
used in order to guarantee that we have a well defined differential
and that the homology groups we obtain are independent of the perturbations
used. First we need to understand what is at stake. Recall that we
are assigning to each element $[\beta]\in\mathcal{B}(Y,K,\alpha)$
a vector space $\varGamma_{[\beta]}$ which is a copy of the universal
Novikov field $\varLambda^{\mathbb{Q},\mathbb{R}}$. In other words,
an element of $\varGamma_{[\beta]}$ is a formal power series 
\begin{equation}
\sum_{r\in\mathbb{R}}a_{r}T^{r}\label{formal power series-1}
\end{equation}
where $a_{r}\in\mathbb{Q}$, and $\#\{r\mid a_{r}\neq0,r>C\}<\infty$
for all $C\in\mathbb{R}$. Therefore, the power series can extend
indefinitely in the negative direction, but not the positive one.
Now, the differential $\partial$ on the chain complex $C_{*}(Y,K,\alpha,\varGamma,\mathfrak{p})$
should be defined by the formula (for $[\beta_{1}]\in\mathfrak{C}(Y,K,\alpha,\mathfrak{p})$)
\begin{equation}
\partial[\beta_{1}]=\sum_{[\beta_{2}]\in\mathfrak{C}(Y,K,\alpha,\mathfrak{p})\mid\text{gr}([\beta_{1}],[\beta_{2}])=1\mod4}\sum_{z}n_{z}([\beta_{1}],[\beta_{2}])T^{-\mathcal{E}_{top}(z)}[\beta_{2}]\label{proposed differential}
\end{equation}
 Here 
\[
\mathcal{E}_{top}(z)=\frac{1}{8\pi^{2}}\int_{\mathbb{R}\times\check{Y}}\text{tr}(F_{A}\wedge F_{A})
\]
where $[A]\in\mathcal{M}_{z}([\beta_{1}],[\beta_{2}])$. In the absence
of perturbations, i.e, $\mathfrak{p}=0$ then $\mathcal{E}_{top}(z)$
would equal $\frac{1}{8\pi^{2}}\|F_{A}^{-}\|_{L^{2}}=2(CS(\beta_{1})-CS(\beta_{2}))$,
since $A$ would solve the $\alpha$-ASD equation $F_{A}^{+}=0$.
Therefore, the exponents in $T^{-\mathcal{E}_{top}(z)}$ would always
be non-positive and do not accumulate on any finite subinterval, since
they must always differ from each other by some integer combination
of $64\pi^{2}$ and $128\pi^{2}\alpha$ (this follows from the formula
(\ref{topological energy}) for the topological energy in the closed
4-manifold case). In other words, the candidate for the differential
(\ref{proposed differential}) does make sense as an element (\ref{formal power series-1})
of $\varLambda^{\mathbb{Q},\mathbb{R}}$ .

In the presence of perturbations $\mathfrak{p}$, the inequality that
is satisfied is 
\[
\mathcal{E}_{top}(z)+f_{\mathfrak{p}}([\beta_{1}])-f_{\mathfrak{p}}([\beta_{2}])\geq0
\]
Since $f_{\mathfrak{p}}$ is fully gauge equivariant as we mentioned
when we introduced the holonomy perturbations, the lower bound for
$\mathcal{E}_{top}(z)$
\begin{equation}
\mathcal{E}_{top}(z)\geq f_{\mathfrak{p}}([\beta_{2}])-f_{\mathfrak{p}}([\beta_{1}])\label{inequality perturbed}
\end{equation}
 is independent of the specific moduli space $\mathcal{M}_{z}([\beta_{1}],[\beta_{2}])$
we are looking at, so the expression (\ref{proposed differential})
continues to be well defined as an element of $\varLambda^{\mathbb{Q},\mathbb{R}}$.

After we know that $\partial$ makes sense, proving $\partial^{2}=0$
is in this setting is no different than the situation for monopole
Floer homology with local coefficients \cite[Section 30]{MR2388043},
except for two main differences:
\begin{enumerate}
\item In the instanton setup, bubbling in general can play a role.
\item We have excluded the reducible flat connection from our chain group,
so we need to guarantee that the broken trajectories considered to
show that $\partial^{2}=0$ do not include factorizations through
this reducible connection.
\end{enumerate}
Regarding the first point, this is not a problem for showing that
$\partial^{2}=0$, since bubbles drop the dimension of the moduli
spaces involve by at least $4$ \cite[Proposition 3.22]{MR2860345},
and our transversality assumptions for the moduli space of flow lines
guarantee that negative dimensional moduli spaces are empty.

The second point involves analyzing the non-degeneracy of the reducible
connection and some index formulas. 
\begin{itemize}
\item Even after perturbations, there is still one reducible connection
$[\theta_{\alpha}]$ up to gauge, the same one as the unperturbed
case, since $[\theta_{\alpha}]$ is isolated from the irreducible
connections so one can choose the holonomy perturbations in such a
way that they vanish near $[\theta_{\alpha}]$. Notice that $[\theta_{\alpha}]$
is unobstructed since by Poincare duality $\check{H}^{2}(\check{Y};\mathfrak{g}_{\alpha})\simeq\check{H}^{1}(\check{Y};\mathfrak{g}_{\rho})=0$.
\item If $\mathcal{M}_{z}([\beta_{0}],[\beta_{2}])$ represents a $d$-dimensional
moduli space, and there is a broken trajectory belonging to $\mathcal{M}_{z_{1}}([\beta_{0}],[\theta_{\alpha}])\times\mathcal{M}_{z_{2}}([\theta_{\alpha}],[\beta_{2}])$
with dimensions $d_{1},d_{2}$ respectively, the fact that $[\theta_{\alpha}]$
has one dimensional stabilizer implies that 
\[
\dim\mathcal{M}_{z}([\beta_{0}],[\beta_{2}])=\dim\mathcal{M}_{z_{1}}([\beta_{0}],[\theta_{\alpha}])+\dim\mathcal{M}_{z_{2}}([\theta_{\alpha}],[\beta_{2}])+1
\]
so in particular the right hand side is bounded from below by $3$,
since each $d_{1},d_{2}$ must be at least one dimensional. Hence
moduli spaces admitting $\mathcal{M}_{z}([\beta_{0}],[\beta_{2}])$
such factorizations through the reducible connection $[\theta_{\alpha}]$
must be at least three dimensional, so they can be ignored for the
definition of the differential. Notice that this is the same as what
happens in the ordinary case of Instanton Floer homology for integer
homology spheres, but in that case the bound for the dimension is
$5$, since the stabilizer of the trivial connection is three-dimensional.
\end{itemize}
This means that we can define the \textbf{Instanton Floer-Novikov
homology for knots} $HI(Y,K,\alpha)$ as $\ker\partial/\text{im}\partial$.
Notice that our notation makes implicit that the homology we obtain
is independent of the choice of perturbation $\pi\in\mathcal{P}$.
To see why this is true, we adapt the proof of independence in \cite[Section 5.3]{MR1883043}
to our situation, and discuss more generally the functoriality properties
of these instanton Floer-Novikov groups.

Suppose that $(W,\varSigma):(Y_{1},K_{1})\rightarrow(Y_{2},K_{2})$
is a concordance of the knots $K_{1},K_{2}$. By this we mean that
$Y_{1},Y_{2}$ will be both be integer homology spheres and $W$ homology
cobordism, i.e $H_{*}(W;\mathbb{Z})=H_{*}(Y_{i};\mathbb{Z})$ for
$i=1,2$. Moreover, $\varSigma$ will be an embedded annulus with
$\partial\varSigma=-K_{1}\sqcup K_{2}$. 

Finally, suppose that the cobordism $(W,\varSigma)$ is $\alpha$-admissible
in the sense of Definition \ref{def: alpha admissible}. Recall that
this means that for the unique reducible up to gauge $\theta_{W,\alpha}$
we have $H^{1}(W\backslash\varSigma;L_{\theta_{W,\alpha}}^{\otimes2})=0$.
Since $\theta_{Y_{1},\alpha}$ and $\theta_{Y_{2},\alpha}$ also satisfy
$H^{1}(Y_{i}\backslash K_{i};L_{\theta_{Y_{i},\alpha}}^{\otimes2})=0$
for $i=1,2$, it is immediate that on the completion $H^{1}(W^{*}\backslash\varSigma^{*};L_{\theta_{W^{*},\alpha}}^{\otimes2})=0$,
and an straightforward adaptation of Lemma \ref{orbifold cohomology}
will imply that for $\alpha$-admissible cobordisms the reducible
$\theta_{W,\alpha}$ is isolated (and non-degenerate as well).

We want to define a cobordism map 
\[
m_{(W,\varSigma)}:C_{*}(Y_{1},K_{1},\alpha,\varGamma_{\mathfrak{p}},\mathfrak{p}_{1})\rightarrow C_{*}(Y_{2},K_{2},\alpha,\varGamma_{\mathfrak{p}},\mathfrak{p}_{2})
\]
 via the formula 
\begin{equation}
m_{(W,\varSigma)}[\beta_{1}]=\sum_{[\beta_{2}]}\sum_{z:[\beta_{1}]\rightarrow[\beta_{2}]}m_{z}([\beta_{1}],\check{W},[\beta_{2}])T^{-\mathcal{E}_{top}(z)}[\beta_{2}]\label{cobordism map}
\end{equation}
Here the sum is taking place over all homotopy classes $z$ for which
the moduli space $\mathcal{M}_{z}([\beta_{1}],\check{W},[\beta_{2}])$
is zero dimensional, and the notation $\check{W}$ is emphasizing
that we are regarding $\check{W}$ as an orbifold. Before showing
that $m_{z}([\beta_{1}],\check{W},[\beta_{2}])$ is well defined,
let's discuss first how the cobordism maps should interact with the
gradings $HI_{i}(Y,K,\alpha)$ of the Floer groups. 

If $\text{gr}([\beta_{i}])$ denotes the (absolute) $\mod4$ grading
of $[\beta_{i}]$ then \cite[Proposition 4.4]{MR2805599}, \cite[Section 4.3]{MR1703606}
shows that 
\begin{equation}
\dim\mathcal{M}_{z}([\beta_{1}],\check{W},[\beta_{2}])=\text{gr}([\beta_{1}])-\text{gr}([\beta_{2}])-\frac{3}{2}(\chi(W)+\sigma(W))-\chi(\varSigma)\mod4\label{dimension moduli space}
\end{equation}
In particular, given our assumptions on $(W,\varSigma)$, we can see
that the cobordism map $m_{(W,\varSigma)}$ is a sum over all elements
$[\beta_{2}]$ whose relative grading is the same as $[\beta_{1}]$.
What is left to see is why the formula for $m_{(W,\varSigma)}[\beta_{1}]$
defines an element in $\varLambda^{\mathbb{Q},\mathbb{R}}$.

This requires further discussion of how the $ASD$ equation is perturbed
on a cobordism. Here we follow \cite[Section 3.8]{MR2860345}, \cite[Section 2.2]{Daemi[2019]},
\cite[Section 3.1]{MR2192061} and \cite[Section 4]{Miller[2019]}.
On the cobordism $W$ we choose a collar neighborhood of each boundary
component, and if $t$ denotes the coordinate (say in the collar $[0,1)\times Y_{1}$)
, we choose a $t$ -dependent holonomy perturbation $\mathfrak{p}_{t}$
equal to $\mathfrak{p}_{1}$ on the near $[0,1/4)\times Y_{1}$. On
$(3/4,1)\times Y_{1}$ the perturbation $\mathfrak{p}_{t}$ vanishes
and then we interpolate on $[1/4,3/4]\times Y_{1}$ choosing an auxiliary
perturbation $\tilde{\mathfrak{p}}_{1}\in\mathcal{P}_{Y_{1}}$. The
net effect is that on $[0,1)\times Y_{1}$ the perturbed equations
take the form 
\[
F_{A}^{+}+\beta_{1}(t)U_{1}(A)+\tilde{\beta_{1}}(t)\tilde{U}_{1}(A)=0
\]
 where $\beta_{1},\tilde{\beta}_{1}$ denote suitable cut-off functions
and $U_{1},\tilde{U}_{1}$ denote the perturbation terms associated
to $\beta_{1},\tilde{\beta}_{1}$. Similar remarks apply to $Y_{2}$.
These were the perturbations used in \cite[Section 3.8]{MR2860345}.
To deal with transversality issues involving flat connections on the
cobordism, we also need interior holonomy perturbations, supported
on a compact subset of $W\backslash\text{nbd}(\varSigma\cup\partial W)$,
which were introduced for closed 4-manifolds in \cite{MR2192061},
although analogous constructions appear for example in \cite{MR910015,MR2052970}. 

These are constructed in a similar way to how cylinder functions were
constructed on a 3 manifold. We sketch the construction given in \cite[Section 3.8]{MR2860345},
which also appears in \cite[Definition 4.2]{Miller[2019]}. Namely,
one chooses a closed ball $B\subset W\backslash\text{nbd}(\partial W)$
and a finite collection of smooth submersions $q_{i}:S^{1}\times B\rightarrow W\backslash\text{nbd}(\varSigma\cup\partial W)$
so that $q_{i}(1,b)=b$ and $q_{i}(-,b)$ is an immersion for all
$1\leq i\leq n$ and $b\in B$. Choose also a self dual two form $\omega$
whose support is contained in $B$. Then for $\mathbf{q}=(q_{1},\cdots,q_{n})$
we have a section 
\[
V_{\mathbf{q},\omega}:\mathcal{C}(W,\varSigma,\alpha)\rightarrow\varOmega^{2,+}(W;\mathfrak{g}_{E})
\]
given by 
\[
V_{\mathbf{q},\omega}(A)(x)=\text{Hol}_{\mathbf{q}}(A)\otimes\omega(x)
\]
Again, after introducing suitable completions one constructs a Banach
space of secondary holonomy perturbations with a notion of $L^{1}$
convergence \cite[p.51]{Miller[2019]}. Therefore, on the cobordism
with cylindrical ends $W^{*}$ the perturbed $\alpha$-ASD equations
we must consider are of the form 
\begin{equation}
F_{A}^{+}+U_{\mathfrak{p}}(A)+V_{\boldsymbol{\omega}}(A)=0\label{perturbed ASD}
\end{equation}
where the term $U_{\mathfrak{p}}(A)$ is generic notation for the
cylindrical holonomy perturbations and $V_{\boldsymbol{\omega}}(A)$
denotes the interior holonomy perturbations. The important property
we need to know about $U_{\mathfrak{p}}$ and $V_{\boldsymbol{\omega}}$
is that their $L^{\infty}$ norms are uniformly bounded: that is,
there exists $K>0$ such that 
\begin{itemize}
\item $\|U_{\mathfrak{p}}(A)\|_{L^{\infty}}\leq K\|\mathfrak{p}\|_{\mathcal{P}}$
(this is statement iii) in \cite[Proposition 3.7]{MR2860345}.
\item $\|V_{\boldsymbol{\omega}}(A)\|_{L^{\infty}}\leq K\|\boldsymbol{\omega}\|_{\tilde{\mathcal{P}}}$
(this is statement 2) in \cite[Proposition 4.4]{Miller[2019]}, which
first appeared in \cite[Section 3.2]{MR2192061}).
\end{itemize}
In fact, for transversality purposes, which is the reason why these
perturbations were introduced in the first place, for any given $\epsilon>0$,
we can assume that we chose perturbations $\mathfrak{p}$ and $\boldsymbol{\omega}$
satisfying $\|\mathfrak{p}\|_{\mathcal{P}},\|\boldsymbol{\omega}\|_{\tilde{\mathcal{P}}}<\epsilon$
, so in particular that will mean that 
\[
|f_{\mathfrak{p}}([\beta])|<\epsilon
\]
 for all $[\beta]\in\mathcal{B}(Y,K,\alpha)$. To bound $\mathcal{E}(z)$
from below, we follow the proof of \cite[Proposition 2.15]{Daemi[2019]}
and split $W^{*}$ into three regions 
\begin{align*}
 & \frac{1}{8\pi^{2}}\int_{W^{*}}\text{tr}(F(A)\wedge F(A))\\
= & \frac{1}{8\pi^{2}}\int_{\mathbb{R}^{-}\times\check{Y}_{1}}\text{tr}(F(A)\wedge F(A))+\frac{1}{8\pi^{2}}\int_{W}\text{tr}(F(A)\wedge F(A))+\frac{1}{8\pi^{2}}\int_{\mathbb{R}^{+}\times\check{Y}_{2}}\text{tr}(F(A)\wedge F(A))\\
\geq & \left[f_{\mathfrak{p}}([A\mid_{\{0\}\times\check{Y}_{1}}])-f_{\mathfrak{p}}([\beta_{1}])\right]+\frac{1}{8\pi^{2}}\int_{W}\left(|F^{-}(A)|^{2}-|F^{+}(A)|^{2}\right)+\left[f_{\mathfrak{p}}([\beta_{2}])-f_{\mathfrak{p}}([A\mid_{\{0\}\times\check{Y}_{2}}])\right]\\
\geq & -2\epsilon+f_{\mathfrak{p}}([\beta_{2}])-f_{\mathfrak{p}}([\beta_{1}])-\frac{1}{8\pi^{2}}\int_{W}|U_{\mathfrak{p}}(A)+V_{\boldsymbol{\omega}}(A)|^{2}\\
\geq & -2\epsilon+f_{\mathfrak{p}}([\beta_{2}])-f_{\mathfrak{p}}([\beta_{1}])-\frac{3\epsilon^{2}}{8\pi^{2}}
\end{align*}
In these steps we used the inequality \ref{inequality perturbed}
to deal the first and third integrals while we used the equation \ref{perturbed ASD}
on the second integral. The specific bound is not that important,
only knowing that it does not depend on the component of the moduli
space that is being analyzed. Therefore $m_{(W,\varSigma)}[\beta_{1}]$
will define an element of $\varLambda^{\mathbb{Q},\mathbb{R}}$ once
we know that the numbers $m_{z}([\beta_{1}],\check{W},[\beta_{2}])$
are well defined.

For that we follow \cite[Proposition 5.9]{MR1883043}. That is, we
want to show that the 0-dimensional moduli spaces $\mathcal{M}_{z}([\beta_{1}],\check{W},[\beta_{2}])$
with topological energy $\mathcal{E}(z)$ are compact. This is because
if we start with a sequence $[A_{i}]$ in $\mathcal{M}_{z}([\beta_{1}],\check{W},[\beta_{2}])$
(which is defined using a bundle $P([\beta_{1}],[\beta_{2}])$, then
it would converge weakly to

$\bullet$ An ideal instanton $([A_{\infty}],x_{\infty})$ on a bundle
$Q$ over $W^{*}$, asymptotic on each cylindrical end to $[\beta_{1}'],[\beta_{2}']$
respectively.

$\bullet$ A broken trajectory $([A_{1}],x_{1})$ over $\mathbb{R}\times Y_{1}$
connection $[\beta_{1}]$ and $[\beta_{1}']$, and a broken trajectory
$([A_{2}],x_{2})$ over $\mathbb{R}\times Y_{2}$ connecting $[\beta_{2}']$
and $[\beta_{2}]$ .

Additivity of the index says that (again because the critical points
are non-degenerate)
\begin{align}
0 & = & \text{ind}P([\beta_{1}],[\beta_{2}])\label{inequality indicies}\\
 & = & \text{ind}[A_{\infty}]+\text{ind}[A_{1}]+\text{ind}[A_{2}]+\dim\check{H}^{0}(Y_{1};\mathfrak{g}_{\beta_{1}'})+\dim\check{H}^{0}(Y_{2};\mathfrak{g}_{\beta_{2}'})+4(|x_{\infty}|+|x_{1}|+|x_{2}|)\nonumber 
\end{align}

$\bullet$ If $[A_{\infty}]$ is irreducible then by transversality
$\text{ind}[A_{\infty}]\geq0$ and since $[A_{1}],[A_{2}]$ are both
irreducible given that at least one of their limits is irreducible,
then $\text{ind}[A_{1}]\geq0,\text{ind}[A_{2}]\geq0$ by transversality.
Then the only way for the above equality to hold is if $[\beta_{1}]=[\beta_{1}']$,
$[\beta_{2}]=[\beta_{2}']$ and there were no bubbles, i.e, $|x_{\infty}|=|x_{1}|=|x_{2}|=0$.

$\bullet$ If $[A_{\infty}]$ is reducible then our assumptions on
homology  imply that $H_{1}(W\backslash\varSigma;\mathbb{Z})\simeq H_{1}(W^{*}\backslash\varSigma^{*};\mathbb{Z})\simeq\mathbb{Z}$
and because of the holonomy condition $[A_{\infty}]$ is determined
to be the unique (up to gauge) reducible with $S^{1}$ stabilizer
which is asymptotic to the reducibles $[\theta_{\alpha,Y_{1}}]$ and
$[\theta_{\alpha,Y_{2}}]$ respectively. In this case $\text{ind}[A_{\infty}]$
can be computed from the dimension formula \ref{dimension moduli space},
with the caveat that the formula as written only works assuming that
the limits are irreducible connections. When the limits are reducible,
one needs to take into account the stabilizer in the formula  and
one concludes that $\text{ind}[A_{\infty}]=-1$. Since $\dim\check{H}^{0}(Y_{1};\mathfrak{g}_{\theta_{\alpha,Y_{1}}})=\check{H}^{0}(Y_{2};\mathfrak{g}_{\theta_{\alpha,Y_{2}}})=1$
one finds in \ref{inequality indicies} that 
\[
-1=\text{ind}[A_{1}]+\text{ind}[A_{2}]+4(|x_{\infty}|+|x_{1}|+|x_{2}|)\geq0
\]
 which is impossible. Therefore, $m_{z}([\beta_{1}],\check{W},[\beta_{2}])$
is well defined. 

Finally, to verify the chain property one needs to compute 
\begin{align*}
 & \left(\partial_{Y_{2}}m_{W}+m_{W}\partial_{1}\right)[\beta_{1}]\\
= & \sum_{[\beta_{2}']}\sum_{z_{2}}\sum_{[\beta_{2}]}\sum_{z}m_{z}([\beta_{1}],\check{W},[\beta_{2}])n_{z_{2}}([\beta_{2}],[\beta_{2}'])T^{-\mathcal{E}_{top}(z)}T^{-\mathcal{E}_{top}(z_{2})}[\beta_{2}']\\
+ & \sum_{[\beta_{2}]}\sum_{w}\sum_{[\beta_{1}]}\sum_{z_{1}}n_{z_{1}}([\beta_{1}],[\beta_{1}'])m_{w}([\beta_{1}'],\check{W},[\beta_{2}])T^{-\mathcal{E}_{top}(z_{1})}T^{-\mathcal{E}_{top}(w)}[\beta_{2}]
\end{align*}
and show that it vanishes. Again, the usual argument will still work
as long as one remember that the topological energy is additive under
concatenation of paths.

Similar arguments can be applied for showing that the composition
law for $\alpha$-admissible cobordisms holds, and in this way we
have verified that $HI(Y,K,\alpha)$ is a topological invariant of
the data $(Y,K,\alpha)$, together with a choice of cone parameter
$\nu$, which has been omitted from our notation.

\section{\label{sec:Reduced-Version}Reduced Version}

Now we will define a reduced version $HI_{red}(Y,K,\alpha)$ of the
singular instanton Floer-Novikov homology for knots $HI(Y,K,\alpha)$
we just construced, following \cite{MR1910040}. As we also mentioned
in the introduction, the case of $\alpha=1/4$ , where local coefficients
are not needed, was defined by Daemi and Scaduto earlier \cite{Daemi-Scaduto[2019]}.

Our conventions will differ slightly from those of Frøyshov, since
we are defining the homology version of the Floer groups, which means
that our grading of the groups differ, and in fact are closer to the
ones used in \cite[Section 3.3.2]{MR1883043}, \cite[Section 9]{MR3394316}.The
key difference between the reduced and unreduced versions of instanton
Floer homologies is that the reduced version takes into account the
flow-lines between the critical points and the reducible flat connection.
Moreover, in the reduced version we can define a $U$-map, which in
the situation of knots will arise from the $\mu$-map evaluated at
a point $x\in K$. 

Now we want to define maps which take into account the interaction
with the reducible connection $[\theta_{\alpha}]$. Consider a critical
point $[\beta]$ with $\text{gr}([\beta])=1$. From the grading formula
(\ref{absolute grading}), we can see that in principle there are
non-empty moduli spaces $\mathcal{M}_{1}([\beta],[\theta_{\alpha}])$
asymptotic to $[\beta]$ as $t\rightarrow-\infty$ and to $[\theta_{\alpha}]$
as $t\rightarrow\infty$. After taking the quotient by the $\mathbb{R}$
action, we get a $0$-dimensional moduli space $\check{\mathcal{M}}_{1}([\beta],[\theta_{\alpha}])$.
We would like to define a map $\delta_{1}[\beta]$ obtained by counting
the points in this 0-dimensional moduli space. However, because of
non-monotonicity there can be a priori infinitely many components
of this moduli space
\[
\mathcal{M}_{1}([\beta],[\theta_{\alpha}])=\bigcup_{z}\mathcal{M}_{1,z}([\beta],[\theta_{\alpha}])
\]
Therefore, we define 
\begin{align*}
\delta_{1}:CI_{1}(Y,K,\alpha,\mathfrak{p})\rightarrow\varLambda\\{}
[\beta]\rightarrow\sum_{z}\#\check{\mathcal{M}}_{1,z}([\beta],[\theta_{\alpha}])T^{-\mathcal{E}_{top}(z)}
\end{align*}
Recall that $\varLambda$ is the Novikov field, while $CI_{1}(Y,K,\alpha,\mathfrak{p})$
denotes the chain complex of $\mathfrak{p}$-perturbed flat connections
whose absolute grading is $1$. The formula for $\delta_{1}[\beta]$
does give a well defined element in $\varLambda$ for exactly the
same reasons as the differential $\partial$ from the previous section
being well defined. Just as in the case where $K$ is absent, it
is straightforward to see that $\delta_{1}$ descends to a map in
homology, that is:
\begin{lem}
The map $\delta_{1}$ satisfies $\delta_{1}\partial=0$, so it induces
a map in homology $\delta_{1}:HI_{1}(Y,K,\alpha)\rightarrow\varLambda$.
\end{lem}

Likewise, suppose that $\text{gr}([\beta])=2$, so that a priori there
are non-empty moduli spaces $\mathcal{M}_{1}([\theta_{\alpha}],[\beta])$.
Analogous to $\delta_{1}$, define an element $\delta_{2}\in CI_{2}(Y,K,\alpha,\mathfrak{p})$
by the formula 
\[
\delta_{2}=\sum_{[\beta]\in CI_{2}}\sum_{z}\#\check{\mathcal{M}}_{1,z}([\theta_{\alpha}],[\beta])T^{-\mathcal{E}_{top}(z)}[\beta]
\]
As before, it is straightforward to check that:
\begin{lem}
The element $\delta_{2}$, descends to an element in homology, i.e,
$\partial\delta_{2}=0$ so that $\delta_{2}\in HI_{2}(Y,K,\alpha)$.
\end{lem}

The next map to define is the $\mu$ -map, which we will denote $\mu_{K}$,
to emphasize the fact that it is not the ordinary $\mu$-map. First
we need to understand the homotopy type of the space of connections
mod gauge, i.e, $\mathcal{B}(Y,K,\alpha)$, since for $x\in K$, $\mu_{K}(x)$
will be a degree 2 element in $H^{*}(\mathcal{B}(Y,K,\alpha);\mathbb{Q})$.

The main idea is to take advantage of the fact that singular connections
have a stronger notion of framing than ordinary connections. We will
follow the discussion in \cite[Section 5]{MR1284567} and \cite[Section 4]{MR1432428}.
Recall that if $G$ is a compact Lie group that acts on a topological
space $Z$, then the homotopy quotient $Z//G$ is defined as $Z\times_{G}EG$,
where $EG$ is a contractible space with a free $G$ action. The natural
map $Z//G\rightarrow Z/G$ induces a map $H^{*}(Z/G,\mathbb{Z})\rightarrow H^{*}(Z//G,\mathbb{Z})$,
which is an isomorphism when $G$ acts freely. Lemma 5.1 in \cite{MR1284567}
shows the following:
\begin{lem}
If $U(1)$ acts on $Z$ and the stabilizer of every point in $Z$
is $\{\pm1\}$, then the pull-back map $H^{*}(Z/U(1),\mathbb{Q})\rightarrow H^{*}(Z//U(1),\mathbb{Q})$
is an isomorphism.
\end{lem}

When $V$ is a complex vector bundle over $Z$ with a lift of the
$G$ action to $V$, we can also define $V//G=V\times_{G}EG$ and
the $G$-equivariant Chern classes of $V$ as $c_{i,G}(V)=c_{i}(V//G)\in H^{2i}(Z//G,\mathbb{Z})$.
These are the pull-backs of the Chern classes on $Z/G$.

For our setup, when we are working with the pair $(X,\varSigma)$,
the bundle $E$ decomposes near $\varSigma$ as $E=L\oplus L^{-1}$.
Moreover, we wrote an exact sequence \ref{eq: exact sequence gauge group}
for the gauge group, which in particular implies that over $\varSigma$
the $SU(2)$ gauge is broken to an $U(1)$ gauge. If $x\in\varSigma$
is a base-point and $\mathcal{G}_{x}$ the gauge transformations which
act trivially on $E_{x}=L_{x}\oplus L_{x}^{-1}$, then we have the
framed configuration space 
\[
\mathcal{B}^{o}(X,\varSigma,\alpha)=\mathcal{C}(X,\varSigma,\alpha)/\mathcal{G}_{x}(X,\varSigma)
\]
Since $\mathcal{G}/\mathcal{G}_{x}\simeq U(1)$, a residual gauge
group isomorphic to $U(1)$ acts on $\mathcal{B}^{o}(X,\varSigma,\alpha)$,
and $\mathcal{B}^{o}(X,\varSigma,\alpha)/U(1)=\mathcal{B}(X,\varSigma,\alpha)$.
Now we are ready to define the universal $SO(3)$ bundle and the corresponding
$\mu$ map.
\begin{defn}
Define the \textbf{universal $SO(3)$ bundle 
\[
\mathbb{E}^{ad}=\mathcal{C}(X,\varSigma,\alpha)\times_{\mathcal{G}_{x}}\mathfrak{g}_{E}\rightarrow(\mathcal{B}^{*}(X,\varSigma,\alpha)\times X)
\]
}Moreover, the $U(1)$ bundle
\begin{equation}
\mathbb{L}^{o}=\mathcal{C}(X,\varSigma,\alpha)\times_{\mathcal{G}_{x}}L\rightarrow\mathcal{B}^{o}(X,\varSigma,\alpha)\times\varSigma\label{universal line bundle}
\end{equation}
 descends to the \textbf{universal $U(1)$ bundle }
\[
\mathbb{L}^{\otimes2}\rightarrow(\mathcal{B}^{*}(X,\varSigma,\alpha)\times\varSigma)
\]
Define for $\eta\in H_{i}(X;\mathbb{Q})$ and $\eta_{\varSigma}\in H_{j}(\varSigma;\mathbb{Q})$
the \textbf{$\mu$-maps}
\begin{align*}
\mu(\eta)=-\frac{1}{4}p_{1}(\mathbb{E}^{ad})/\eta\in H^{4-i}(\mathcal{B}^{*}(X,\varSigma,\alpha);\mathbb{Q})\\
\mu_{\varSigma}(\eta_{\varSigma})=-\frac{1}{2}e(\mathbb{L}^{\otimes2})/\eta_{\varSigma}\in H^{2-j}(\mathcal{B}^{*}(X,\varSigma,\alpha);\mathbb{Q})
\end{align*}
\end{defn}

\begin{rem}
1) We follow the sign conventions of Kronheimer for $\mu_{\varSigma}$
\cite[Section 2.1]{MR1432428}.

2) Notice that since the homotopy type of $\mathcal{B}^{*}(X,\varSigma,\alpha)$
is independent of $\alpha$, the $\mu$-maps corresponding to different
values of $\alpha$ can be identified with each other, which is why
we do not indicate the value of $\alpha$ in our notation.

3) We can also view these as elements in $\mathcal{B}(X,\varSigma,\alpha)$,
since the reducible connections form a stratum of infinite codimension
in $\mathcal{B}(X,\varSigma,\alpha)$.

4) Notice that our construction singles out one of the line bundles
in the decomposition $E=L\oplus L^{-1}$. In fact, had we used $L^{-1}$
instead of $L$, then $\mu_{\varSigma}$ would differ only in sign.
\end{rem}

Clearly a similar procedure can be used to define $\mu_{K}$ in the
case of $(Y,K)$: a cheap way to do this is to consider $X=[0,1]\times Y$
and $\varSigma=[0,1]\times K$. Therefore, for $x\in K$ we let
\begin{align*}
u_{K}(x):CI_{*}(Y,K,\alpha,\mathfrak{p})\rightarrow CI_{*-2}(Y,K,\alpha,\mathfrak{p})\\{}
[\beta_{0}]\rightarrow\sum_{[\beta_{1}]\mid\text{gr}([\beta_{0}],[\beta_{1}])=2}\sum_{z}<u_{K}(x),\mathcal{M}_{2,z}([\beta_{0}],[\beta_{1}])>T^{-\mathcal{E}_{top}(z)}[\beta_{1}]
\end{align*}
Contrary to the case of $\delta_{1},\delta_{2}$, $u_{K}(x)$ will
not descend to a map between the Floer homology groups. To see why
this is the case, observe that the maps
\[
\partial u_{K}-u_{K}\partial
\]
involve considering three dimensional moduli spaces, for which we
had said factorizations through the reducibles can occur. More precisely,
notice that for any $[\beta_{0}]$, we have {[}here we denote for
convenience $<u_{K}(x),\mathcal{M}_{2,z}([\beta_{0}],[\beta_{1}])>$
as $U_{z}([\beta_{0}],[\beta_{1}])${]}
\begin{align}
 & (\partial u_{K}-u_{K}\partial)[\beta_{0}]\label{chain map u}\\
= & \sum_{[\beta_{2}]\mid\text{gr}([\beta_{2}],[\beta_{1}])=1}\sum_{z_{12}}\sum_{[\beta_{1}]\mid\text{gr}([\beta_{0}],[\beta_{1}])=2}\sum_{z_{01}}U_{z_{01}}([\beta_{0}],[\beta_{1}])n_{z_{12}}([\beta_{1}],[\beta_{2}])T^{-\mathcal{E}_{top}(z_{01})}T^{-\mathcal{E}_{top}(z_{12})}[\beta_{2}]\nonumber \\
 & -\sum_{[\beta_{2}]\mid\text{gr}([\beta_{2}],[\beta_{1}])=2}\sum_{w_{12}}\sum_{[\beta_{1}]\mid\text{gr}([\beta_{1}],[\beta_{0}])=1}\sum_{w_{01}}n_{w_{01}}([\beta_{0}],[\beta_{1}])U_{w_{12}}([\beta_{1}],[\beta_{2}])T^{-\mathcal{E}_{top}(w_{01})}T^{-\mathcal{E}_{top}(w_{12})}[\beta_{2}]\nonumber 
\end{align}
The typical argument would look at a three dimensional moduli space
$\mathcal{M}_{3}([\beta_{0}],[\beta_{2}])$ and consider the possible
ends of this moduli space. Some of the ends correspond to the terms
in \ref{chain map u}, but when $\text{gr}[\beta_{0}]=1$, a priori
it is also possible to have factorizations of the form 
\[
\mathcal{M}_{1}([\beta_{0}],[\theta_{\alpha}])\times\mathcal{M}_{1}([\theta_{\alpha}],[\beta_{2}])
\]
which needs to be accounted for. In fact, we have the analogue of
\cite[Theorem 4]{MR1910040}, \cite[Proposition 8]{MR2738582} and
\cite[Lemma 7.6]{MR1883043}.
\begin{lem}
The map $u_{K}$ satisfies the relation 
\begin{equation}
\partial u_{K}-u_{K}\partial-\frac{1}{2}\delta_{1}\otimes\delta_{2}=0\label{relation u-map}
\end{equation}
\end{lem}

\begin{proof}
Of the references cited above, the closest to our argument is in fact
\cite[Proposition 8]{MR2738582}, since the monopole case also uses
a universal $U(1)$ bundle to define the corresponding $u$-map. Therefore
we obtain the same formula as the one Frøyshov writes for the monopole
case, except for the difference in conventions for the constants in
front of the $\mu$-maps.

In fact, the only place where one needs to be careful with the previous
argument is that any of the proofs quoted above use the holonomy of
a connection $A$ along the path $\mathbb{R}\times\{x\}\subset\mathbb{R}\times Y$.

More precisely, consider a 3-dimensional moduli space $\mathcal{M}_{3,z}([\beta_{0}],[\beta_{2}])$,
and choose a representative $A\in\mathcal{M}_{3,z}([\beta_{0}],[\beta_{1}])$
whose ``centre of mass'' is $0$, i.e, 
\[
\int_{\mathbb{R}\times Y}t|F_{A}|^{2}=0
\]
If $\text{ad}\beta_{0},\text{ad}\beta_{1}$ are the corresponding
(perturbed) flat $SO(3)$ bundles corresponding to $[\beta_{0}],[\beta_{1}]$,
then we can choose a base point $y\in Y$, which is \emph{close} to
$x\in K$, without being equal to it. Using a normal neighborhood
$\nu(K)$ of $K$, we may assume that $y$ belongs to a normal disk
to $K$, centered at $x$, for which $y$ has polar coordinates $(r,\theta)$.
Since we are away from the knot, there is no controversy as to what
we mean by 
\[
h_{A}(r,\theta)=\text{hol}_{A}\left(\mathbb{R}\times\{y\}\right)
\]
That is, the holonomy of $A$ along the path $\mathbb{R}\times\{y\}$,
where $y$ has coordinates $(r,\theta)$. Comparing the frames for
the fibres of the bundles $\text{ad}\beta_{0},\text{ad}\beta_{1}$,
we get an element in $SO(3)$, i.e, $h_{A}(r,\theta)\in SO(3)$. Now,
for fixed $\theta$, as $r$ decreases the decomposition $E=L\oplus L^{-1}$
, $\text{ad}E=\mathbb{R}\oplus L^{\otimes2}$, becomes asymptotically
parallel with respect to $A$. Therefore, we obtain an element 
\[
h_{A}(\theta)=\lim_{r\rightarrow0}\text{hol}_{A}\left(\mathbb{R}\times\{y\}\right)\in U(1)
\]
which is obtained by comparing the frames for the fibres of the $U(1)$
bundles $\widetilde{\text{ad}\beta_{0}},\widetilde{\text{ad}\beta_{1}}$,
over the $U(1)$ line bundle $\mathbb{L}^{\otimes2}\rightarrow\varSigma$
{[}compare with the description of the universal bundle \ref{universal line bundle}{]}.
The existence of this limit follows for example from \cite{MR1152376},
or one can also use the fact that we are working with orbifold connections,
as Kronheimer and Mrowka do in \cite[Section 3.1]{MR3956896}.

However, we still need to analyze what happens as we vary the angle
at which we approach the point $x$. It is not difficult to see that
as we vary the angle by a full revolution, i.e, $\theta\rightarrow\theta+2\pi$,
then the holonomy picks out $h_{A}(\theta)$ picks out the asymptotic
holonomy factor $e^{-4\pi i\alpha}$, in other words 
\[
h_{A}(\theta+2\pi)=e^{-4\pi i\alpha}h_{A}(\theta)
\]
So in the case of rational holonomy, we can take 
\[
\begin{cases}
h_{A}(x)\equiv[h_{A}(\theta)]^{q} & \alpha=\frac{2p+1}{2q}\\
h_{A}(x)\equiv[h_{A}(\theta)]^{2q+1} & \alpha=\frac{2p}{2q+1}
\end{cases}
\]
as our desired holonomy map. For example, the case of $\alpha=\frac{1}{4}=\frac{1}{2q}$
implies that we should square the limiting holonomy maps, which is
exactly what Kronheimer and Mrowka do in \cite[Section 3.1]{MR3956896}.
The difference with their construction is mainly stylistically, since
they pass first to a local cover of a neighborhood of $x$, where
the pull-backs of the connections extend smoothly.

Once we know how to take the holonomy along a point on the knot $x\in K$,
the proof follows in exactly the same way as in \cite[Proposition 8]{MR2738582},
where the coefficient of $\delta_{1}\otimes\delta_{2}$ is the Euler
number of the rank $1$ Hermitian vector bundle over $S^{2}=D^{2}\cup_{S^{1}}D^{2}$
whose ``clutching map'' $S^{1}\rightarrow U(1)$ has degree $1$.
\end{proof}

\textbf{}

Now we interpret the equation \ref{relation u-map} to find out how
the reduced Floer groups should be defined.
\begin{itemize}
\item Case when $[\beta]\in HI_{1}(Y,K,\alpha)$ and $\delta_{1}[\beta]=0$:
then $\partial u_{K}[\beta]-u_{K}\partial[\beta]=0$, which means
that $u_{K}$ descends to a map 
\[
u_{K}:\ker\delta_{1}\subset HI_{1}(Y,K,\alpha)\rightarrow HI_{-1}(Y,K,\alpha)=HI_{3}(Y,K,\alpha)
\]
\end{itemize}
\textbf{}
\begin{itemize}
\item Case when $[\beta_{0}]\in HI_{0}(Y,K,\alpha)$: abusing notation write
a representative of $[\beta_{0}]$ as $[\beta]+\partial[\beta_{1}]$,
where $\partial[\beta]=0$ and $[\beta_{1}]\in HI_{1}(Y,K,\alpha)$.
Then 
\[
u_{K}[\beta_{0}]=u_{K}[\beta]+u_{K}\partial[\beta_{1}]=u_{K}[\beta]+\partial u_{K}[\beta_{1}]-\frac{1}{2}\delta_{1}([\beta_{1}])\delta_{2}
\]
Therefore, we must identify elements which differ by an element on
the ``ray'' $\varLambda\delta_{2}$ , in other words, we get a map
\[
u_{K}:HI_{0}(Y,K,\alpha)\rightarrow\text{coker}(\delta_{2})=HI_{2}(Y,K,\alpha)/(\varLambda\delta_{2})
\]
\end{itemize}

\begin{itemize}
\item On the other summands $HI_{2}(Y,K,\alpha)$ and $HI(Y,K,\alpha)$
, the $u_{K}$ map is actually well defined without additional considerations
so we get maps 
\[
\begin{cases}
u_{K}:HI_{2}(Y,K,\alpha)\rightarrow HI_{0}(Y,K,\alpha)\\
u_{K}:HI_{3}(Y,K,\alpha)\rightarrow HI_{1}(Y,K,\alpha)
\end{cases}
\]
\end{itemize}
The definition the reduced Floer homology groups is now identical
to \cite[Definition 1]{MR1910040}, since as long as we work with
$\alpha$-admissible cobordism it is straightforward to adapt section
$3$ of \cite{MR1910040} (which analyzes the behavior of $\delta_{1},\delta_{2},u_{K}$
under cobordisms) to obtain:
\begin{thm}
Let $K\subset Y$ be a knot and $\alpha\in\mathbb{Q}\cap(0,1/2)$
be such that $\triangle_{K}(e^{-4\pi i\alpha})\neq0$. The \textbf{reduced
Instanton Floer homology groups} $HI_{i}^{red}(Y,K,\alpha)$ 
\begin{align}
HI_{0}^{red}(Y,K,\alpha) & = & HI_{0}(Y,K,\alpha)/(\sum\text{im}(u_{K}^{2l+1}\delta_{2})\label{reduced Floer groups}\\
HI_{1}^{red}(Y,K,\alpha) & = & \cap_{l\geq0}\ker(\delta_{1}u_{K}^{2l})\subset HI_{1}(Y,K,\alpha)\nonumber \\
HI_{2}^{red}(Y,K,\alpha) & = & HI_{2}(Y,K,\alpha)/(\sum\text{im}(u_{K}^{2l}\delta_{2}))\nonumber \\
HI_{3}^{red}(Y,K,\alpha) & = & \cap_{l\geq0}\ker(\delta_{1}u_{K}^{2l+1})\subset HI_{3}(Y,K,\alpha)\nonumber 
\end{align}
are topological invariants of the data $(Y,K,\alpha)$, together with
the choice of cone parameter $\nu$ used to define the groups. 

Moreover, the \textbf{Frøyshov knot invariants }
\begin{equation}
h(Y,K,\alpha)=\chi_{\varLambda}(HI^{red}(Y,K,\alpha))-\chi_{\varLambda}(HI(Y,K,\alpha))\in\mathbb{Z}\label{eq:h-invariant knot}
\end{equation}
where $\chi_{\varLambda}$ denotes the Euler characteristic with respect
to the Novikov field $\varLambda$, are also invariants of $(Y,K,\alpha)$
(and the cone parameter).
\end{thm}

\begin{rem}
The usual $h$-invariants have a factor $\frac{1}{2}$ in front of
the difference in Euler characteristics. We choose not to include
this factor because in our case the Euler characteristics of the groups
may be odd, since these groups need not be $2$-periodic as opposed
to the $4$-periodic instanton Floer homology groups on $Y$, and
we prefer to obtain an integer rather than a half-integer.
\end{rem}

\section{\label{sec:Singular-Orbifold-Furuta-Ohta}Singular Orbifold Furuta-Ohta
and Tori Signature}

In this section we define an analogue of the Furuta-Ohta invariant
$\lambda_{FO}(X)$ \cite{MR1237394} to the case of an embedded torus
$T$ inside $X$ satisfying certain topological conditions.

In other words, we want to define an invariant $\lambda_{FO}(X,T,\alpha)$,
which in the best case scenario can be interpreted as a signed count
of irreducible representations $\pi_{1}(X\backslash T)\rightarrow SU(2)$
(modulo conjugacy) satisfying a certain holonomy condition determined
by the parameter $\alpha$. As usual, perturbations of the flatness
equation will be needed, so the interpretation of $\lambda_{FO}(X,T,\alpha)$
as a count of flat connections is slightly more complicated. In any
case, our construction will be cooked up in such a way that when we
take $X=S^{1}\times Y$ and $T=S^{1}\times K$ then $\lambda_{FO}(X,T,\alpha)$
agrees with $2\lambda_{CH}(Y,K,\alpha)$.

First we want to explain why we consider only torus complements $X\backslash T$,
and not more general surface complements $X\backslash\varSigma$.
Moreover, we will discuss what conditions are needed to make the definition
of $\lambda_{FO}(X,T,\alpha)$ work.

Recall from section \ref{sec:Review-of-the} that $SU(2)$ bundles
$E\rightarrow X$ are characterized by two topological invariants,
the instanton number $k$ and the monopole number $l$. The moduli
space of $\alpha$-ASD connections $\mathcal{M}(X,\varSigma,k,l,\alpha)$
satisfying the asymptotic condition \ref{asymptotic holonomy} has
the expected dimension \cite[Eq 1.6]{MR1241873}
\begin{equation}
\dim\mathcal{M}(X,\varSigma,k,l,\alpha)=8k+4l-3(b_{2}^{+}-b^{1}+1)-(2g-2)\label{eq:dimension moduli}
\end{equation}
while the formula for the topological energy is \cite[Eq. 1.7]{MR1241873}
\begin{equation}
\mathcal{E}(X,\varSigma,k,l,\alpha)=\frac{1}{8\pi^{2}}\int_{\check{X}}\text{tr}(F_{A}\wedge F_{A})=k+2\alpha l-\alpha^{2}\varSigma\cdot\varSigma\label{eq:energy moduli}
\end{equation}
Since $H_{*}(X;\mathbb{Z})\simeq H_{*}(S^{1}\times S^{3};\mathbb{Z})$
these formulas simplify to 
\begin{align}
\dim\mathcal{M}(X,\varSigma,k,l,\alpha)=8k+4l-(2g-2)\label{simplified dimension and energy}\\
\mathcal{E}(X,\varSigma,k,l,\alpha)=k+2\alpha l\nonumber 
\end{align}
Notice that $\alpha$-flat connections (i.e, flat connections satisfying
\ref{asymptotic holonomy}) are equivalent to energy zero $\alpha$-ASD
instantons (i.e, ASD instantons satisfying \ref{asymptotic holonomy}).
In particular, this means that $\alpha$-flat connections only exist
a priori on bundles whose monopole and instanton numbers are related
as 
\[
k+2\alpha l=0
\]
Clearly $k=l=0$ is always a solution of this equation. Since $k,l$
must always be integers, for irrational values of $\alpha$, $k=l=0$
is the only solution. In fact, the next lemma shows that this continue
to hold regardless of the value of $\alpha$.
\begin{lem}
\label{lem: existence of flat connection} Suppose that $E$ is an
$SU(2)$ bundle over $(X,\varSigma)$ with instanton and monopole
numbers $(k,l)$. If $E$ supports an $\alpha$-flat connection then
$(k,l)=(0,0)$ and thus $E$ is the trivial bundle over $X$.
\end{lem}

\begin{proof}
As discussed before, any $\alpha$-flat connection can only exist
on a bundle $E$ for which 
\[
k+2\alpha l=0
\]
If $\alpha$ is irrational we already explained that $k=l=0$ is automatic.

For $\alpha$ rational, it suffices to show that $l$ vanishes, since
the previous equation will force $k$ to vanish and we will be done.
To understand why $l$ vanishes, we must use the fact that $l$ can
be computed as \cite[Eq 17]{MR1432428}
\begin{equation}
l=\lambda+\alpha\varSigma\cdot\varSigma\label{computation of l}
\end{equation}
Here we use the fact the orbifold connection has a locally well defined
restriction as an abelian connection on $\varSigma$, so the curvature
decomposes as 
\[
F_{A}=\left(\begin{array}{cc}
\omega & 0\\
0 & -\omega
\end{array}\right)
\]
when regarded as a 2-form on $\varSigma$. The quantity $\lambda$
can then be computed as 
\[
\lambda=\frac{i}{2\pi}\int_{\varSigma}(-\omega)
\]
In the case of a flat connection it is clear that $\omega=0$ thus
$\lambda=0$. Since $\varSigma\cdot\varSigma=0$ equation \ref{computation of l}
now implies that $l=0$, as desired.
\end{proof}
This means that for counting $\alpha$-flat connections we can simply
concentrate on the case $k=l=0$, for which the expected dimension
of the moduli space is \ref{simplified dimension and energy}
\[
\dim\mathcal{M}(X,\varSigma,0,0,\alpha)=-2(g-1)
\]
So it is now clear that the expected dimension of the moduli space
of $\alpha$-flat connections is zero dimensional if and only if $g=1$,
i.e, $\varSigma$ must be an embedded torus $T$. In this case the
energy and dimension formulas become
\begin{align}
\dim\mathcal{M}(X,T,k,l,\alpha)=8k+4l\label{simplified dimension and energy torus}\\
\mathcal{E}(X,T,k,l,\alpha)=k+2\alpha l\nonumber 
\end{align}

Our next objective will be to analyze which further hypothesis on
$T$ must be made in order to have a well defined invariant $\lambda_{FO}(X,T,\alpha)$. 

But first, we also analyze what are the possible $\alpha$-ASD instantons
which are reducible on $(X,\varSigma)$.
\begin{lem}
\label{reducible connection k,l=00003D0}Suppose that $E$ is an $SU(2)$
bundle over $(X,\varSigma)$ with instanton and monopole numbers $(k,l)$.
If $E$ supports an $\alpha$-ASD connection which is reducible then
$(k,l)=(0,0)$ and thus $E$ is the trivial bundle over $X$.
\end{lem}

\begin{proof}
We follow the remarks after Proposition 5.9 in \cite{MR1241873},
more specifically the remark ``iii) Transversality''. For an $\alpha$-reducible
$ASD$ connection the bundle $E$ must first of all split globally
as $E=L\oplus L^{-1}$, where $L$ is a complex line bundle which
admits an $\alpha$-ASD connection. This in turn can be represented
by a smooth, harmonic, anti-self-dual 2-form $\omega$ whose cohomology
class represents $c_{1}(L)+\alpha[\varSigma]$. Again, because $b_{2}(X)=0$
this means in fact that $\omega$ represents the class $0$, i.e,
it must be an $\alpha$-flat twisted connection. But we already know
from Lemma \ref{lem: existence of flat connection} that this forces
$k$ and $l$ to be both zero so we are done.
\end{proof}
Back to the case where $\varSigma=T$, we need to discuss under which
conditions we can expect to define a count of $\alpha$-representations
$\pi_{1}(X\backslash T)\rightarrow SU(2)$. As in the case of a knot
$K$ inside a homology sphere $Y$, we have to guarantee that the
$\alpha$-reducible representations are isolated from the $\alpha$-irreducible
representations. To analyze the $\alpha$-reducible representations
we need to consider the first homology of the torus complement, i.e,
$H_{1}(X\backslash T;\mathbb{Z})$. Now, $H_{1}(X\backslash T;\mathbb{Z})$
will be sensitive on the embedding of the torus, for example, $T$
could be null-homologous or not. In fact, we are interested in the
case where $T$ is homologically indistinguishable from the product
situation $S^{1}\times K\subset S^{1}\times Y$, so we will make the
following assumption.
\begin{assumption}
$T$ will be an embedded torus inside $X$ such that the natural map
$H_{1}(T;\mathbb{Z})\twoheadrightarrow H_{1}(X;\mathbb{Z})$ is a
surjection.
\end{assumption}

More precisely, if we take our torus $T$ to be the image of an embedding
$\imath_{T}:S^{1}\times S^{1}\hookrightarrow X$, then we will assume
that $\vartheta=(\imath_{T})_{*}(S^{1}\times\{pt\})$ is a generator
of $H_{1}(X;\mathbb{Z})$. A tubular neighborhood of $T$ looks like
$T\times D^{2}$, and if we use polar coordinates $(r,\theta)$ for
the second factor, then 
\[
H_{1}(X\backslash T;\mathbb{Z})\simeq\mathbb{Z}[\vartheta]\oplus\mathbb{Z}[\mu_{T}]
\]
where $\mu_{T}$ is a ``meridian'' for the torus $T$, and can be
represented for example as $(\imath_{T\times D^{2}})_{*}(\{pt\}\times S^{1}\times(\epsilon,0))$,
where $\epsilon$ is sufficiently small.

If $\rho$ is an $\alpha$ flat connection then we want to understand
the Zariski tangent space $\check{H}^{1}(\check{X};\mathfrak{g}_{\rho})$.
Just as in Lemma \ref{orbifold cohomology} we have:
\begin{lem}
\label{orbifold cohomology group torus}Suppose $\alpha\in\mathbb{Q}\cap(0,1/2)$
and that $\rho$ is an $\alpha$-flat connection on the orbifold $\check{X}$.
Then the first (orbifold) cohomology $\check{H}^{1}(\check{X};\mathfrak{g}_{\rho})$
can be identified with $\ker:H^{1}(X\backslash T;\mathfrak{g}_{\rho})\rightarrow H^{1}(\mu_{T};\mathfrak{g}_{\rho})$.
\end{lem}

\begin{proof}
The argument is completely analogous. Namely, after applying a Mayer-Vietoris
decomposition to 
\[
\check{X}=(\check{X}\backslash\check{\nu}_{\epsilon}(T))\cup\check{\nu}(T)
\]
we end up with the analogue of equations \ref{Map Mayer 1} and \ref{Map mayer 2}
\[
\check{H}^{1}(\check{X};\mathfrak{g}_{\rho})\rightarrow^{(\imath_{X\backslash T}^{*},\imath_{\nu(T)}^{*})}H^{1}(X\backslash T;\mathfrak{g}_{\rho})\oplus\left(\mathbb{R}[\lambda_{T}]\oplus\mathbb{R}[\vartheta]\right)\rightarrow^{i_{T_{\epsilon},\nu(K)}^{*}-i_{T_{\epsilon,Y\backslash K}}^{*}}\mathbb{R}[\mu_{T}]\oplus\mathbb{R}[\lambda_{T}]\oplus\mathbb{R}[\vartheta]
\]
The composition being $0$ now says that $\check{\omega}\rightarrow\left\langle \omega\mid_{S_{\mu_{T}}^{1}},[\mu_{T}]\right\rangle $
vanishes, so there is a map $\check{H}^{1}(\check{X};\mathfrak{g}_{\rho})\rightarrow\ker(H^{1}(X\backslash T;\mathfrak{g}_{\rho})\rightarrow H^{1}(\mu;\mathfrak{g}_{\rho}))$.
The surjectivity and injectivity of this map are proven in exactly
the same way as before.
\end{proof}
In particular, for an $\alpha$-reducible representation $\rho$ we
have $\mathfrak{g}_{\rho}=\mathbb{R}\oplus L^{\otimes2}$ and thus
\begin{align*}
 & \check{H}^{1}(\check{X};\mathfrak{g}_{\rho_{\alpha}})\\
\simeq & \ker\left(H^{1}(X\backslash T;\mathbb{R})\oplus H^{1}(X\backslash T;L^{\otimes2})\rightarrow H^{1}(\mu_{T};\mathbb{R})\oplus H^{1}(\mu_{T};L^{\otimes2})\right)\\
\simeq & \ker\left(\mathbb{R}[\vartheta]\oplus\mathbb{R}[\mu_{T}]\oplus H^{1}(X\backslash T;L^{\otimes2})\rightarrow\mathbb{R}[\vartheta]\right)\\
\simeq & \mathbb{R}[\mu_{T}]\oplus H^{1}(X\backslash T;L^{\otimes2})
\end{align*}
So at an $\alpha$-reducible representation $\rho$, $\check{H}^{1}(\check{X};\mathfrak{g}_{\rho_{\alpha}})$
is at least one dimensional, and for $\rho$ to be isolated from the
irreducible representations it suffices to assume that $H^{1}(X\backslash T;L^{\otimes2})$
vanishes. Using \cite[Corollary 65]{Echeverria[DraftFuruta]}, we
can characterize this in terms of the Alexander polynomial of the
torus complement $X\backslash T$.
\begin{cor}
Let $\alpha\in\mathbb{Q}\cap(0,1/2)$ and $T$ be an embedded oriented
torus such that $H_{1}(T;\mathbb{Z})\twoheadrightarrow H_{1}(X;\mathbb{Z})$.
Then $\check{H}^{1}(\check{X};\mathfrak{g}_{\rho})$ is one dimensional
for every $\alpha$-reducible representation $\rho$ if and only if
$H^{1}(X\backslash T;L_{\rho}^{\otimes2})=0$, where $\mathfrak{g}_{\rho}=\mathbb{R}\oplus L_{\rho}^{\otimes2}$.
Equivalently, for every $\alpha$-reducible representation we have
$\triangle_{X\backslash T}(\hat{\rho})\neq0$, where $\hat{\rho}\in\widehat{\pi_{1}(X\backslash T)}$
is the character determined by the local system $H^{1}(X\backslash T;L^{\otimes2})$.
\end{cor}

\begin{rem}
For the examples of embedded tori we will analyze, we find it easier
to verify directly that $H^{1}(X\backslash T;L^{\otimes2})=0$ vanishes,
rather than using the condition on the Alexander polynomial, since
our constructions will arise from some operation on a knot inside
an integer homology sphere. It would be interesting to study an example
of an embedded torus where it is easier to verify the condition on
the Alexander polynomial directly. It would probably need to be more
four dimensional in nature.
\end{rem}

Notice that whenever the previous condition is satisfied then the
obstruction space (i.e, the second cohomology group $\check{\mathcal{H}}^{2}(\check{X};A_{\rho})$
of the deformation complex defined by $\rho$ as a solution of the
$\alpha$-ASD equations) must vanish. This is because the Euler characteristic
of the deformation complex is (minus) the virtual dimension of the
moduli space, which is zero in our situation, so 
\[
\dim\check{\mathcal{H}}^{0}(\check{X};A_{\rho})-\dim\check{\mathcal{H}}^{1}(\check{X};A_{\rho})+\dim\check{\mathcal{H}}^{2}(\check{X};A_{\rho})=0
\]
The first factor is $1$ because that is the dimension of the stabilizer
of $A_{\rho}$ and the one in the middle is also $1$ by assumption
so that forces $\check{\mathcal{H}}^{2}(\check{X};A_{\rho})$ to vanish.

Now we are finally ready to give a definition $\lambda_{FO}(X,T,\alpha)$.

\begin{defn}
Suppose that $\alpha\in\mathbb{Q}\cap(0,1/2)$ and $T$ be an oriented
embedded torus such that $H_{1}(T;\mathbb{Z})\twoheadrightarrow H_{1}(X;\mathbb{Z})$.
Suppose moreover that for every $\alpha$-reducible representation
$\rho$ we have $H^{1}(X\backslash T;L^{\otimes2})=0$, where $\mathfrak{g}_{\rho}=\mathbb{R}\oplus L^{\otimes2}$.
Choose a homology orientation for $(X,T)$, that is, an orientation
of $H^{1}(X;\mathbb{Z})$, which in turn is determined by the orientation
of $T$. Given this homology orientation, there is an orientation
of the moduli spaces $\mathcal{M}(X,T,k,l,\alpha)$ \cite[Section 2.i)]{MR1308489}.
We define the \textbf{singular Furuta-Ohta invariant }$\lambda_{FO}(X,T,\alpha)$
as follows.

Choose the trivial $SU(2)$ bundle $E\rightarrow X$ corresponding
to the instanton and monopole numbers $k=l=0$. Then, after perturbations
if necessary, the irreducible $\alpha$-ASD connections $\mathcal{M}^{*}(X,T,0,0,\alpha)$
will form a 0-dimensional compact moduli space. We define $\lambda_{FO}(X,T,\alpha)\in\mathbb{Z}$
as the signed count of elements inside $\mathcal{M}^{*}(X,T,0,0,\alpha)$.
\end{defn}

\begin{rem}
The perturbations we have in mind for the statement of the previous
theorem are exactly the same as the interior holonomy perturbations
that were needed to define the cobordism maps for the Floer groups
$HI(Y,K,\alpha)$.
\end{rem}

Despite the fact that $\lambda_{FO}(X,T,\alpha)$ is morally defined
as a count of flat connections, it is important to notice that it
is the equation $F_{A}^{+}=0$ which is perturbed, not the flatness
equation $F_{A}=0$. In other words, $\lambda_{FO}(X,T,\alpha)$ is
better interpreted as a degree $0$ Donaldson invariant.

This raises the question of whether this is the only degree $0$ Donaldson
invariant which can be defined for the embedded torus $T$. In fact,
it is possible to define additional invariants $D_{0}(X,T,\alpha,k)$
where $k\in\mathbb{Z}$ is an integer, and provided $\alpha\neq1/4$.
Interestingly enough, \textbf{for $k\neq0$, the invariants $D_{0}(X,T,\alpha,k)$
}\textit{can be defined for any embedded torus, independent of whether
it is null-homologous or not.} To see why this is the case, we need
to go back to the formulas \ref{simplified dimension and energy torus}
for the energy and dimension of the moduli spaces
\begin{align*}
\dim\mathcal{M}(X,T,k,l,\alpha)=8k+4l\\
\mathcal{E}(X,T,k,l,\alpha)=k+2\alpha l
\end{align*}
Notice that $\mathcal{M}(X,T,k,l,\alpha)$ is zero-dimensional whenever
\[
l=-2k
\]
The corresponding energy of this moduli space is 
\[
\mathcal{E}(X,T,k,-2k,\alpha)=k(1-4\alpha)
\]
In particular, when $\alpha=1/4$, the energy of $\mathcal{M}(X,T,k,-2k,1/4)$
is zero, which means it can only consist of $\alpha$-flat connections.
But we already know from Lemma \ref{lem: existence of flat connection}
that this can only happen when $k=0$, which means that 
\[
k\neq0\implies\mathcal{M}(X,T,k,-2k,1/4)=\emptyset
\]
 However, when $\alpha\neq1/4$, the moduli spaces $\mathcal{M}(X,T,k,-2k,\alpha)$
are a priori non-empty, at least provided that $\mathcal{E}(X,T,k,-2k,\alpha)\geq0$.
A more important question is whether they are compact.

Since they are already $0$ dimensional, the only way for $\mathcal{M}(X,T,k,-2k,\alpha)$
to be non-compact is if a sequence of $\alpha$-ASD connections $[A_{i}]$
inside $\mathcal{M}(X,T,k,-2k,\alpha)$ bubbles off and converges
weakly to an $\alpha$-ASD connection $[A_{\infty}]$ on some moduli
space $\mathcal{M}(X,T,k',l',\alpha)$ of \textit{negative} dimension.
In fact, since bubbles drop dimensions by $4$, $\dim\mathcal{M}(X,T,k',l',\alpha)\leq-4$.
Fortunately, by Lemma \ref{reducible connection k,l=00003D0}, $\mathcal{M}(X,T,k',l',\alpha)$
can admit no reducible $\alpha$-ASD connections, so for generic perturbations
there is no risk in assuming that $\mathcal{M}(X,T,k',l',\alpha)$
is empty (this is explained in great detail in \cite[Sections 3 and 5]{MR2192061}). 

Therefore, for generic perturbations $\mathcal{M}(X,T,k,-2k,\alpha)$
is in fact compact, and a count of signed points in $\mathcal{M}(X,T,k,-2k,\alpha)$
will be independent of the perturbation chosen, because a path of
perturbations will generically miss non-empty negative dimensional
moduli spaces. Notice that for $k\neq0$, $\alpha\neq1/4$, $\mathcal{M}(X,T,k,-2k,\alpha)$
has no reducibles to begin with, which in particular means that the
count of signed points inside $\mathcal{M}(X,T,k,-2k,\alpha)$ can
be made regardless of whether $T$ is null-homologous or not.

These observations allows us to define the additional invariants $D_{0}(X,T,k,\alpha)$
we promised earlier.
\begin{defn}
\label{Def deg 0 Donaldson invariants}Suppose that $\alpha\in\mathbb{Q}\cap(0,1/2)$
and $k\in\mathbb{Z}\backslash\{0\}$ is a non-zero integer such that
$k(1-4\alpha)\geq0$. Let $T$ be a oriented embedded torus inside
$X$ (null-homologous or not). After choosing an orientation of $H^{1}(X;\mathbb{R})$,
define $D_{0}(X,T,k,\alpha)$ as the signed count of points inside
the moduli space $\mathcal{M}(X,T,k,-2k,\alpha)=\mathcal{M}^{*}(X,T,k,-2k,\alpha)$.
When $\alpha=1/4$, set $D_{0}(X,T,k,1/4)=0$.
\end{defn}

\begin{rem}
Notice that the choice of homology orientation is no longer determined
in a canonical way by an orientation of $T$, if we allow $T$ to
be null-homologous. 

Also, it is not all clear what is the geometric meaning of the invariants
$D_{0}(X,T,\alpha,k)$. It is possible they may not contain any interesting
topological information about the torus $T$. For example, if we could
solve the issue of the implicit dependence of the invariants on the
cone angle being used, and moreover if we succeeded in defining them
for irrational values of $\alpha$ as well, then one could try to
use a deformation argument to show that for $k\neq0$, $D_{0}(X,T,\alpha,k)$
must vanish, since in this case $D_{0}(X,T,\alpha,k)$ could be compared
to $D_{0}(X,T,1/4,k)$, which we already know vanishes. 

Likewise, we can construct an analogue of the invariants $D_{0}(X,T,\alpha,k)$
for the case of an embedded sphere $S^{2}\hookrightarrow X$. Namely,
the expected dimension of the moduli spaces $\mathcal{M}(X,S^{2},k,-2k,\alpha)$
is now $2$, as can be seen from the formula \ref{simplified dimension and energy}.
As long as $\alpha\neq1/4$ and $k\neq0$, some of these moduli spaces
could be non-empty, and they will be compact and free of reducibles
by a similar argument. Hence, by pairing them with $\mu_{S^{2}}(x)$
for $x\in S^{2}$, we can define a degree-two Donaldson invariant
$D_{2}(X,S^{2},k,\alpha)$. However, the fact that these invariants
cannot be defined when $\alpha=1/4$ (which is the best value $\alpha$
could take from many points of view), suggests to us these invariants
$D_{2}(X,S^{2},k,\alpha)$ will probably end up giving no interesting
topological information. 
\end{rem}

We finish this section by analyzing the action of $H^{1}(X;\mathbb{Z}/2)$
on the moduli spaces $\mathcal{M}(X,T,k,-2k,\alpha)$, as was promised
in the introduction. We follow \cite[Section 4.6]{MR2189939} and
\cite[Section 3]{MR3704245} in order to describe this action.

First of all, $H^{1}(X;\mathbb{Z}_{2})=\hom(\pi_{1}(X);\mathbb{Z}_{2})$
parametrizes isomorphism classes of complex line bundles (with connection)
$\chi$ with holonomy $\{\pm1\}$ (along the loop $\vartheta$ in
our case). Since $\chi$ lifts to an integral homology class, the
bundle $L_{\chi}$ is trivial and thus for any $(k,l)$, the bundles
$E(k,l)$ and $E(k,l)\otimes L_{\chi}$ are isomorphic. Thus the action
of $H^{1}(X;\mathbb{Z}_{2})$ on $\mathcal{M}(X,T,k,-2k,\alpha)$
can be regarded as the one which sends a connection $[A]$ to $[A\otimes\chi]$. 

In general this action may or may not be free. We only care about
the freeness of the action on the irreducible part of the moduli space
$\mathcal{M}^{*}(X,T,k,-2k,\alpha)$ (again, when $k\neq0$, this
coincides with the entire moduli space $\mathcal{M}(X,T,k,-2k,\alpha)$),
which we will show in the next lemma.

Here we only need to analyze the action on the unperturbed moduli
spaces, since the idea is that once we know the action is free in
the unperturbed case, one can find perturbations that are $H^{1}(X;\mathbb{Z}_{2})$
equivariant and still guarantee transversality for the moduli spaces
\cite[section 4.6]{MR2189939}.
\begin{lem}
\label{lem:FREE action}Suppose that  the (unperturbed) moduli
space $\mathcal{M}^{*}(X,T,k,-2k,\alpha)$ is non-empty. Then $H^{1}(X;\mathbb{Z}/2)$
acts freely on $\mathcal{M}^{*}(X,T,k,-2k,\alpha)$.
\end{lem}

\begin{proof}
Let $[A]\in\mathcal{M}^{*}(X,T,k,-2k,\alpha)$ be an irreducible $\alpha$-ASD
connection. The connection $A$ induces a connection $A^{ad}$ on
the adjoint bundle $E^{ad}(k,-2k)$ of $E(k,-2k)$. On the adjoint
bundle it makes sense to consider the gauge group $\mathcal{G}_{SO(3)}(X,T)$
of all $SO(3)$ gauge transformations (as opposed to the $SU(2)$
gauge transformations which are the ones we have been working with).
As in the non-singular case (i.e, when $T$ is not present), it is
still the case that \cite[Section 5.1]{MR2805599}
\[
\mathcal{G}_{SO(3)}(X,T)/(\mathcal{G}(X,T)/\{\pm1\})\simeq H^{1}(X;\mathbb{Z}/2)
\]
Therefore, the action of $H^{1}(X;\mathbb{Z}/2)$ is free on $[A]$
if and only if the stabilizer of $A^{ad}$ with respect to the full
gauge group $\mathcal{G}_{SO}(X,T)$ is trivial. In general, since
$A$ is irreducible with respect to $\mathcal{G}(X,T)$, $\text{stab}_{SO(3)}A^{ad}$
can only be one of three possibilities: $1$, $\mathbb{Z}_{2}$ or
the Klein-4 group $V_{4}$. Therefore we must rule out that $\mathbb{Z}_{2}$
and $V_{4}$ can arises as potential stabilizers.

The case of $V_{4}$ is easy: a connection $A$ with stabilizer $V_{4}$
must be flat \cite[Section 4]{MR3704245}, and thus cannot belong
to $\mathcal{M}(X,T,k,-2k,\alpha)$ for $k\neq0$, since these a priori
do not support any flat connections.

In the case that $k=0$ but $\alpha\neq1/4$, we just need to use
the fact that every $\alpha$-representation $\rho$ such that $\rho^{ad}$
has stabilizer $V_{4}$ also has holonomy $V_{4}$, hence $\rho^{ad}$
will correspond in general to representations of $\pi_{1}(X\backslash T)$
with image into $V_{4}\subset SO(3)$ (\cite[Section 4]{MR3704245},
\cite[Examples 2.9]{MR3880205}). Given that every element of $V_{4}$
has order $2$, the only value of $\alpha$ compatible with $V_{4}$
representations corresponds to $\alpha=1/4$. Notice that $\rho$
will then have image contained in the quaternionic subgroup $Q_{8}=\{\pm1,\pm\mathbf{i},\pm\mathbf{j},\pm\mathbf{k}\}$
when we identify $SU(2)$ with the unit quaternions.

For the case of $k=0$ and $\alpha=1/4$, we need to use the fact
that the existence of Klein-4 representations is a homological phenomenon.
Namely, they exist whenever one can find three nontrivial real line
bundles which are \textit{distinct} \cite[Section 3]{MR3704245} on
the manifold. In our case, the manifold to consider is $X\backslash\text{nbd}(T)$,
where $\text{nbd}(T)$ is an open neighborhood of the torus $T$.
Since $H^{1}(X\backslash\text{nbd}(T);\mathbb{Z}_{2})\simeq\mathbb{Z}_{2}\oplus\mathbb{Z}_{2}$,
one can find three distinct nontrivial real line bundles which determines
a Klein-4 representation. Call these (real) line bundles $\epsilon_{1},\epsilon_{2},\epsilon_{3}=\epsilon_{1}+\epsilon_{2}$,
where $\epsilon_{1}$ generates the first factor of $H^{1}(X\backslash\text{nbd}(T);\mathbb{Z}_{2})$,
while $\epsilon_{2}$ generates the second factor. It is easy to see
that in this case 
\begin{align*}
 & w_{2}(E^{ad}\mid_{X\backslash\text{nbd}(T)})\\
= & w_{2}\left(\epsilon_{1}\oplus\epsilon_{2}\oplus\epsilon_{3}\right)\\
= & w_{1}(\epsilon_{1})w_{1}(\epsilon_{2})+w_{1}(\epsilon_{1})w_{1}(\epsilon_{3})+w_{1}(\epsilon_{2})w_{1}(\epsilon_{3})\\
= & w_{1}(\epsilon_{1})w_{1}(\epsilon_{2})+\left[w_{1}(\epsilon_{1})\right]^{2}+w_{1}(\epsilon_{1})w_{1}(\epsilon_{2})+w_{1}(\epsilon_{2})w_{1}(\epsilon_{1})+\left[w_{1}(\epsilon_{2})\right]^{2}\\
= & w_{1}(\epsilon_{2})w_{1}(\epsilon_{1})\neq0
\end{align*}
 In other words, the Klein 4 representation we found exists on a bundle
with non-trivial $w_{2}$. However, the bundle $E^{ad}(0,0)$ over
$X$ has vanishing $w_{2}$ (since $X$ is an integral homology $S^{1}\times S^{3}$,
or alternatively, because $E(0,0)$ was the trivial bundle to begin
with), which means that its restriction to the torus complement should
have vanishing $w_{2}$ as well by naturality of the Stiefel-Whitney
classes. Therefore, $\mathcal{M}^{*}(X,T,0,0,1/4)$ is also free of
connections with stabilizer $V_{4}$ (with respect to $\mathcal{G}_{SO(3)}(X,T)$).

The case of $\mathbb{Z}_{2}$ stabilizer corresponds to the so-called
twisted reducibles \cite[Section 2 i)]{MR1338483}: these are those
connections which preserve a splitting $\mathfrak{g}_{E}=\lambda\oplus P$
, where now $\lambda$ is a non-orientable real line bundle and $P$
is a non-orientable real two plane bundle with orientation bundle
isomorphic to $\lambda$. Now we can use the fact that our connections
have a prescribed model near the surface $T$: as explained before
Lemma 2.22 in \cite{MR1338483}, if $A$ were a twisted reducible,
then $P$ would have to coincide with $\pm L^{\otimes2}$ in a tubular
neighborhood of $T$, and therefore $\lambda$ must be trivial on
$T$. 

Now, because $H_{1}(T;\mathbb{Z}_{2})$ maps onto $H_{1}(X;\mathbb{Z}_{2})$,
then $\lambda$ will be trivial on all of $X$, which means that it
is orientable, thus we obtain a contradiction. Notice that the paragraph
we refer to from \cite{MR1338483} starts by stating that $A$ must
be a non-flat connection. However, this condition is not used in this
argument, rather it was assumed by Kronheimer and Mrowka because they
were interested in obtaining a generic metrics theorem in the presence
of twisted reducibles, and in general this cannot be achieved whenever
there are flat connections.
\end{proof}
\begin{rem}
Since for $k\neq0$, $\alpha\neq1/4$, the moduli spaces $\mathcal{M}(X,T,k,-2k,\alpha)$
are free of reducible $\alpha$-ASD connections, free of $\alpha$-flat
connections and free of twisted reducible connections, one can also
obtain transversality for these moduli spaces using the generic metrics
theorem, thanks to \cite[Lemma 2.17]{MR1338483}. Hence, one could
avoid using holonomy perturbations for defining the invariants $D_{0}(X,T,k,\alpha)$
for $k\neq0$.
\end{rem}

Before discussing some examples and properties of $\lambda_{FO}(X,T,\alpha,k)$,
we will prove the splitting formula for $\lambda_{FO}(X,T,\alpha,k)$,
which was one of our main motivations for defining the invariant.

\section{\label{sec:The-Splitting-Formula}The Splitting Formula}

Our first step for finding the splitting formula requires understanding
how $\alpha$-admissibility for self-concordances $(W,\varSigma):(Y,K)\rightarrow(Y,K)$
is related to the reducible representations being isolated from the
irreducible representations in the case of $(X,\varSigma)$. The next
lemma says that in fact both notions are equivalent.
\begin{lem}
Suppose that $(W,\varSigma):(Y,K)\rightarrow(Y,K)$ is a self-concordance
of a knot and we choose a parameter $\alpha$ for which $\triangle_{K}(e^{-4\pi i\alpha})\neq0$.
Let $(X,T)$ be the closed 4-manifold obtained by closing up $(W,\varSigma)$.
Then $\lambda_{FO}(X,T,\alpha)$ is well defined if and only if the
cobordism $(W,\varSigma)$ is $\alpha$-admissible. 
\end{lem}

\begin{proof}
We can think of $(X,T)$ as being obtained from $(W,\varSigma)$ after
attaching a tube $(I\times Y,I\times K)$ to the boundary of $(W,\varSigma)$,
where $I$ is some interval. That is, 
\[
(X,T)=(W,\varSigma)\cup(I\times Y,I\times K)
\]
Then we want to apply Mayer-Vietoris to this decomposition of $(X,T)$,
where we enlarged $I$ a little bit so that the overlap of $(W,\varSigma)$
with $(I\times Y,I\times K)$ is the disjoint union 
\[
(W,\varSigma)\cap(I\times Y,I\times K)=(I_{1}\times Y,I_{1}\times K)\sqcup(I_{2}\times Y,I_{2}\times K)
\]
 where $I_{1},I_{2}$ are two small subintervals. Recall that on the
orbifold $\check{W}$ there is only one $\alpha$-flat reducible $\theta_{W,\alpha}$
, while on $Y$ we have the reducible $\theta_{\alpha}$. For any
$\alpha$-flat reducible $A_{\rho}$ on $\check{X}$, it must restrict
to $\theta_{W,\alpha}$ and $\theta_{\alpha}$ on $\check{W}$ and
$\check{Y}$ respectively, which means that exact sequence for the
(orbifold) cohomology groups reads
\begin{align*}
0\rightarrow & \check{H}^{0}(\check{X};\mathfrak{g}_{\rho})\\
\rightarrow & \check{H}^{0}(\check{W};\mathfrak{g}_{\theta_{W,\alpha}})\oplus\check{H}^{0}(I\times\check{Y};\mathfrak{g}_{\theta_{\alpha}})\\
\rightarrow & \check{H}^{0}(\left(I_{1}\sqcup I_{2}\times\check{Y}\right);\mathfrak{g}_{\theta_{\alpha}})\\
\rightarrow & \check{H}^{1}(\check{X};\mathfrak{g}_{\rho})\\
\rightarrow & \check{H}^{1}(\check{W};\mathfrak{g}_{\theta_{W,\alpha}})\oplus\check{H}^{1}(I\times\check{Y};\mathfrak{g}_{\theta_{\alpha}})\\
\rightarrow & \check{H}^{1}((I_{1}\sqcup I_{2})\times\check{Y});\mathfrak{g}_{\theta_{\alpha}})\\
\rightarrow & \cdots
\end{align*}
Since $\triangle_{K}(e^{-4\pi i\alpha})\neq0$, we have that $\check{H}^{1}(\check{Y};\mathfrak{g}_{\theta_{\alpha}})=0$
so we can simplify the previous exact sequence into 
\begin{align*}
0\rightarrow & \mathbb{R}\\
\rightarrow & \mathbb{R}\oplus\mathbb{R}\\
\rightarrow & \mathbb{R}\oplus\mathbb{R}\\
\rightarrow & \check{H}^{1}(\check{X};\mathfrak{g}_{\rho})\\
\rightarrow & \check{H}^{1}(\check{W};\mathfrak{g}_{\theta_{W,\alpha}})\\
\rightarrow & 0\\
\rightarrow & \cdots
\end{align*}
 From this we can conclude that the alternating sum of the dimensions
of these vector spaces is zero, which means that 
\[
\dim\check{H}^{1}(\check{X};\mathfrak{g}_{\rho})=1+\dim\check{H}^{1}(\check{W};\mathfrak{g}_{\theta_{W,\alpha}})
\]
Thus, if $(W,\varSigma)$ is $\alpha$-admissible (i.e, $\dim\check{H}^{1}(\check{W};\mathfrak{g}_{\theta_{W,\alpha}})=0$)
then $\check{H}^{1}(\check{X};\mathfrak{g}_{\rho})$ vanishes for
all $\alpha$-reducible representations $\rho$, and conversely, if
$\lambda_{FO}(X,T,\alpha)$ can be defined (i.e, $\dim\check{H}^{1}(\check{X};\mathfrak{g}_{\rho})=1$)
then $\dim\check{H}^{1}(\check{W};\mathfrak{g}_{\theta_{W,\alpha}})$
must vanish which is the condition for the cobordism to be $\alpha$-admissible.
\end{proof}
We are finally ready to state the splitting formula. 
\begin{thm}
\label{splitting theorem}(Splitting Theorem) Suppose that $(W,\varSigma):(Y,K)\rightarrow(Y,K)$
is a self-concordance of a knot and we choose a parameter $\alpha\in\mathbb{Q}\cap(0,1/2)$
for which $\triangle_{K}(e^{-4\pi i\alpha})\neq0$. Let $(X,T)$ be
the closed 4-manifold obtained by closing up $(W,\varSigma)$. Suppose
that $\lambda_{FO}(X,T,\alpha)$ is well defined, or equivalently,
that $(W,\varSigma)$ is $\alpha$-admissible.   Then we have the
\textbf{splitting formula 
\begin{equation}
\sum_{k}D_{0}(X,T,\alpha,k)T^{-\mathcal{E}_{top}(X,T,k,-2k,\alpha)}=2\text{Lef}(W\mid HI(Y,K,\alpha))=2\text{Lef}(W\mid HI^{red}(Y,K,\alpha))-2h(Y,K,\alpha)\label{Splitting Formula}
\end{equation}
}
\end{thm}

\begin{proof}
(first equality of Theorem \ref{splitting theorem}) The argument
is standard and analogous to the one given in \cite[Section 11]{MR2738582},
\cite[Section 11.1]{MR2465077} and \cite[Section 9]{MR3811774}.
In fact, since we already analyzed the action of $H^{1}(X;\mathbb{Z}/2)$
on the moduli spaces, as we mentioned in the introduction, it can
be regarded as a consequence of Proposition 5.5 in \cite{MR2805599}
and the remarks after it.

Namely, the idea is to compare the moduli spaces $\mathcal{M}^{*}(X,T,k,-2k,\alpha)$
($k\in\mathbb{Z}$) with the 0-dimensional moduli spaces $\mathcal{M}_{0}([\beta],W,[\beta])$,
i.e, those which are asymptotic to $[\beta]$ as $t\rightarrow\pm\infty$.
Notice that because of the failure of monotonicity, specifying the
dimension is not enough (when $\alpha\neq1/4$), i.e, $\mathcal{M}_{0}([\beta],W,[\beta])$
needs to be furthermore indexed by the energy of the moduli space
\[
\mathcal{M}_{0}([\beta],W,[\beta])=\bigcup_{k\in\mathbb{Z}}\mathcal{M}_{0,E(k)}([\beta],W,[\beta])
\]
Now, we can focus on one of the individual moduli spaces $\mathcal{M}^{*}(X,T,k,-2k,\alpha)$
and introduce a parameter $R$ which keeps track of the length of
the cylinder in the usual stretching the neck argument, so that we
are consider the manifold $\check{X}(R)$ where a cylinder of length
$2R$ $[-R,R]\times\check{Y}$ has been introduced along $\check{Y}$.

The gluing argument says that for $R$ sufficiently large, $\mathcal{M}^{*}(X,T,k,-2k,\alpha)$
can be identified with $\bigcup_{[\beta]\in\mathfrak{C}^{*}(Y,K,\alpha)}\mathcal{M}_{0,E(k)}([\beta],W,[\beta])$.
Notice that we do not need to worry about $\mathcal{M}_{0,E(k)}([\theta_{\alpha}],W,[\theta_{\alpha}])$
because of the $\alpha$-admissibility condition. 

However, this correspondence is two to one, since when we are closing
the bundle over $W^{*}$ to produce the bundle over $X$, there are
two ways to do this since the stabilizer of $[\beta]$ is $\mathbb{Z}_{2}=Z(SU(2))$,
given that we are dealing with an irreducible connection \cite[Remark p.893]{MR2860345}.
This explains the factor of $2$ in the statement of the theorem.
\end{proof}
Now we will proof the second part of the Splitting Theorem. The proof
is an adaptation word by word of the one Anvari gives in \cite{Anvari[2019]},
so we will just illustrate one of the cases needed to proof this formula,
the other one can be found in the annotated version \cite{Echeverria[DraftFuruta]}.
\begin{proof}
(second equality of Theorem \ref{splitting theorem}) We start by
recalling the behavior of the maps 
\begin{align*}
\delta_{1,n}= & \delta_{1}u_{K}^{n}:HI_{1+2n}(Y,K,\alpha)\rightarrow\varLambda\\
\delta_{2,n}= & u_{K}^{n}\delta_{2}:\varLambda\rightarrow HI_{2-2n}(Y,K,\alpha)
\end{align*}
 which appear implicitly in our definition of the reduced Floer groups
\ref{reduced Floer groups}. In the case of a self-concordance $(W,\varSigma)$,
the induced cobordism map $m_{\check{W}}:HI_{*}(Y,K,\alpha)\rightarrow HI_{*}(Y,K,\alpha)$
acts on the $\delta_{1,n}$ and $\delta_{2,n}$ as follows \cite[Theorem 7]{MR1910040}:
there are integers $a_{ij},b_{ij}$ such that 
\begin{align}
\delta_{1,n}m_{\check{W}}= & \delta_{1,n}+\sum_{i=0}^{n-1}a_{in}\delta_{1,n}\label{relations cobordisms}\\
m_{\check{W}}\delta_{2,n}= & \delta_{2,n}+\sum_{i=0}^{n-1}b_{in}\delta_{2,n}\nonumber 
\end{align}
Since $m_{\check{W}}$ preserves gradings we can see that $a_{in}=0$
and $b_{in}=0$ whenever $i,n$ have opposite parity. We will also
use the fact (which follows from Lemma \ref{relation u-map}) that
either $\delta_{1}$ or $\delta_{2}$ must vanish. Finally, we also
need the fact that for a commutative diagram of exact sequences of
finite dimensional vector spaces over an arbitrary field $\varLambda$
\begin{align*}
0 & \rightarrow & V_{0} & \rightarrow & V_{1} & \rightarrow & V_{2}\\
 &  & \downarrow & _{\alpha} & \downarrow & _{\beta} & \downarrow_{\gamma}\\
0 & \rightarrow & W_{0} & \rightarrow & W_{1} & \rightarrow & W_{2}
\end{align*}
we have 
\begin{equation}
\text{tr}(\beta)=\text{tr}(\alpha)+\text{tr}(\gamma)\label{additivity trace}
\end{equation}
To show that
\begin{equation}
h(Y,K,\alpha)=\text{Lef}(W\mid HI^{red}(Y,K,\alpha))-\text{Lef}(W\mid HI(Y,K,\alpha))\label{desiderata}
\end{equation}
 we make cases based on the vanishing of $\delta_{1},\delta_{2}$.

\textbf{Case $\delta_{2}=0$: }In this situation the reduced Floer
groups \ref{reduced Floer groups} simplify to 
\[
\begin{cases}
HI_{1}^{red}(Y,K,\alpha)=\cap_{l}\ker(\delta_{1,2l})\\
HI_{3}^{red}(Y,K,\alpha)=\cap_{l}\ker(\delta_{1,2l+1})\\
HI_{0}^{red}(Y,K,\alpha)=HI_{0}(Y,K,\alpha)\\
HI_{2}^{red}(Y,K,\alpha)=HI_{2}(Y,K,\alpha)
\end{cases}
\]
Therefore the difference in Lefschetz numbers simplifies to 
\begin{align*}
 & \text{Lef}(W\mid HI^{red}(Y,K,\alpha))-\text{Lef}(W\mid HI(Y,K,\alpha))\\
= & \text{Tr}(W\mid HI_{1}(Y,K,\alpha)\oplus HI_{3}(Y,K,\alpha))-\text{Tr}(W\mid HI_{1}^{red}(Y,K,\alpha)\oplus HI_{3}^{red}(Y,K,\alpha))\\
= & \text{Tr}(W\mid HI_{1}(Y,K,\alpha))-\text{Tr}(W\mid HI_{1}^{red}(Y,K,\alpha))+\text{Tr}(W\mid HI_{3}(Y,K,\alpha))-\text{Tr}(W\mid HI_{3}^{red}(Y,K,\alpha))
\end{align*}
 In this case $h(Y,K,\alpha)$ \ref{eq:h-invariant knot} simplifies
to
\begin{align*}
 & \chi_{\varLambda}(HI^{red}(Y,K,\alpha))-\chi_{\varLambda}(HI(Y,K,\alpha))\\
= & -\dim_{\varLambda}HI_{1}^{red}(Y,K,\alpha)-\dim_{\varLambda}HI_{3}^{red}(Y,K,\alpha)+\dim_{\varLambda}HI_{1}(Y,K,\alpha)+\dim_{\varLambda}HI_{3}(Y,K,\alpha)\\
= & \dim_{\varLambda}HI_{1}(Y,K,\alpha)-\dim_{\varLambda}HI_{1}^{red}(Y,K,\alpha)+\dim_{\varLambda}HI_{3}(Y,K,\alpha)-\dim_{\varLambda}HI_{3}^{red}(Y,K,\alpha)\\
= & \dim_{\varLambda}(HI_{1}(Y,K,\alpha)/\cap_{l}\ker(\delta_{1,2l}))+\dim_{\varLambda}(HI_{3}(Y,K,\alpha)/\cap_{l}\ker(\delta_{1,2l+1}))
\end{align*}
So it clearly suffices to show that 
\begin{align*}
\dim_{\varLambda}(HI_{1}(Y,K,\alpha)/\cap_{l}\ker(\delta_{1,2l}))=\text{Tr}(W\mid HI_{1}(Y,K,\alpha))-\text{Tr}(W\mid HI_{1}^{red}(Y,K,\alpha))\\
\dim_{\varLambda}(HI_{3}(Y,K,\alpha)/\cap_{l}\ker(\delta_{1,2l+1}))=\text{Tr}(W\mid HI_{3}(Y,K,\alpha))-\text{Tr}(W\mid HI_{3}^{red}(Y,K,\alpha))
\end{align*}
to obtain our result. The argument is completely analogous in both
cases, so let us do verify the second identity. The result will be
obtained from induction on the sequence of subspaces 
\[
Z_{3,k}=\cap_{l=0}^{k}\delta_{1,2l+1}
\]
From \ref{relations cobordisms} we can see that for $k=0$ we have
$\delta_{1}m_{\check{W}}=\delta_{1}$, which in other words means
that there is an exact sequence
\begin{align*}
0 & \rightarrow & Z_{3,0} & \rightarrow & HI_{3} & \rightarrow^{\delta_{1}} & \varLambda\\
 &  & \downarrow & _{m_{3,0}} & \downarrow & _{m_{\check{W}}} & \downarrow_{\text{id}}\\
0 & \rightarrow & Z_{3,0} & \rightarrow & HI_{3} & \rightarrow_{\delta_{1}} & \varLambda
\end{align*}
Here $m_{3,0}$ denotes the restriction of the cobordism map to $Z_{3,0}$.
The additivity of the trace formula \ref{additivity trace} now says
that 
\[
\text{tr}(m_{\check{W}})=\text{tr}(m_{3,0})+\text{tr}(\text{id})=\text{tr}(m_{3,0})+1=\text{tr}(m_{3,0})+\dim_{\varLambda}(HI_{3}(Y,K,\alpha)/Z_{3,0})
\]
provided that $\delta_{1}\neq0$ (which in any case is the only interesting
situation since we already assumed that $\delta_{2}=0$). An induction
argument on $k$ \cite[p. 7]{Anvari[2019]} based on the identities
\ref{relations cobordisms} now says that 
\[
\text{tr}(m_{\check{W}})=\text{tr}(m_{3,k})+\dim(HI_{3}(Y,K,\alpha)/Z_{3,k})
\]
Since the sequence of the $Z_{3,k}$ stabilizer once $k$ is large
enough then we find that
\[
\text{tr}(m_{\check{W}}\mid HI_{3}(Y,K,\alpha))-\text{tr}(m_{\check{W}}^{red}\mid HI_{3}^{red}(Y,K,\alpha))=\dim(HI_{3}(Y,K,\alpha)/Z_{3,k})
\]
Again, the case for $HI_{1}(Y,K,\alpha)$ is completely analogous
so \ref{desiderata} has been verified in this situation. 

Finally, the case $\delta_{1}=0$ is dealt with in a similar way,
but the interested reader can find a proof in \cite{Echeverria[DraftFuruta]}.
\end{proof}
Now we will show that under the previous circumstances $(X,T)$ can
be assigned an $h$-invariant. 
\begin{thm}
\label{thm:h-invariant for the knot} Suppose that $(X,T)$ can be
written as the closed-up version of two different self-concordances
$(W,\varSigma):(Y,K)\rightarrow(Y,K)$ and $(W',\varSigma'):(Y',K')\rightarrow(Y',K')$
which satisfy $\triangle_{K}(e^{-4\pi i\alpha})\neq0$ and $\triangle_{K'}(e^{-4\pi i\alpha})\neq0$
for some $\alpha\in\mathbb{Q}\cap(0,1/2)$. Suppose $\lambda_{FO}(X,T,\alpha)$
can be defined, or equivalently, either of the concordances (and hence
the other) is $\alpha$-admissible. Then 
\[
\text{Lef}(W\mid HI^{red}(Y,K,\alpha))=\text{Lef}(W\mid HI^{red}(Y',K',\alpha))
\]
 and thus  we can use the splitting formula \ref{Splitting Formula}
to define $h(X,T,\alpha)$ as $h(Y,K,\alpha)=h(Y',K',\alpha)$.
\end{thm}

\begin{proof}
The proof is indistinguishable from the one Frøyshov gives for monopole
$h$-invariant in \cite[Section 13]{MR2738582}. First Frøyshov proves
Lemma 10 in \cite{MR2738582}, which in our case translates to the
following statement: let $A$, $B$ be $r\times r$ matrices with
coefficients in the universal Novikov field $\varLambda^{\mathbb{C},\mathbb{R}}$
, and let $m$ be a natural number such that $\text{tr}(A^{n})=\text{tr}(B^{n})$
for all natural numbers $n$ satisfying $m\leq n<2r+m$. Then $A$
and $B$ have the characteristic polynomial, and in particular $\text{tr}(A)=\text{tr}(B)$.

Frøyshov states this Lemma for matrices with coefficients over the
complex field $\mathbb{C}$, but a quick inspection of the proof reveals
that the only property used about $\mathbb{C}$ is that it is algebraically
closed, so that the characteristic polynomial of $A$ (or $B$) has
roots, which are the eigenvalues of $A$ (or $B$). That $\varLambda^{\mathbb{C},\mathbb{R}}$
is algebraically closed is proven in \cite[Lemma A.1]{MR2573826}.
Once we know this lemma holds, the argument Frøyshov gives is the
same, just replace every occurrence of $X,W,Y,X_{\infty},X_{j,\infty},W_{j,n}$
,etc in his proof with their orbifold versions $\check{X},\check{W},\check{Y},\check{X}_{\infty},\check{X}_{j,\infty},\check{W}_{j,n}$. 

To give more details, as in \cite[Section 13]{MR2738582} one assumes
that $Y,Y'$ are the inverse images of maps $f:X\rightarrow S^{1}$,
$f':X\rightarrow S^{1}$ in such a way that $f^{-1}(1)=Y$, $f^{\prime-1}(1)=Y'$.
The important features of these maps is that they were homotopic through
some map $F:[0,1]\times X\rightarrow S^{1}$, since $Y,Y'$ represented
the same class in $H_{3}(X;\mathbb{Z})\simeq H^{1}(X;\mathbb{Z})\simeq[X,S^{1}]$.

Now, if we let $X_{T}=X\backslash T$ denote the torus complement
there is no problem in assuming that as maps $f_{T}=f:X_{T}\rightarrow S^{1}$,
$f'_{T}=f':X_{T}\rightarrow S^{1}$, we also have that $f^{-1}(1)=Y\backslash K$,
$f^{-1}(1)=Y'\backslash K'$. Since $T\cap Y=K$ and $T\cap Y'=K'$,
by restricting $F$ to $[0,1]\times X_{T}$, we find out that that
$f_{T},f_{T}'$ are homotopic as well (as elements of $[X_{T},S^{1}]$),
which means that the infinite cyclic covers they determine are in
fact diffeomorphic. After that Frøyshov's argument goes through.

\end{proof}

\section{\label{sec:Some-examples}Some examples and Properties}

\subsection{Product Case}

\ 

As a corollary of the splitting theorem \ref{splitting theorem} we
can verify our basic desiderata for $\lambda_{FO}(X,T,\alpha)$.
\begin{cor}
Suppose that $(Y,K)$ is such that $\triangle_{K}(e^{-4\pi i\alpha})\neq0$
for $\alpha\in\mathbb{Q}\cap(0,1/2)$. Then $\lambda_{FO}(X,T,\alpha)$
can be defined on $(X,T)=(S^{1}\times Y,S^{1}\times K)$.

Moreover, $\lambda_{FO}(X,T,\alpha)=2\lambda_{CLH}(Y,K,\alpha)$
and the additional degree zero Donaldson invariants $D_{0}(X,T,\alpha,k)$
vanish for $k\neq0$.
\end{cor}

\begin{proof}
In this case the self-concordance is $(W,\varSigma)=([0,1]\times Y,[0,1]\times K)$
which will clearly be $\alpha$-admissible. 

According to the splitting formula \ref{Splitting Formula} we have
for $(X,T)=(S^{1}\times Y,S^{1}\times K)$
\[
\sum_{k}D_{0}(X,T,\alpha,k)T^{-\mathcal{E}(X,T,k,-2k,\alpha)}=2\text{Lef}(Id\mid HI(Y,K,\alpha))=2\chi_{\varLambda}(HI(Y,K,\alpha))=2\lambda_{CH}(Y,K,\alpha)
\]
from which we will conclude that 
\[
\begin{cases}
\lambda_{FO}(S^{1}\times Y,S^{1}\times K,\alpha)=2\lambda_{CH}(Y,K,\alpha)\\
D_{0}(X,T,\alpha,k)=0 & k\neq0
\end{cases}
\]
\end{proof}
\begin{rem}
In particular, notice that one cannot use the invariants $D_{0}(X,T,\alpha,k)$
(for $k\neq0$) to obtain new invariants for knots, as one might have
suspected all along. 
\end{rem}

\subsection{Flip symmetry }

\ 

Now we briefly discuss the proof of Theorem \ref{Flip symmetry} from
the introduction, i.e, the flip symmetry $\mathcal{F}$ obtained by
changing the holonomy parameter from $\alpha$ to $\frac{1}{2}-\alpha$.
This is defined similar to the action of $H^{1}(X;\mathbb{Z}/2)$,
in the sense that there is a natural map $\mathcal{F}:\mathcal{M}(X,T,k,l,\alpha)\rightarrow\mathcal{M}(X,T,k+l,-l,\frac{1}{2}-\alpha)$
obtained by tensoring the bundle $E(k,l)$ with a line bundle $\chi$
whose holonomy on the small circles linking $T$ is $-1$. The new
values for $k$ and $l$ after the flip symmetry is performed were
computed in \cite[Section 2, iv)]{MR1241873}. 

Moreover, in \cite[Appendix 1, ii)]{MR1308489} they discuss the effect
of $\mathcal{F}$ on the orientation of the moduli spaces. In particular,
$\mathcal{F}$ preserves or reverses orientation according to the
parity of $\frac{1}{4}T\cdot T-(g-1)$, which in our case vanishes,
which means that $\mathcal{F}$ is orientation preserving. 

Recalling that $D_{0}(X,T,k,\alpha)$ was computed using $\mathcal{M}(X,T,k,-2k,\alpha)$,
this means that $D_{0}(X,T,k,\alpha)=D_{0}\left(X,T,-k,\frac{1}{2}-\alpha\right)$
as was claimed in Theorem \ref{Flip symmetry}. 

The effect on the Floer homologies $\mathcal{F}:HI(Y,K,\alpha)\rightarrow HI(Y,K,\frac{1}{2}-\alpha)$
can be analyzed in a similar way. The only thing to worry about is
about the effect of $\mathcal{F}$ on the gradings of a critical point
before and after the flip has been performed. Clearly $\mathcal{F}([\theta_{\alpha}])=\mathcal{F}([\theta_{1/2-\alpha}])$
and more generally $\mathcal{F}(\mathcal{M}([\beta],[\theta_{\alpha}])=\mathcal{M}([\mathcal{F}\beta],[\theta_{1/2-\alpha}])$
which from the formula for the absolute grading \ref{Flip symmetry}
will imply that $\text{gr}([\mathcal{F}\beta])=\text{gr}([\beta])$,
in other words, $\mathcal{F}$ will be grading preserving. In particular,
the Euler characteristic is preserved under the flip operation, which
could also be checked directly from the formula for $\lambda_{CLH}(Y,K,\alpha)$.

The effect of $\mathcal{F}:HI^{red}(Y,K,\alpha)\rightarrow HI^{red}(Y,K,\frac{1}{2}-\alpha)$
can be analyzed in a similar way. The only thing to be aware of is
that under $\mathcal{F}$ the $u$-map $\mu_{K}(x)$ changes sign,
i.e, $\mathcal{F}(\mu_{K}(x))=-\mu_{K}(x)$, as is discussed in \cite[Section 4.2]{MR1432428}.
However, changing the sign of the $u$ map still is compatible with
the definition of the reduced Floer groups \ref{reduced Floer groups}
so the Euler characteristic of $HI^{red}(Y,K,\alpha)$ is preserved
under flip symmetries. From the formula for the Frøyshov knot invariants
\ref{eq:h-invariant knot} it is immediate that $h(Y,K,\alpha)=h(Y,K,\frac{1}{2}-\alpha)$,
which was the first statement of Theorem \ref{Flip symmetry}.

\subsection{Duality }

Just as for the non-singular versions, the Floer groups $HI(Y,K,\alpha)$
are related to those of $HI(-Y,-K,\alpha)$, where $(-Y,-K)$ denotes
the pair $(Y,K)$ with the opposite orientation on both factors. When
there is no knot present, the way to understand the Floer homology
$HI(-Y)$ in terms of the Floer homology $HI(Y)$ is standard, what
we need to do discuss is that happens in the presence of the knot
$K$ as well as the effect on the local systems $\varGamma_{[\beta]}$.
We will follow the discussion in \cite[Sections 22.5 and  32.1]{MR2388043},
\cite[Section 7.4]{MR3394316}. 

The Chern-Simons functionals $CS_{\check{Y}}$ and $CS_{-\check{Y}}$
on the orbifolds $\check{Y}$ and $-\check{Y}$ are related as 
\[
CS_{\check{Y}}=-CS_{-\check{Y}}
\]
Therefore the critical points of the functionals can be identified
in a natural way with each other. The perturbation on $-Y$ can be
taken to be $-\mathfrak{p}$, while the vector field $\text{grad}CS_{-\check{Y}}$
is the negative of $\text{grad}CS_{\check{Y}}$. In other words, if
$\gamma(t)$ is a trajectory on $\check{Y}$ of $CS_{\check{Y}}$
then $\gamma(-t)$ is a trajectory on $-\check{Y}$ of $CS_{-\check{Y}}$.
The coefficient system on $-\check{Y}$ can be described as a coefficient
system on $\mathcal{B}(\check{Y},\alpha)=\mathcal{B}(Y,K,\alpha)$,
by saying that the fiber at each point is still $\varLambda^{\mathbb{Q},\mathbb{R}}$,
but now along paths $z$ from $[\beta_{0}]$ to $[\beta_{1}]$ one
multiplies by $T^{+\mathcal{E}_{top}(z)}$ instead of $T^{-\mathcal{E}_{top}(z)}$. 

To obtain the grading formula, consider the cylinders $\mathbb{R}\times\check{Y}$
and $\mathbb{R}\times-\check{Y}$ . Let $[\beta]\in\mathfrak{C}(Y,K,\alpha)$
be a critical point and $[\bar{\beta}]$ the corresponding class in
$\mathfrak{C}(-Y,-K,\alpha)$. Recall that \ref{absolute grading}
\[
\text{gr}([\beta])=-1-\dim\mathcal{M}([\theta_{\alpha}],[\beta])\mod4=\dim\mathcal{M}([\beta],[\theta_{\alpha}])\mod4
\]

A flow line on the moduli space $\mathcal{M}([\bar{\beta}],[\bar{\theta}_{\alpha}])$
can be identified with a flow line of the moduli space $\mathcal{M}([\theta_{\alpha}],[\beta])$,
since the time direction has been reversed, which means 
\[
\text{gr}([\bar{\beta}])=\mathcal{M}([\bar{\beta}],[\bar{\theta}_{\alpha}])\mod4=\dim\mathcal{M}([\theta_{\alpha}],[\beta])\mod4=-1-\text{gr}([\beta])
\]

Therefore we have found (\cite[Proposition 4.3]{MR1703606}):
\begin{thm}
For each $i\in\mathbb{Z}/4\mathbb{Z}$, there is an isomorphism 
\begin{equation}
HI_{i}(-Y,-K,\alpha)\simeq HI_{-i-1}(Y,K,\alpha)\label{duality isomorphism}
\end{equation}
 
\end{thm}

\subsection{Some examples of tori inside mapping tori}

\ 

These examples can be considered as the orbifold version of \cite{MR2033479}.
Before writing a general statement, let's consider a toy model. Suppose
that we have a knot $K'$ inside an integer homology sphere $Y'$
and we choose holonomy $\alpha'=\frac{1}{15}$ along the $K'$. 

Moreover, assume that after taking the $3$-fold branched cover along
$K'$ we obtain a 3 manifold $Y$ which is still an integer homology
sphere . Notice that $Y$ comes with a natural $\mathbb{Z}_{3}$
action $\tau$, whose fixed point set is a knot $K$. Now choose holonomy
$\alpha=\frac{1}{5}$ along the knot $K$. Then the mapping torus
\[
X_{\tau}=([0,1]\times Y)/(\{0\}\times Y\sim^{\tau}\{1\}\times Y)
\]
 of $(Y,\tau)$ will be a homology $S^{1}\times S^{3}$ with a natural
torus $T_{\tau}$ obtained as the mapping torus of $K\subset Y$.
In this situation the analogue of \cite[Proposition 3.1]{MR2033479}
will tell us that there is a two to one correspondence between $\alpha$-representations
on $(X_{\tau},T_{\tau})$ and $\alpha$-representations on $(Y,K)$
which are $\tau$-equivariant. 

Clearly, every $\alpha'$- representation of $(Y',K')$ will pullback
to an $\alpha$-representation on $(Y,K)$ which is $\tau$ -equivariant.
So the only question is whether this exhausts all the possibilities
for being a $\tau$-equivariant representation. In fact, it does not!
Suppose we had chosen holonomy $\tilde{\alpha}'=\frac{6}{15}$ along
$K'$. Then the pull-back of an $\tilde{\alpha}'$ representation
of $(Y',K')$ will have holonomy $6/5$ along $K$ upstairs. Recall
that we are using the normalization for the holonomy to be between
$0$ and $1/2$, and since 
\[
\frac{6}{5}=\frac{1}{5}+2\cdot\frac{1}{2}
\]
this means that after performing two half-twists, holonomy $6/5$
is equivalent to holonomy $1/5$. 

As Langte Ma pointed out to the author, there are still more holonomy
values allowed. Consider for example $\tilde{\alpha}'=\frac{11}{15}$
. Since 
\[
\frac{11}{5}=\frac{1}{5}+4\cdot\frac{1}{2}
\]
after four half-twists, holonomy $11/5$ is equivalent to holonomy
$1/5$. Moreover, since $\frac{1}{2}<\frac{11}{15}<1$, due to the
normalization conventions we can take this one to be equivalent to
holonomy $\frac{11}{15}-\frac{1}{2}=\frac{7}{30}$, in other words,
\[
\frac{11}{5}\sim\frac{7}{30}
\]
Holonomy $\frac{7}{30}$ pulls up to holonomy $\frac{7}{10}=\frac{1}{5}+\frac{1}{2}$
, which differs from $\frac{1}{5}$ by one half-twists. Moreover 
\[
\sigma_{K'}(e^{-4\pi i\frac{11}{15}})=\sigma_{K'}(e^{-4\pi i(\frac{7}{30}+\frac{1}{2})})=\sigma_{K'}(e^{-4\pi i\frac{7}{30}})
\]
so the value of the knot signature is well defined regardless of the
identification we use.

This exhaust all the possibilities, since 
\[
\frac{1}{5}+6\cdot\frac{1}{2}=\frac{16}{5}
\]
which would be induced by holonomy $\frac{16}{15}$ downstairs, which
is bigger than one, hence can be ignored.

In general all the allowable values of holonomy on $(Y',K')$ which
give rise to holonomy $\alpha$ upstairs will be of the form 
\[
\frac{\alpha}{p}+\frac{n}{p}
\]
where $n$ is an any integer such that 
\[
0<\frac{\alpha}{p}+\frac{n}{p}<1
\]
which clearly means that there are only finitely many values $n$
can take. In fact, since $n$ must be an integer and $\alpha<\frac{1}{2}$
the only possibilities are 
\[
n=0,1,\cdots,\left\lfloor p-\alpha\right\rfloor =p-1
\]
where $\left\lfloor \right\rfloor $ denotes the floor function. In
our toy model $p=3$ and $\alpha=1/5$ so $n=0,1,2$ are the only
possibilities, which means that $\alpha'_{0}=\frac{1}{15}$, $\alpha'_{1}=\frac{6}{15}=\frac{3}{5}$,
$\alpha_{2}'=\frac{11}{15}$ are the only holonomy values that have
the desired properties.

Our claim is that in this case 
\begin{align*}
 & \lambda_{FO}(X_{\tau},T_{\tau},\alpha)\\
= & 2\lambda_{CLH}(Y',K',\alpha'_{0})+2\lambda_{CLH}(Y',K',\alpha'_{1})+2\lambda_{CLH}(Y',K',\alpha'_{2})\\
= & 24\lambda_{C}(Y')+\sigma_{K'}(e^{-4\pi i\alpha_{0}^{\prime}})+\sigma_{K'}(e^{-4\pi i\alpha_{1}^{\prime}})+\sigma_{K}(e^{-4\pi i\alpha_{2}'})
\end{align*}
Besides the correspondence between the different representation spaces,
the other important thing we need to check is how to compare the orientations
between the different moduli spaces, and how to identify the representations
in case perturbations are needed. We will address each of these issues
in stages, but first we state the main result. 
\begin{thm}
\label{Mapping tori calculation}Let $(Y',K')$ be a pair of an oriented
knot inside an integer homology sphere and take $\alpha'=\frac{r}{pq}$
, where $p,q,r$ are all odd integers, relatively prime and such that
$0<\alpha'<\frac{1}{2}$. Suppose moreover that $0<\frac{r}{q}<\frac{1}{2}$
and assume that the $p$-fold branched cover along $K'$ is an integer
homology sphere $Y$. Let $K$ denote the fixed point set of the $\mathbb{Z}_{p}$
action $\tau$ on $Y$, which will be another oriented knot. For $\alpha=\frac{r}{q}=p\alpha'$
consider the $p$ holonomy values 
\[
\alpha_{0}=\alpha',\alpha_{1}'=\alpha'+\frac{1}{p},\cdots,\alpha_{p-1}'=\alpha'+\frac{p-1}{p}
\]

Assume furthermore that for each $j=0,\cdots,p$, $\triangle_{K'}(e^{-4\pi i\alpha'_{j}})\neq0$
and that the reducible representation $\theta_{\alpha}$ is $\tau$
non-degenerate, i.e, $\check{H}^{1,\tau}(\check{Y};\mathfrak{g}_{\theta_{\alpha}})=0$.
Denote $(X_{\tau},T_{\tau})$ the $\tau$-mapping torus of $(Y,K)$.
Then $\lambda_{FO}(X_{\tau},T_{\tau},\alpha)$ is well defined and
moreover 
\[
\lambda_{FO}\left(X_{\tau},T_{\tau},\alpha\right)=\sum_{j=0}^{p-1}2\lambda_{CLH}\left(Y',K',\alpha'_{j}\right)=8p\lambda_{C}(Y')+\sum_{j=0}^{p-1}\sigma_{K'}(e^{-4\pi i\alpha'_{j}})
\]
\end{thm}

\begin{rem}
i) As mentioned in the acknowledgements section, the first version
of this paper had a mistake in the previous formula since we did not
count all the possible representations. We would like to thank Langte
Ma for pointing out this mistake.

ii) There are certainly infinitely many examples that satisfy our
assumptions. For example, the Brieskorn spheres $\varSigma(p,a,b)$
may be realized as the $p$-fold branched covering of $S^{3}$ along
a torus knot $T(a,b)$ \cite{MR1734525}. A way to guarantee $\theta_{\alpha}$
is $\tau$ isolated as an $\alpha$ representation of $\pi_{1}(\varSigma(p,a,b)\backslash K)$
is for it to be isolated in the ordinary sense, i.e, $\check{H}^{1}(\check{\varSigma}(p,a,b);\mathfrak{g}_{\theta_{\alpha}})=0$.
Since the Alexander polynomial of both $K'$ and $K$ have at most
a finite number of roots on the unit circle, there are infinitely
many numbers of the form $\frac{r}{pq}$ which will serve our purposes.

iii) A statement involving one of the integers $p,q$ being even would
be slightly more complicated, since the center of $SU(2)$ is $\mathbb{Z}_{2}$
and there can be a lifting issue when one tries to lift an action
of $\mathbb{Z}_{p}$ for $p$ even on a 3-manifold to an $SU(2)$
bundle \cite{MR1072911}. 

iv) Notice that for $\frac{1}{2}<\alpha_{j}'=\alpha'+\frac{j}{p}<1$
we have $\sigma_{K'}(e^{-4\pi i(\alpha_{j}'-\frac{1}{2})})=\sigma_{K'}(e^{-4\pi i\alpha_{j}'})$,
so whether or not we want to ``renormalize'' $\alpha'_{j}$ so that
it belongs to the interval $(0,1/2)$ does not affect the formula. 

\end{rem}

To keep the proof manageable we will break it into several pieces:
first we will discuss the identification at the level of the critical
sets, then we will address how the Zariski tangent spaces are related,
third we will discuss how to relate the critical sets after perturbations
have been introduced into the picture, and finally we will discuss
how to relate the orientations of the moduli spaces. Our steps should
be regarded as the orbifold version of the corresponding statements
in \cite{MR2033479}. 

First we need to review some brief facts about the orbifold fundamental
group, as well as some aspects about equivariant gauge theory. Our
main sources are \cite{MR1072911,MR1734525,MR1876288,MR2033479}. 

At this point it is a matter of preference whether one wants to think
we are working over $Y\backslash K$ or the orbifold $\check{Y}$
for analyzing the action, so we will change perspectives whenever
it is more convenient. 

Over the manifold $Y\backslash K$ we have an action $\mathbb{Z}_{p}$
and a principal $SU(2)$ bundle $P\rightarrow Y\backslash K$. Let
$\tau$ denote the generator of $\mathbb{Z}_{p}$. When one is trying
to understand the action of a group on some principal bundle, one
needs to make the assumption that 
\begin{equation}
\tau^{*}(P)\simeq P\label{condition bundles}
\end{equation}
Usually on a closed-4 manifold, this means verifying that the characteristic
numbers of the bundle are preserved ($c_{2}(P)$ in the $SU(2)$ case
for example). On any $3$-manifold the $SU(2)$ bundles are necessarily
trivial, so one may think that \ref{condition bundles} is automatically
guaranteed. However, we now need to take into account that we are
working with connections with a prescribed singularity along the knot,
so we need to check that the action of $\tau$ preserves the model
connection. Conversely, from the orbifold perspective, this is the
same as checking that the isotropy data of the bundle is preserved
\cite{MR1703606}. 

Recall that using a tubular neighborhood for the knot $K$, the model
connection \ref{model connection} in a trivialization could be understood
given by the matrix valued 1-form 
\[
i\left(\begin{array}{cc}
\alpha & 0\\
0 & -\alpha
\end{array}\right)d\theta
\]
where $(r,\theta)$ are polar coordinates. We are choosing the action
$\mathbb{Z}_{p}$ in such a way that it becomes an isometry and orientation
preserving, and from the local model of a branched cover it is not
difficult to see that $\mathbb{Z}_{p}$ acts on the coordinates as
\[
\tau\cdot(r,\theta)=(r,\theta+2\pi/p)
\]
Clearly this will preserve $d\theta$ and thus the local model of
the connection. In other words, we have verified the singular (orbifold)
analogue of the condition \ref{condition bundles}. If $\mathcal{G}(Y,K,\alpha)$
is the usual gauge group and $\mathscr{G}(Y,K,\alpha)$ the group
of bundle automorphisms of $P$ covering an element of $\mathbb{Z}_{p}$,
then we have an exact sequence 
\begin{equation}
1\rightarrow\mathcal{G}(Y,K,\alpha)\rightarrow\mathscr{G}(Y,K,\alpha)\rightarrow\mathbb{Z}_{p}\rightarrow1\label{exact}
\end{equation}
There is an action of $\mathscr{G}(Y,K,\alpha)$ by pullbacks on the
space of connections $\mathcal{C}(Y,K,\alpha)$. Let $\tilde{\tau}:P\rightarrow P$
denote a lift of $\tau$. By \ref{exact}, for any two lifts $\tilde{\tau}_{1},\tilde{\tau}_{2}$
of $\tau$, there is a gauge transformation $g\in\mathcal{G}(Y,K,\alpha)$
such that 
\[
\tilde{\tau}_{2}=\tilde{\tau}_{1}\cdot g
\]
Thus there is a well defined action $\tau^{*}$ on $\mathcal{B}^{*}(Y,K,\alpha)$.
We will denote the fixed point set of $\tau^{*}$ by $\mathcal{B}^{\tau}(Y,K,\alpha)$.
Let $[B]\in\mathcal{B}^{\tau}(Y,K,\alpha)$. Then we can find a representative
$B$ and a lift $\tilde{\tau}\in\mathscr{G}(Y,K,\alpha)$ such that
\[
\tilde{\tau}^{*}B=B
\]
If there were another such $\tilde{\tau}'$ then $\tilde{\tau}'\circ\tilde{\tau}^{-1}$
would be an element of $\mathcal{G}(Y,K,\alpha)$ fixing $B$, hence
it would belong to the stabilizer of $B$, which since we assume was
irreducible must be $\mathbb{Z}/2$. In other words 
\[
\tilde{\tau}'=\pm\tilde{\tau}
\]
which means $\tilde{\tau}$ is well defined up to a sign. Moreover,
$(\tilde{\tau})^{p}$ is an element of $\mathcal{G}(Y,K,\alpha)$
fixing $B$, so by the same token 
\[
(\tilde{\tau})^{p}=\pm1
\]
Therefore we can write 
\begin{equation}
\mathcal{B}^{\tau}(Y,K,\alpha)=\sqcup_{[\tilde{\tau}]}\mathcal{B}^{\tilde{\tau}}(Y,K,\alpha)\label{decomposition equivariant}
\end{equation}
where the disjoint union is over the lifts such that $\tilde{\tau}^{p}=\pm1$
and the equivalence relation is $\tilde{\tau}_{1}\sim\tilde{\tau}_{2}$
if and only if $\tilde{\tau}_{2}=\pm g\cdot\tilde{\tau}_{1}\cdot g^{-1}$
for some gauge transformation $g\in\mathcal{G}(Y,K,\alpha)$.

Each $\mathcal{B}^{\tilde{\tau}}(Y,K,\alpha)$ can be described as
follows: for a fixed lift $\tilde{\tau}$, let $\mathcal{C}^{\tilde{\tau}}(Y,K,\alpha)$
denote the irreducible connections $B$ such that $\tilde{\tau}^{*}B=B$.
Define $\mathcal{G}^{\tilde{\tau}}(Y,K,\alpha)=\{g\in\mathcal{G}(Y,K,\alpha)\mid g\tilde{\tau}=\pm\tilde{\tau}g\}$.
Then $\mathcal{B}^{\tilde{\tau}}(Y,K,\alpha)=\mathcal{C}^{\tilde{\tau}}(Y,K,\alpha)/\mathcal{G}^{\tilde{\tau}}(Y,K,\alpha)$.

We are after the analogue of \cite[Proposition 2.1]{MR2033479}, which
will be useful for the discussion of the spectral flow calculations.
Namely, any lift $\tilde{\tau}:P\rightarrow P$ can be written in
the base-fiber coordinates as 
\[
\tilde{\tau}(y,f)=(\tau(y),\sigma(y)f)
\]
where $\sigma:Y\backslash K\rightarrow SU(2)$. Notice that when we
regard it as an orbifold, then the automorphism $\sigma$ must be
$S^{1}$ valued along $K$, because the automorphisms of $P$, i.e
$\mathcal{G}(Y,K,\alpha)$, consists of maps $Y\rightarrow SU(2)$
which restrict to $S^{1}\subset SU(2)$ along $K$. 

Therefore, we will say that the lift $\tilde{\tau}$ is \textbf{constant
}if there exists $u\in S^{1}\subset SU(2)$ such that $\sigma(y)=u$
for all $y\in\check{Y}$. Notice for $y\in\text{Fix}(\tau)$ we must
have that 
\[
\tilde{\tau}^{p}(y,f)=(y,\sigma^{p}(y)f)=(y,\pm f)
\]
Since $\sigma\mid_{K}$ is circle valued, for fixed $y\in K$, $\sigma(y)$
must be one of the $p$-th roots of $\pm1$. Clearly this is a discrete
set, and given that we can assume that the gauge transformations are
continuous, this automatically says that $\sigma\mid_{K}$ is constant. 

Let $u_{j}=\pm\left(\begin{array}{cc}
e^{-2\pi ij/p} & 0\\
0 & e^{2\pi ij/p}
\end{array}\right)$ and consider the constant lift 
\[
\tilde{\tau}_{u}(y,f)=(\tau(y),uf)
\]
The orbifold bundles $P^{\text{ad}}/\tilde{\tau}$ and $P^{\text{ad}}/\tilde{\tau}_{u}$
have the same holonomy along $K'$, which is $2\alpha'+\frac{2j}{p}$
(modulo some twists to normalize it so that it belongs to the interval
$(0,1/2)$ again), therefore they are isomorphic. Take any isomorphism
and pull it back to an (equivariant) gauge transformation $g^{ad}\in\mathcal{G}_{SO(3)}(Y,K,\alpha)$.
Then as $SO(3)$ orbifold bundles we have that 
\[
\mathcal{B}_{ad}^{\tilde{\tau}}(Y,K,\alpha)=\mathcal{B}_{ad}^{u}(Y,K,\alpha)
\]
Since $H^{1}(Y;\mathbb{Z}/2)=0$ there is no obstruction to lifting
$g^{ad}$ to a gauge transformation $g\in\mathcal{G}(Y,K,\alpha)$,
in fact there are two choices for such a lift. Since $\mathcal{G}^{\tilde{\tau}}(Y,K,\alpha)$
incorporates the ambiguity of the lift already, we have as well that
\[
\mathcal{B}^{\tilde{\tau}}(Y,K,\alpha)=\mathcal{B}^{u}(Y,K,\alpha)
\]
 We will start now discussing the relation between equivariant $\alpha$-
representations on $(Y,K,\alpha)$ and the different $\alpha'_{j}$
-representations on $(Y',K',\alpha')$.
\begin{lem}
Assume the hypothesis of Theorem \ref{Mapping tori calculation}.
Then there is a bijective correspondence between $\bigcup_{j=0}^{p-1}\mathcal{R}(Y',K',\alpha'_{j}$)
and $\mathcal{R}^{\tau}(Y,K,\alpha)$. Here 
\[
\alpha_{j}=\alpha'+\frac{j}{p}=\frac{\alpha}{p}+\frac{j}{p}\sim\begin{cases}
\frac{\alpha+j}{p} & \text{if }\frac{\alpha+j}{p}\in(0,1/2)\\
\frac{\alpha+j}{p}-\frac{1}{2} & \text{if }\frac{\alpha+j}{p}\in(1/2.1)
\end{cases}
\]
Likewise, there is a two to one correspondence between $\mathcal{R}^{*}(X_{\tau},T_{\tau},\alpha)$
and $\mathcal{R}^{\tau}(Y,K,\alpha)$.
\end{lem}

\begin{proof}
 From covering space theory, the covering 
\[
Y\backslash K\rightarrow Y'\backslash K'
\]
induces a homotopy exact sequence 
\begin{equation}
1\rightarrow\pi_{1}(Y\backslash K)\rightarrow\pi_{1}(Y'\backslash K')\rightarrow\mathbb{Z}_{p}\rightarrow1\label{split exact sequence}
\end{equation}
Recall that the orbifold fundamental group can be defined in terms
of the knot complement as 
\begin{equation}
\check{\pi}_{1}\left(Y,K,q\right)=\pi_{1}(Y\backslash K)/\left\langle \mu_{K}^{q}\right\rangle \label{orbifold group}
\end{equation}
 Moreover, $\tau$ induces an action $\tau_{*}$ on $\pi_{1}(Y\backslash K)$
which we denote as $\tau\cdot h$ for $h\in\pi_{1}(Y\backslash K)$.
The induced action on $\mathcal{R}(Y\backslash K)$ is given by \cite[Section 2]{MR1734525}
\[
\tau^{*}(\rho)(h)=\rho(\tau\cdot h)
\]
This action descends to an action on the orbifold group and we denote
the fixed point as $\check{\pi}_{1}^{\tau}(Y,K,r/q)$. From the exact
sequence we obtain a \textit{split} exact sequence  
\begin{equation}
1\rightarrow\check{\pi}_{1}(Y,K,q)\rightarrow\check{\pi}_{1}(Y',K',pq)\rightarrow\mathbb{Z}_{p}\rightarrow1\label{short exact 2}
\end{equation}
Notice that every $\alpha'_{j}$ -representation of $(Y',K')$ (for
any $j=0,\cdots,p-1$) can be regarded as an element of $\check{\pi}_{1}(Y',K',pq)$.

What we need to check is that the pull-back of any $\alpha'_{l}$
\textit{irreducible} representation of $(Y',K')$ continues to be
an $\alpha$- irreducible representation of $(Y,K)$, and conversely,
an $\tau$-equivariant $\alpha$- irreducible representation of $(Y,K)$
pushes down to an $\alpha_{l}'$-irreducible representation of $(Y',K')$,
for some $\alpha_{l}'$. The second statement will follow from the
first one, so let's focus on the former. Since the sequence \ref{short exact 2}
splits, we can write $\check{\pi}_{1}(Y',K',pq)$ as a semi-direct
product of $\check{\pi}_{1}(Y,K,q)$ and $\mathbb{Z}_{p}$. As discussed
in \cite[Section 7.4]{MR2340988}, the representations of a semi-direct
product, like $\check{\pi}_{1}(Y',K',pq)\simeq\check{\pi}_{1}(Y,K,q)\rtimes\mathbb{Z}_{p}$
are determined in terms of the representations of $\check{\pi}_{1}(Y,K,q)$
and $\mathbb{Z}_{p}$. 

Let $\rho':\check{\pi}_{1}(Y,K,q)\rightarrow SU(2)$ be an $\alpha_{l}'$
representation and denote by $\rho$ the induced $\alpha$ representation
on $\check{\pi}_{1}(Y,K,q)$.  If $t$ is a generator of $\mathbb{Z}_{p}$
(with unit $1$) and $e$ is the identity of $\pi_{1}(Y\backslash K)$,
then $\rho(h)=\rho'(h,1)$ for $h\in\pi_{1}(Y\backslash K)$ and $u=\rho'(t)=\rho'(e,t)\in SU(2)$
determine the induced representations from the semidirect product
\ref{short exact 2}. Notice that 
\begin{align*}
 & \rho'((h,t^{m}))\\
= & \rho'((e,t^{m})\cdot(h,1))\\
= & \rho'((e,t^{m}))\rho'(h,1)\\
= & u^{m}\rho(h)
\end{align*}

The fact that $\rho$ must be $\tau$ equivariant implies that \cite[Proposition 7.7]{Lin-Ruberman-Saveliev[2018]}
\[
\tau^{*}\rho=u\rho u^{-1}
\]
Now, if we assume that $\rho$ is abelian, then $\tau^{*}\rho$ is
completely specified by how $\tau$ acts on $H_{1}(Y\backslash K;\mathbb{Z})$.
Since $\tau$ is a diffeomorphism of odd order, it is not difficult
to check that $\tau_{*}=\text{id}$ on $H_{1}(Y\backslash K;\mathbb{Z})$.
Hence $u$ and $\rho$ commute. Therefore we have 
\begin{align*}
 & \rho'((h_{1},t^{m_{1}})\cdot(h_{2},t^{m_{2}}))\\
= & u^{m_{1}}\rho(h_{1})u^{m_{2}}\rho(h_{2})\\
= & u^{m_{2}}\rho(h_{2})u^{m_{1}}\rho(h_{1})\\
= & \rho'((h_{2},t^{m_{2}})\cdot(h_{1},t^{m_{1}}))
\end{align*}
 This means that $\rho'$ vanishes on commutators hence it must be
an abelian representation, giving rise to a contradiction. 

Therefore, irreducibility is preserved under pull back of connections
and push-down of connections. Once we know this, that $\bigcup_{j=0}^{p-1}\mathcal{R}^{*}(Y',K',\alpha'$)
and $\mathcal{R}^{*,\tau}(Y,K,\alpha)$ are in bijection is immediate. 

The two to one correspondence between $\mathcal{R}^{*}(X_{\tau},T_{\tau},\alpha)$
and $\mathcal{R}^{*,\tau}(Y,K,\alpha)$ is similar and follows the
proof of \cite[Proposition 2.1]{MR2033479}. First of all, we have
a splitting exact sequence 
\[
1\rightarrow\check{\pi}_{1}(Y,K,q)\rightarrow\check{\pi}_{1}(X_{\tau},T_{\tau},q)\rightarrow\mathbb{Z}\rightarrow0
\]
The same argument we just gave applies to show that the pullback of
an irreducible $\alpha$-representation $\rho_{\check{X}}$ of $(X_{\tau},T_{\tau})$
gives rise to an irreducible $\rho_{\check{Y}}$ $\alpha$-representation
of $(Y,K)$ (or one can also think about this in terms of unique continuation
coming from restricting the corresponding $\alpha$ -flat connection
on $(X_{\tau},T_{\tau})$ to a slice $(Y,K)$). Notice, that the pull
back representation is equivariant, so we can write as before 
\[
\tau^{*}\rho_{\check{Y}}=u\rho_{\check{Y}}u^{-1}
\]
where $u=\rho_{\check{X}}(e,t)$ and $\rho_{\check{Y}}=\rho_{\check{X}}(h,1)$.
The correspondence is two to one because replacing $u$ by $-u$ gives
rise to a new representation of $(X_{\tau},T_{\tau})$ which induces
the same representation $\rho_{\check{Y}}$. This is in fact coming
from the action of $H^{1}(X_{\tau};\mathbb{Z}/2)$, and the fact that
the new representation is a new one is granted by the freeness of
this action, which is a true because of Lemma \ref{lem:FREE action}.
  That every equivariant representation on $(Y,K)$ induces one
on $(X_{\tau},T_{\tau})$ is also clear (\cite[Theorem 6.1]{MR2189939}
and \cite[Proposition 3.1]{MR2033479}).
\end{proof}

Now we will discuss the non-degeneracy condition. First we will show
that under the assumptions of Theorem \ref{Mapping tori calculation}
that $\lambda_{FO}(X_{\tau},T_{\tau},\alpha)$ is well defined. Likewise,
we will show that the Zariski tangent spaces of the $\alpha$-representations
on $(X_{\tau},T_{\tau})$ can be identified with the equivariant Zariski
tangent spaces of the $\alpha$-representations of $(Y,K,\alpha)$. 
\begin{lem}
Suppose that the conditions of Theorem \ref{Mapping tori calculation}
hold. Then $\lambda_{FO}(X_{\tau},T_{\tau},\alpha)$ is well defined.
Moreover, the unperturbed moduli space $\mathcal{M}^{*}(X,T,0,0,\alpha)$
use to compute $\lambda_{FO}(X_{\tau},T_{\tau},\alpha)$ is non-degenerate
if and only if $\mathcal{R}^{*,\tau}(Y,K,\alpha)$ is non-degenerate. 
\end{lem}

\begin{proof}
Our proof is essentially the same as the one in \cite[Proposition 3.3]{MR2033479}.
Let $\rho_{\check{X}}$ denote an $\alpha$-representation on $(X_{\tau},T_{\tau})$.
By restriction to a slice it induces an $\alpha$ representation $\rho_{\check{Y}}$
on $(Y,K)$. To study the non-degeneracy of the representation we
use Lemma \ref{orbifold cohomology group torus}. Therefore, we want
to compute $H^{1}(X\backslash T;\mathfrak{g}_{\rho_{\check{X}}})$. 

From the fibration $X\backslash T\rightarrow S^{1}$ with fiber $Y\backslash K$
we have by Leray-Serre a spectral sequence whose $E_{2}^{pq}$ page
is 
\[
H^{p}(S^{1},H^{q}(Y\backslash K;\mathfrak{g}_{\rho_{\check{Y}}}))
\]
This spectral sequence collapses for all $p\geq2$ so 
\begin{equation}
H^{1}(X\backslash T;\mathfrak{g}_{\rho_{\check{X}}})=H^{0}(S^{1},H^{1}(Y\backslash K;\mathfrak{g}_{\rho_{\check{Y}}}))\oplus H^{1}(S^{1},H^{0}(Y\backslash K;\mathfrak{g}_{\rho_{\check{Y}}}))\label{decomposition Leray Serre}
\end{equation}

Now we analyze this decomposition in two cases:

$\bullet$ Case when $\rho_{\check{X}}$ is a reducible $\alpha$-flat
connection: in this situation $\mathfrak{g}_{\rho_{\check{X}}}\simeq\mathbb{R}\oplus L_{\rho_{\check{X}}}^{\otimes2}$
and we just need to analyze the $L_{\rho_{\check{X}}}^{\otimes2}$
part of the decomposition \ref{decomposition Leray Serre}. Since
$\rho_{\check{X}}$ induces the reducible connection $\rho_{\theta_{\alpha}}$
on $Y\backslash K$ \ref{decomposition Leray Serre} becomes 
\[
H^{1}(X\backslash T;L_{\rho_{\check{X}}}^{\otimes2})=H^{0}(S^{1},H^{1}(Y\backslash K;L_{\theta_{\alpha}}^{\otimes2}))\oplus H^{1}(S^{1},H^{0}(Y\backslash K;L_{\theta_{\alpha}}^{\otimes2}))
\]

Observe that $H^{0}(Y\backslash K;L_{\theta_{\alpha}}^{\otimes2})$
vanishes since $\dim H^{0}(Y\backslash K;\mathbb{R}\oplus L_{\theta_{\alpha}}^{\otimes2})=\dim\text{stab}\theta_{\alpha}=1$
and we already know that $\dim H^{0}(Y\backslash K;\mathbb{R})=1$.
Therefore $H^{1}(S^{1},H^{0}(Y\backslash K;L_{\theta_{\alpha}}^{\otimes2}))$
vanishes. 

To compute $H^{0}(S^{1},H^{1}(Y\backslash K;L_{\theta_{\alpha}}^{\otimes2}))$
notice that the generator of $\pi_{1}(S^{1})$ acts on $H^{1}(Y\backslash K;L_{\theta_{\alpha}}^{\otimes2})$
as $\tau^{*}:H^{1}(Y\backslash K;L_{\theta_{\alpha}}^{\otimes2})\rightarrow H^{1}(Y\backslash K;L_{\theta_{\alpha}}^{\otimes2})$,
thus $H^{0}(S^{1},H^{1}(Y\backslash K;L_{\theta_{\alpha}}^{\otimes2}))$
is the fixed point set of $\tau^{*}$, which is the equivariant cohomology
$H^{1,\tau}(Y\backslash K;L_{\theta_{\alpha}}^{\otimes2})$. This
term must vanish by the assumption on Theorem \ref{Mapping tori calculation}
that $\check{H}^{1,\tau}(\check{Y};\mathfrak{g}_{\theta_{\alpha}})=0$. 

In conclusion, $H^{1}(X\backslash T;L_{\rho_{\check{X}}}^{\otimes2})$
vanishes for all $\alpha$-reducible representations $\rho_{\check{X}}$
which is the condition needed for $\lambda_{FO}(X_{\tau},T_{\tau},\alpha)$
to be well defined.

$\bullet$ Case when $\rho_{\check{X}}$ is an irreducible $\alpha$-flat
connection: since the restriction $\rho_{\check{Y}}$ is an irreducible
$\alpha$-flat connection then $H^{0}(Y\backslash K;\mathfrak{g}_{\rho_{\check{Y}}})$
will vanish which means that the second term in \ref{decomposition Leray Serre}
will vanish as well. 

As in the previous case, $H^{0}(S^{1},H^{1}(Y\backslash K;\mathfrak{g}_{\rho_{\check{Y}}}))$
can be identified with $H^{1,\tau}(Y\backslash K;\mathfrak{g}_{\rho_{\check{Y}}}))$.
Also, that the meridian $\mu_{T}$ restricts in a natural way to the
meridian of $\mu_{K}$, and any local system restricted to either
$\mu_{T}$ or $\mu_{K}$ always reduces (since the loops have abelian
fundamental group so any flat connection becomes reducible). Because
of the holonomy condition, it cannot be the trivial local system $\simeq\mathbb{R}^{3}$
in either case which means that 
\[
\mathbb{R}\simeq H^{1}(\mu_{T};\mathfrak{g}_{\rho})\simeq H^{1}(\mu_{K};\mathfrak{g}_{\check{\rho}_{Y}})\simeq H^{1,\tau}(\mu_{K};\mathfrak{g}_{\check{\rho}_{Y}})
\]
where the last isomorphism follows from the fact that $\tau$ restricts
to the identity on the knot. Therefore, $\ker:H^{1}(X\backslash T;\mathfrak{g}_{\rho})\rightarrow H^{1}(\mu_{T};\mathfrak{g}_{\rho})$
will vanish if and only if $\ker:H^{1,\tau}(Y\backslash K;\mathfrak{g}_{\rho_{\check{Y}}})\rightarrow H^{1}(\mu_{K};\mathfrak{g}_{\check{\rho}_{Y}})$
vanishes, which means that $\mathcal{M}^{*}(X,T,0,0,\alpha)$ is non-degenerate
if and only if $\mathcal{R}^{*,\tau}(Y,K,\alpha)$ is non-degenerate. 
\end{proof}
The case where perturbations are needed follow essentially the arguments
in \cite{MR2033479}. Namely, for computing $\lambda_{FO}(X_{\tau},T_{\tau},\alpha)$
as well as the different $\lambda_{CLH}(Y',K',\alpha')$ we can always
use (finitely many) holonomy perturbations whose stays away from the
singularity (i.e, the torus or knot). This was already discussed before
the proof of Theorem \ref{Signed Count}, where we used this condition
to identify the Euler characteristic of our Floer groups with the
Casson Lin Herald invariant. Since the support of the holonomies do
not meet the singularity, we can still find equivariant perturbations
$\mathfrak{p}^{\tau}$ to achieve transversality for $\mathcal{R}^{*,\tau}(Y,K,\alpha,\mathfrak{p}^{\tau})$
as in \cite[Section 5.1]{MR2033479}. Since the action is free away
from the knot $K$, these equivariant holonomy perturbations can be
pushed down to the quotient $(Y',K')$ in such a way that they guarantee
transversality for the different moduli spaces $\bigcup_{j=0}^{l-1}\mathcal{R}(Y',K',\alpha',\mathfrak{p}^{\tau})$,
as was done in \cite[Section 3.8]{MR1876288}. 

That the perturbations needed to achieve transversality for $\mathcal{R}^{*,\tau}(Y,K,\alpha)$
and $\mathcal{R}^{*}(X_{\tau},T_{\tau},\alpha)$ can be chosen in
a consistent matter corresponds precisely to \cite[Section 5.3]{MR2033479}.
Moreover, that the correspondence between the perturbed versions of
$\mathcal{R}^{*,\tau}(Y,K,\alpha)$ and $\mathcal{R}^{*}(X_{\tau},T_{\tau},\alpha)$
continues to be $2$ to $1$ is proven in exactly the same way as
\cite[Proposition 5.3]{MR2033479}.

Finally, that the orientations of the moduli spaces for $\bigcup_{j=0}^{l-1}\mathcal{R}(Y',K',\alpha',\mathfrak{p}^{\tau})$,
$\mathcal{R}^{*,\tau}(Y,K,\alpha)$ and $\mathcal{R}^{*}(X_{\tau},T_{\tau},\alpha)$
can be chosen in a consistent way is a consequence of the analysis
in \cite[Section 3.5]{MR2033479}, where we just need to use the existence
of a constant lift $\tilde{\tau}$ , which we already know exist from
our previous discussion when we defined the spaces $\mathcal{B}^{\tilde{\tau}}(Y,K,\alpha)$.

\subsection{Some tori inside circle bundles over homology $S^{1}\times S^{2}$}

\ 

Now we will work out an orbifold version of the examples discussed
in \cite[Section 8]{MR2189939}. There Ruberman and Saveliev analyzed
a family of homologies $S^{1}\times S^{3}$ for which $\lambda_{FO}(X)$
can be computed, and in fact vanishes identically. The four manifolds
$X$ they considered arise as circle bundles $\pi:X\rightarrow Y_{0}$
over a 3 manifold with the integral homology of $S^{1}\times S^{2}$
(i.e, a homology handle).

The manifolds $Y_{0}$ can be obtained from doing $0$ surgery on
a knot $K$ in an integral homology sphere $Y$, and in order to guarantee
that $\lambda_{FO}(X)$ is well defined, it is assumed that $\triangle_{K}(t)\equiv1$
and moreover that the Euler class $e\in H^{2}(Y;\mathbb{Z})=\mathbb{Z}$
of the bundle satisfies $e=1$. In order to be able to compute $\lambda_{FO}(X)$,
they furthermore assume that $\pi_{2}(Y)=0$ so that the homotopy
exact sequence of the $S^{1}$ bundle $\pi:X\rightarrow Y_{0}$ allows
them to consider $\pi_{1}X$ as a central extension of $\pi_{1}Y_{0}$
by the integers 
\[
1\rightarrow\mathbb{Z}\rightarrow\pi_{1}X\rightarrow^{\pi_{*}}\pi_{1}Y_{0}\rightarrow1
\]
From here one can identify $\mathcal{M}^{*}(X,SO(3))$ with $\mathcal{R}^{*}(Y,SO(3))$
as well as the corresponding Zariski tangent spaces. 

The natural tori inside $X$ that can be used for trying to compute
$\lambda_{FO}(X,T,\alpha)$ are related to the \textbf{Longitudinal
Floer Homology} Kronheimer and Mrowka define in \cite[Section 4.4]{MR2860345}.
First of all, observe that $Y_{0}$ contains a natural knot $K_{0}$,
which is the core of the solid torus used in the surgery. Furthermore,
$K_{0}$ represents a primitive element in the first homology of $Y_{0}$,
so in particular $H_{1}(K_{0};\mathbb{Z})$ generates $H_{1}(Y_{0};\mathbb{Z})$.

Therefore, when we look at the inverse image of $K_{0}$ under the
map $\pi:X\rightarrow Y_{0}$ it is clear that we obtain an embedded
torus $T$ satisfying the condition that $H_{1}(T;\mathbb{Z})$ surjects
onto $H_{1}(X;\mathbb{Z})$. 
\begin{example}
Consider the unknot $K=\circ\subset S_{3}$. After doing $0$ -surgery
on $K$ we obtain $Y_{0}=S^{1}\times S^{2}$. A natural $S^{1}$ bundle
over $Y_{0}$ satisfying the our requirements is $\pi:S^{1}\times S^{3}\rightarrow S^{1}\times S^{2}$
where we are using the Hopf fibration $S^{3}\rightarrow S^{2}$ on
the second factor and the trivial projection on the first factor.
Notice that in this case $K_{0}$ can be identified with $S^{1}\times\{pt\}$
and therefore the natural torus $T$ is $T=S^{1}\times S^{1}$, where
the second factor is the standard circle as well (an unknot). It is
clear in this case that for any $\alpha$ we have $\lambda_{FO}(X,T,\alpha)=0$,
since the fundamental group of $X\backslash T$ is abelian.
\end{example}

In fact, our expectation is that the previous example is the norm
in the following sense:
\begin{conjecture}
Let $\pi:X\rightarrow Y_{0}$ as before and consider the torus $T=\pi^{-1}(K_{0})$.
Then whenever $\lambda_{FO}(X,T,\alpha)$ can be defined it will vanish.
\end{conjecture}

To give insight into why we are making this conjecture, we need to
explain a bit more how $\lambda_{FO}(X)$ was shown to vanish by Ruberman
and Saveliev and how this computation would be modified in the case
of $\lambda_{FO}(X,T,\alpha)$. 

As we mentioned before, Ruberman and Saveliev identified $\mathcal{M}^{*}(X,SO(3))$
with $\mathcal{R}^{*}(Y,SO(3))$ in an orientation preserving way.
The action of $H^{1}(X;\mathbb{Z}_{2})$ continues to be free in this
situation, so in particular $\lambda_{FO}(X)$ can be computed as
one half the signed count of elements in $\mathcal{R}^{*}(Y,SO(3))$
(after perturbations if needed). This count ends up being the same
as $\triangle_{K}''(1)$, which is zero in this case. 

The interesting feature of this calculation is that $\triangle_{K}''(1)$
is the Euler characteristic of the Instanton Floer homology Floer
defined for a homology $S^{1}\times S^{2}$. Therefore, we expect
that $\lambda_{FO}(X,T,\alpha)$ is related to the Euler characteristic
of an orbifold version of Instanton Floer homology on homology handles.
But this is precisely the Longitudinal Floer homology $HIL(Y_{0},K_{0})$
Kronheimer and Mrowka defined in \cite{MR2860345}! Here the connections
are allowed to have a singularity along the knot $K_{0}$, where they
used holonomy $\alpha=1/4$ in order to avoid the compactness issues
due to non-monotonicity. 

If one is willing to use local coefficients, there is no difficulty
in obtaining a version $HIL(Y_{0},K_{0},\alpha)$ which is always
well defined, since there are no reducible $\alpha$-flat connections
to worry about in this situation (because we are now using the non-trivial
$SO(3)$ bundle over $Y_{0}$). They conjectured that $\chi(HIL(Y_{0},K_{0}))$
should be $2\triangle_{K}''(1)$ in general, which in our situation
would imply that $\chi(HIL(Y_{0},K_{0}))$ vanishes. 

For the other values of $\alpha$ there is a brief discussion in \cite[Section 4.3]{MR2189924},
where it is promised that a definition of what we are defining as
$HIL(Y_{0},K_{0},\alpha)$ would be constructed eventually. In any
case, based on the Property 3 Collin states in the survey we conjecture
the following:
\begin{conjecture}
\label{conjecture Euler characteristic}For $\alpha\in\mathbb{Q}\cap(0,1/2)$
the Euler characteristic of $HIL(Y_{0},K_{0},\alpha)$ over the Novikov
ring $\varLambda$ equals $2\triangle_{K}^{\prime\prime}(1)$. In
particular, for the tori we are considering arising as $\pi^{-1}(K_{0})$
we would have that $\lambda_{FO}(X,T,\alpha)$ vanishes in all of
these cases. 
\end{conjecture}

\begin{rem}
 Besides the fact that conjecture \ref{conjecture Euler characteristic}
might be used for computing $\lambda_{FO}(X,T,\alpha)$ for this family
of examples, it could be of independent interest to compute the Euler
characteristic of $HIL(Y_{0},K_{0},\alpha)$. In fact, we plan to
verify this in the future, adapting the ideas of Floer as described
in \cite{MR1362829}.
\end{rem}

\textbf{}

\textbf{}

\textbf{}

\textbf{}

\uline{}

\textbf{}

\bibliographystyle{plain}
\bibliography{amsj,/Users/mariano/Dropbox/virginia/works/ARXIV/references}

\begin{quote}
Department of Mathematics, Rutgers University.

\textit{\small{}E-mail address}{\small{}:} \textsf{\footnotesize{}mariano.echeverria@rutgers.edu}{\footnotesize\par}
\end{quote}

\end{document}